\documentclass[10pt, a4]{amsart}

\usepackage{lmodern}
\usepackage[utf8]{inputenc}
\usepackage{amsmath}
\usepackage{graphicx}
\usepackage{amssymb}
\usepackage{esint}
\usepackage[dvipsnames]{xcolor}
\usepackage{tikz}
\usepackage{xxcolor}
\usepackage{floatrow}
\usepackage{color}
\usepackage{amsthm}
\usepackage{epsfig}
\usepackage[english]{babel}
\usepackage{hyperref}
\usepackage[normalem]{ulem}
\usepackage{mathrsfs}
\usepackage{tikz}
\usetikzlibrary{decorations.markings,backgrounds}
\usetikzlibrary{arrows.meta}
\usepackage{pgfplots}
\usepackage{subfigure}
\usepackage{caption}
\usepackage{bbm}

\usepackage{cite}

\usepackage{stmaryrd}

\usepackage{float}
\usepackage{pst-plot}

\hypersetup{
	colorlinks,
	linkcolor={red!80!black},
	citecolor={blue!50!black},
	urlcolor={blue!80!black}
}
 \usepackage{diagbox}

\usepackage[a4paper,top=3.5 cm,bottom=3 cm,left=2.5 cm,right=2.5 cm]{geometry}

\theoremstyle{plain}
\newtheorem{lemma}{Lemma}[section]
\newtheorem{proposition}[lemma]{Proposition}
\newtheorem{theorem}{Theorem}[section]

\theoremstyle{definition}

\theoremstyle{remark}
\newtheorem{remark}[lemma]{Remark}

\numberwithin{equation}{section}

\def\to{\rightarrow}

\usepackage{dsfont}

\renewcommand{\div}{\mathrm{div} \hspace{0.5mm}}        

\newcommand{\supp}{\mathrm{supp}}

\newcommand{\nchi}{{\raise.3ex\hbox{$\chi$}}}

\def\XXint#1#2#3{{\setbox0=\hbox{$#1{#2#3}{\int}$ }
		\vcen{\hbox{$#2#3$ }}\kern-.6\wd0}}

\newcommand{\cR}{\mathcal{R}}

\newcommand{\R}{\mathbb{R}}

\usepackage{mathtools}

\makeatletter
\newcommand{\justified}{%
	\rightskip\z@skip%
	\leftskip\z@skip}
\makeatother
\usepackage{amsfonts}
\usepackage{aurical}

\newcommand\restr[2]{{
		\left.\kern-\nulldelimiterspace 
		#1 
		\vphantom{\big|} 
		\right|_{#2} 
}}

\newcommand{\var}{\varepsilon}

\DeclareFontFamily{U}{mathx}{\hyphenchar\font45}
\DeclareFontShape{U}{mathx}{m}{n}{<-> mathx10}{}
\DeclareSymbolFont{mathx}{U}{mathx}{m}{n}
\DeclareMathAccent{\widebar}{0}{mathx}{"73}

\renewcommand{\i}{\ifmmode\mathit{\mathchar"7010 }\else\char"10 \fi}
\renewcommand{\j}{\ifmmode\mathit{\mathchar"7011 }\else\char"11 \fi}

\renewcommand{\ge}{\geq}

\def\char{{1\!\mbox{\rm l}}}

\usepackage{dsfont}

\definecolor{orange}{rgb}{1,.549,0}
\definecolor{GreenYellow }{rgb}{ 0.15,   0.69, 0}
\definecolor{Yellowone}{rgb}{ 0, 1., 0} \definecolor{Goldenrod }{rgb}{  0, 0.10, 0.84}
\definecolor{Dandelion }{rgb}{ 0, 0.29, 0.84} 
\definecolor{Apricot }{rgb}{ 0, 0.32, 0.52}
\definecolor{Peach }{rgb}{ 0, 0.50, 0.70} 
\definecolor{GreenYellow}{cmyk}{0.15,0,0.69,0}
\definecolor{RoyalPurple}{cmyk}{0.75,0.90,0,0}
\definecolor{Yellow}{cmyk}{0,0,1,0}
\definecolor{BlueViolet}{cmyk}{0.86,0.91,0,0.04}
\definecolor{Goldenrod}{cmyk}{0,0.10,0.84,0}
\definecolor{Periwinkle}{cmyk}{0.57,0.55,0,0}
\definecolor{Dandelion}{cmyk}{0,0.29,0.84,0}
\definecolor{CadetBlue}{cmyk}{0.62,0.57,0.23,0}
\definecolor{Apricot}{cmyk}{0,0.32,0.52,0}
\definecolor{CornflowerBlue}{cmyk}{0.65,0.13,0,0}
\definecolor{Peach}{cmyk}{0,0.50,0.70,0}
\definecolor{MidnightBlue}{cmyk}{0.98,0.13,0,0.43}
\definecolor{Melon}{cmyk}{0,0.46,0.5,0}
\definecolor{NavyBlue}{cmyk}{0.94,0.54,0,0}
\definecolor{YellowOrange}{cmyk}{0,0.42,1,0}
\definecolor{RoyalBlue}{cmyk}{1,0.50,0,0}
\definecolor{Orange}{cmyk}{0,0.61,0.87,0}
\definecolor{Blue}{cmyk}{1,1,0,0}
\definecolor{BurntOrange}{cmyk}{0,0.51,1,0}
\definecolor{Cerulean}{cmyk}{0.94,0.11,0,0}
\definecolor{Bittersweet}{cmyk}{0,0.75,1,0.24}
\definecolor{Cyan}{cmyk}{1,0,0,0}
\definecolor{RedOrange}{cmyk}{0,0.77,0.87,0}
\definecolor{ProcessBlue}{cmyk}{0.96,0,0,0}
\definecolor{Mahogany}{cmyk}{0,0.85,0.87,0.35}
\definecolor{SkyBlue}{cmyk}{0.62,0,0.12,0}
\definecolor{Maroon}{cmyk}{0,0.87,0.68,0.32}
\definecolor{Turquoise}{cmyk}{0.85,0,0.20,0}
\definecolor{BrickRed}{cmyk}{0,0.89,0.94,0.28}
\definecolor{TealBlue}{cmyk}{0.86,0,0.34,0.02}
\definecolor{Red}{cmyk}{0,1,1,0}
\definecolor{Aquamarine}{cmyk}{0.82,0,0.30,0}
\definecolor{OrangeRed}{cmyk}{0,1,0.50,0}
\definecolor{BlueGreen}{cmyk}{0.85,0,0.33,0}
\definecolor{RubineRed}{cmyk}{0,1,0.13,0}
\definecolor{Emerald}{cmyk}{1,0,0.50,0}
\definecolor{WildStrawberry}{cmyk}{0,0.96,0.39,0}
\definecolor{JungleGreen}{cmyk}{0.99,0,0.52,0}
\definecolor{Salmon}{cmyk}{0,0.53,0.38,0}
\definecolor{SeaGreen}{cmyk}{0.69,0,0.50,0}
\definecolor{CarnationPink}{cmyk}{0,0.63,0,0}
\definecolor{Green}{cmyk}{1,0,1,0}
\definecolor{Magenta}{cmyk}{0,1,0,0}
\definecolor{ForestGreen}{cmyk}{0.91,0,0.88,0.12}
\definecolor{VioletRed}{cmyk}{0,0.81,0,0}
\definecolor{PineGreen}{cmyk}{0.92,0,0.59,0.25}
\definecolor{Rhodamine}{cmyk}{0,0.82,0,0}
\definecolor{LimeGreen}{cmyk}{0.50,0,1,0}
\definecolor{Mulberry}{cmyk}{0.34,0.90,0,0.02}
\definecolor{YellowGreen}{cmyk}{0.44,0,0.74,0}
\definecolor{RedViolet}{cmyk}{0.07,0.90,0,0.34}
\definecolor{SpringGreen}{cmyk}{0.26,0,0.76,0}
\definecolor{Fuchsia}{cmyk}{0.47,0.91,0,0.08}
\definecolor{OliveGreen}{cmyk}{0.64,0,0.95,0.40}
\definecolor{Lavender}{cmyk}{0,0.48,0,0}
\definecolor{RawSienna}{cmyk}{0,0.72,1,0.45}
\definecolor{Thistle}{cmyk}{0.12,0.59,0,0}
\definecolor{Sepia}{cmyk}{0,0.83,1,0.70}
\definecolor{Orchid}{cmyk}{0.32,0.64,0,0}
\definecolor{Brown}{cmyk}{0,0.81,1,0.60}
\definecolor{DarkOrchid}{cmyk}{0.40,0.80,0.20,0}
\definecolor{Tan}{cmyk}{0.14,0.42,0.56,0}
\definecolor{Purple}{cmyk}{0.45,0.86,0,0}
\definecolor{Gray}{cmyk}{0,0,0,0.50}
\definecolor{Plum}{cmyk}{0.50,1,0,0}
\definecolor{Black}{cmyk}{0,0,0,1}
\definecolor{Violet}{cmyk}{0.79,0.88,0,0}
\definecolor{White}{cmyk}{0,0,0,0}
\definecolor{rltred}{rgb}{0.75,0,0}
\definecolor{rltgreen}{rgb}{0,0.5,0}
\definecolor{oneblue}{rgb}{0,0,0.75}
\definecolor{marron}{rgb}{0.64,0.16,0.16}
\definecolor{forestgreen}{rgb}{0.13,0.54,0.13}
\definecolor{purple}{rgb}{0.62,0.12,0.94}
\definecolor{dockerblue}{rgb}{0.11,0.56,0.98}
\definecolor{freeblue}{rgb}{0.25,0.41,0.88}
\definecolor{myblue}{rgb}{0,0.2,0.4}
\definecolor{Melon}{rgb}{ 0.46, 0.50, 0}
\definecolor{Melone}{rgb}{ 0, 0.46, 0.50}

\allowdisplaybreaks
\usepackage[colorinlistoftodos]{todonotes}
\usepackage{url}

\usepackage{graphicx,tikz} 
\begin{document}

\title[The compressible Euler-Maxwell system]{A new
characterization of the dissipation structure and the relaxation limit for the compressible Euler-Maxwell system}

\author[T. Crin-Barat]{Timothée Crin-Barat}
\address[T. Crin-Barat]{Chair for Dynamics, Control, Machine Learning and Numerics, Alexander von Humboldt Professorship, Department
of Mathematics, Friedrich-Alexander Universität Erlangen-Nürnberg, 91058 Erlangen, Germany}
\email{timothee.crin-barat@fau.de}

\author[Y.-J. Peng]{Yue-Jun Peng}
\address[Y.-J. Peng]{Laboratoire de Mathématiques Blaise Pascal,
Université Clermont Auvergne / CNRS
63178 Aubière Cedex, France}
\email{yue-jun.peng@uca.fr}

\author[L.-Y. Shou]{Ling-Yun Shou}
\address[L.-Y. Shou]{School of Mathematics and Key Laboratory of Mathematical MIIT, Nanjing University of Aeronautics and Astronautics, Nanjing, 211106,
P. R. China}
\email{shoulingyun11@gmail.com}

\author[J. Xu]{Jiang Xu}
\address[J. Xu]{School of Mathematics and Key Laboratory of Mathematical MIIT, Nanjing University of Aeronautics and Astronautics, Nanjing, 211106,
P. R. China}
\email{jiangxu\underline{~}79math@yahoo.com}



\linespread{1.1}

\maketitle

\begin{abstract}

We investigate the three-dimensional compressible Euler-Maxwell system, a model for simulating the transport of electrons interacting with propagating electromagnetic waves in semiconductor devices. First, we show the global well-posedness of classical solutions being a \textit{sharp} small perturbation of constant equilibrium in a critical regularity setting, uniformly with respect to the relaxation parameter $\varepsilon>0$. Then, for all times $t>0$, we derive quantitative error estimates at the rate $O(\varepsilon)$ between the  rescaled Euler-Maxwell system and the limit drift-diffusion model. To the best of our knowledge, this work provides the first global-in-time strong convergence for the relaxation procedure in the case of ill-prepared data.

In order to prove our results, we develop a new characterization of the dissipation structure for the linearized Euler-Maxwell system with respect to the relaxation parameter $\varepsilon$. This is done by partitioning the frequency space into three distinct regimes: low, medium and high frequencies, each associated with a different behaviour of the solution. Then, in each regime, the use of efficient unknowns and Lyapunov
functionals based on the hypocoercivity theory leads to uniform a priori estimates. 
\end{abstract}

\section{Introduction}
The Euler-Maxwell system for plasma physics is widely used to simulate phenomena such as photoconductive switches, electro-optics, semiconductor lasers, high-speed computers, etc. In these applications, the transported electrons interact with electromagnetic waves and the model takes the form of Euler equations for the conservation laws of mass density, current density and energy density for electrons, coupled to Maxwell's equations for self-consistent electromagnetic fields (see \cite{chenEM,CJW, RG1} for more explanations). In this paper, we shed new light on such interactions between classical fluid mechanics laws and electrical and magnetic forces to establish long-time existence and relaxation limit results. To achieve this, we propose a new approach based on the natural hypocoercive properties of the system arising from these interactions.

We consider the isentropic Euler-Maxwell system in $\R^3$ which, for $(t,x)\in [0,+\infty)\times \R^3$, reads
\begin{equation}\label{EM}
\left\{
\begin{aligned}
&\partial_{t}\rho+\div (\rho u)=0,\\ 
&\partial_{t}(\rho u)+\div (\rho u\otimes u)+\nabla P(\rho)=-\rho(E+u\times B)-\smash{\frac{1}{\var}}\rho u,\\
&\partial_{t}E-\nabla\times B=\rho u,\\
&\partial_{t}B+\nabla\times E=0,
\end{aligned}
\right.
\end{equation}
with the constraints 
\begin{equation}\label{EM2}
\div E=\bar{\rho}-\rho \quad \text{and} \quad 
\div B=0.
\end{equation}
Here $\rho=\rho(t,x)>0$ and $u=u(t,x)\in\R^3$
are, respectively, the density and the velocity
of electrons, $E=E(t,x)\in\R^3$ denotes the electric field, and $B=B(t,x)\in \R^3$ is the magnetic field. In the momentum equation in  $\eqref{EM}_{2}$, the term $\rho(E+u\times B)$ stands for the Lorentz force, $\rho u$ is a damping term associated with friction forces and $\var>0$ is a relaxation parameter. The pressure $P(\rho)$ is assumed to be a smooth function of the density fulfilling $P'(\bar{\rho})>0$ for $\bar{\rho}>0$ a constant density of charged background ions. We are concerned with \eqref{EM}-\eqref{EM2} for the initial data
\begin{align}
(\rho,u,E,B)(0,x)=(\rho_{0},u_{0},E_{0},B_{0})(x),\quad x\in\R^3,\label{EMd}
\end{align}
and focus on solutions that are close to some constant state $(\bar{\rho}, 0, 0, \bar{B})$ at infinity, where $\bar{B}\in\R^3$ is a constant vector. Note that the constraint condition \eqref{EM2} remains true for every $t>0$ if it holds at time $t=0$:  
\begin{align}
&\div E_{0}=\bar{\rho}-\rho_{0},\quad \div B_{0}=0.\label{cc}
\end{align}


One of the main interests of the present paper is to justify the relaxation limit of solutions to \eqref{EM} as  $\var\rightarrow 0$ in a diffusive scaling. To this end, we perform the 
$\mathcal{O}(1/\var)$ change of 
time scale: 
\begin{equation}
\begin{aligned}
&(\rho^{\var},u^{\var},E^{\var} ,B^{\var} )(t,x):=(\rho,\frac{1}{\var}u,E,B)(\frac{t}{\var},x).\label{scaling1}
\end{aligned}
\end{equation}
The new variables satisfy
\begin{equation}\label{EMvar}
\left\{
\begin{aligned}
&\partial_{t}\rho^{\var}+\div (\rho^{\var} u^{\var})=0,\\
&\var^2\partial_{t}(\rho^{\var} u^{\var})+\var^2\div (\rho^{\var} u^{\var} \otimes u^{\var} )+\nabla P(\rho^{\var})=-\rho^{\var} (E^{\var} +\var u^{\var}\times B^{\var} )-\rho^{\var} u^{\var},\\
&\var\partial_{t}E^{\var} -\nabla\times B^{\var} =\var\rho^{\var} u^{\var},\\
&\var\partial_{t}B^{\var} +\nabla\times E^{\var} =0,\\
&\div E^{\var} =\bar{\rho}-\rho^{\var},\\
&\div B^{\var} =0,
\end{aligned}
\right.
\end{equation}
with the initial data
\begin{align} (\rho^{\var},u^{\var},E^{\var},B^{\var})(0,x)=(\rho_{0},\frac{1}{\var}u_{0},E_{0},B_{0})(x),\quad x\in\R^3.\label{EMdvar}
\end{align}
Formally, as $\var\rightarrow0$, $(\rho^{\var},u^{\var},E^{\var} ,B^{\var} )$ converges to $(\rho^{*},u^{*},E^{*},B^{*})$ solving
\begin{equation}\label{EMlimit}
\left\{
\begin{aligned}
&\partial_{t}\rho^{*}+\div (\rho^{*} u^{*})=0,\\
&\rho^{*} u^{*}=-\nabla P(\rho^{*})-\rho^{*} E^{*},\\
&\nabla\times B^{*}=0,\\
&\nabla\times E^{*}=0,\\
&\div E^{*}=\bar{\rho}-\rho^{*},\\
&\div B^{*}=0.
\end{aligned}
\right.
\end{equation}
Clearly, since
\[  \nabla\times B^{*}=0\quad \mbox{and}\quad\div B^{*}=0,  \]
we may take $B^{*}=\bar{B}$. Moreover, due to $\nabla\times E^{*}=0$, there exists a potential function $\phi^{*}$ such that $E^{*}=\nabla \phi^{*}=\nabla(-\Delta)^{-1}(\rho^{*}-\bar{\rho})$. Thus, \eqref{EMlimit} reformulates as the drift-diffusion model for semiconductors:
\begin{equation}\label{DD}
\left\{
\begin{aligned}
&\partial_{t}\rho^{*}-\Delta P(\rho^{*})-\div (\rho^{*} \nabla \phi^{*})=0,\\
&\Delta \phi^{*}=\bar{\rho}-\rho^{*}.
\end{aligned}
\right.
\end{equation}
The velocity field $u^{*}$ satisfies the Darcy's law:
\begin{equation}\label{darcy}
\begin{aligned}
&u^{*}=-\nabla (h(\rho^{*})+\phi^{*}),
\end{aligned}
\end{equation}
where the enthalpy $h(\rho)$ is defined by
\begin{align}
&h(\rho):=\int_{\bar{\rho}}^{\rho}\frac{P'(s)}{s}ds.\label{enth1}
\end{align}




\subsection{Existing literature}
So far there are several results concerning the global existence, large-time behaviour and asymptotic limit for the isentropic Euler-Maxwell system \eqref{EM}. In one dimension, using a Godunov scheme with fractional steps and the compensated compactness theory,  Chen, Jerome and Wang \cite{CJW} constructed global
weak solutions to the initial boundary value problem for arbitrarily large initial data. In the multidimensional case, the question of global weak solutions is quite open and mainly smooth solutions have been studied. Jerome \cite{jerome1}
established the local well-posedness of smooth solutions to the Cauchy problem \eqref{EM}-\eqref{EMd} in the framework of Sobolev spaces $H^{s}(\mathbb{R}^{d})$ with $s>\frac{5}{2}$ according to the standard theory for symmetrizable
hyperbolic systems. The existence of global smooth solutions near constant equilibrium states has been obtained independently
by Peng, Wang \& Gu\cite{PYJEMSIAM}, Duan\cite{Duan2011} and Xu \cite{XuEM}. Xu 
employed the theory of Besov spaces and established the global existence of classical solutions in $B^{s_{*}}$ with the critical regularity index $s_{*}=\frac{5}{2}$ and analyzed the singular limits, such as the relaxation limit and the non-relativistic limit. Ueda, Wang and Kawashima \cite{UedaWangKawa2012} pointed out that the system \eqref{EM} was of regularity-loss type and time-decay estimates were derived in \cite{Duan2011,UedaKawa2011}. Concerning the relaxation from \eqref{EMvar} to \eqref{DD}, Hajjej and Peng \cite{hajjejpeng12} carried out an asymptotic expansion and obtained convergence rates for the relaxation procedure in the case of local-in-time solutions for both well-prepared data and ill-prepared data. Recently, Li, Peng and Zhao  \cite{liEM} studied the relaxation limit for global smooth solutions in periodic domains and obtained error estimates of smooth periodic solutions between \eqref{EMvar} and \eqref{DD} by stream function techniques and Poincar\'e inequality.  Concerning the stability of steady-states, we refer to those works \cite{pengzhu14,peng2015,liuguopeng19}. Let us also mention \cite{DLZ,Peng12,XuEMtwo} pertaining to the global well-posedness of two-fluid Euler-Maxwell equations near constant states. 

In order to investigate the large-time behaviour of solutions to the system \eqref{EM}, as observed by Duan\cite{Duan2011}, Ueda, Wang and Kawashima \cite{UedaKawa2011,UedaWangKawa2012}, one must rely on a \textit{non-symmetric dissipation} mechanism due to the coupled electric and  magnetic fields, which leads to the regularity loss phenomenon. More precisely, let $U_{L}$ be the solution to the linearized system of \eqref{EM} around $(\bar{\rho},0,0,\bar{B})$ with $\var=1$. As shown in \cite{UedaKawa2011}, the Fourier transform $\widehat{U_{L}}$ satisfies the following pointwise estimate: 
\begin{equation}
\begin{aligned}
|\widehat{U_{L}}(t,\xi)|^2\lesssim e^{ -\frac{c |\xi|^2}{(1+|\xi|^2)^2}t}|\widehat{U_{L}}(0,\xi)|^2,\label{pointwisevar}
\end{aligned}
\end{equation}
for all $t>0$, $\xi\in \R^3$ and  some constant $c>0$. The solution $U_{L}$ decays like the heat kernel at low frequencies and, for the high-frequency part, it decays at the price of additional regularity assumption on the initial data. 
Later, Ueda, Duan and Kawashima \cite{UDK1} formulated a new structural condition to analyze the weak dissipative mechanism for general hyperbolic systems with non-symmetric relaxation (including the Euler-Maxwell system \eqref{EM}). 
Xu, Mori and Kawashima \cite{xumorikaw1} developed a general time-decay inequality of $L^p$-$L^q$-$L^r$ type, which allows to get the minimal regularity for the decay estimate of $L^1$-$L^2$ type. Recently,  Mori \cite{Mori} presented a kind of S-K mixed criterion that is applicable also to weakly dissipative
models including the Timoshenko–Cattaneo system.

In the absence of damping term in \eqref{EM}, using the ``space-time resonance method", Germain-Masmoudi \cite{GMEM1} proved the global existence and scattering at the rate $t^{-1/2}$. Subsequently, nontrivial global solutions being small irrotational perturbations of constant solutions of the full two-fluid system were constructed by Guo-Ionescu-Pausader \cite{GuoIon}. In the $2D$ case, there is one critical new difficulty, namely the slow decay of solutions. Deng-Ionescu-Pausader \cite{DAI1} proved the global stability of a constant neutral background by using a combination of improved energy estimates in the Fourier space and an $L^2$ bound on the oscillatory integral operator. The global regularity results described above are restricted to the case of solution with trivial vorticity. Ionescu and Lie \cite{IonLie} initiated the study of long-term regularity of solutions with nontrivial vorticity and proved that sufficiently small solutions extended smoothly on a time of existence that depends only on the size of the vorticity.

\vspace{2mm}

In the manuscript, we are interested in the dissipative mechanism arising from the non-symmetric relaxation and their interactions with respect to the relaxation parameter $\var$
for the Euler-Maxwell system \eqref{EM}.  Before stating the paper's findings, we recall recent efforts devoted to studying partially dissipative hyperbolic systems with symmetric relaxation of the type:
\begin{equation}
\frac{\partial V}{\partial t} + \sum_{j=1}^dA^j(V)\frac{\partial V}{\partial x_j}=\frac{H(V)}{\var}, \label{GEQSYM}
\end{equation}
where the unknown $V=V(t,x)$ is a $N$-vector valued function depending on the time variable $t\in\mathbb{R}_{+}$ and on the space variable $x\in\mathbb{R}^d(d\geq1).$ The $A^j(V)$ ($j=1,..,d$) and $H$ are given smooth functions on $\mathcal{O}_{V}\in\mathbb{R}^N$ (the state space).  

Note that in the absence of source term $H(V)$, \eqref{GEQSYM} reduces to a system of conservation laws. In that case, it is well-known that classical solutions may develop singularities (\textit{e.g.}, shock waves) in finite time, even if initial data are sufficiently smooth and small (see Dafermos \cite{Dafermos1} and Lax \cite{Lax1}). The system \eqref{GEQSYM} with relaxation effect is of interest in numerous physical situations, including gas flow near thermo-equilibrium, kinetic theory with small mean free path and viscoelasticity with vanishing memory (cf. \cite{Cere,Vicenti,Whitham}). It also arises in the numerical simulation of conservation laws (see \cite{JinXin}). A typical example is the following isentropic compressible Euler equations with damping: \begin{equation}\label{euler}
\left\{
\begin{aligned}
& \partial_{t}\rho+\div (\rho u)=0,\\
& \partial_{t}(\rho u)+\div (\rho u\otimes u)+\nabla P(\rho)+\frac{1}{\var} \rho u=0.
\end{aligned}
\right.
\end{equation}
In the case $\var=1$, a natural question arises: what conditions
can be imposed on $H(V)$ so it prevents the finite-time blowup of classical solutions? Chen, Levermore and Liu \cite{chen1994} first formulated a notion of the entropy for \eqref{GEQSYM}, which was a natural extension of the classical one due to Godunov \cite{Godunov2}, Friedrichs and Lax \cite{FL} for conservation laws. However, their dissipative entropy condition is not sufficient to develop a global existence theory for \eqref{GEQSYM}. Later, imposing a technical requirement on the entropy, Yong \cite{Yong} proved the global existence of classical solutions in a neighbourhood of constant equilibrium $\bar{V}\in \mathbb{R}^{N}$ satisfying $H(\bar{V})=0$ under the Shizuta–Kawashima condition \cite{SK}. We also mention that Hanouzet and Natalini \cite{HanouzetNatalini} obtained a similar global existence result for the one-dimensional problem before the work \cite{Yong}. Subsequently, Kawashima and Yong \cite{KY} removed the technical requirement on the dissipative entropy used in \cite{Yong,HanouzetNatalini} and gave a perfect definition
of the entropy notion, which leads to the global existence in regular Sobolev spaces. Then, Bianchini, Hanouzet and Natalini \cite{BHN} showed that smooth solutions approach the constant equilibrium state $\bar{V}$ in the $L^{p}$-norm at the rate $O(t^{-\frac{d}{2}(1-\frac{1}{p})})$, as $t\rightarrow\infty$, for
$p\in[\min\{d,2\},\infty]$, by using the Duhamel principle and a
detailed analysis of the Green kernel estimates for the linearized
problem. 

Recently, Beauchard and Zuazua \cite{BZ} framed the global-in-time existence theory in the spirit of Villani's hypocoercivity \cite{Villani} and established the equivalence of the Shizuta-Kawashima condition and the Kalman rank condition from control theory. Then, Kawashima and the fourth author in \cite{XK1,XK2,XK1D} extended the prior works to the larger setting of critical non-homogeneous Besov spaces $B^{s_{*}}$. Note that the mathematical theory of Kato \cite{Katolocal} and Majda \cite{Majdalocal} for quasilinear hyperbolic systems is invalid in $H^{s_*}$. In recent works, the first author and Danchin \cite{CBD2,CBD1,CBD3} employed 
hybrid Besov norms with different regularity exponents in low and high frequencies, which allows to pinpoint optimal smallness conditions for the global well-posedness of the Cauchy problem of \eqref{GEQSYM} and to get more 
accurate information on the qualitative and quantitative properties of the constructed solutions.
Regarding the relaxation limit as $\varepsilon\to0$ in systems of the type \eqref{GEQSYM}, the first justification is due to Marcati, Milani and Secchi \cite{MMS} in a one-dimensional setting. The limiting procedure was carried out by using the theory of compensated compactness. Then, Liu \cite{Liu} proved, using the approach based on the theory of nonlinear waves, the relaxation to parabolic equations for genuinely nonlinear hyperbolic systems. Marcati and Milani \cite{mar0} considered the time rescaling 
\eqref{scaling1} for the one-dimensional compressible Euler flow
\eqref{euler} and derived Darcy's law in the limit $\var\rightarrow0$, which is analogous to the one derived in \cite{MMS}. Later,  Marcati and Rubino \cite{marcatiparabolicrelaxation}
developed a complete hyperbolic to parabolic relaxation theory for $2\times 2$ genuinely nonlinear hyperbolic balance laws. Junca and Rascle \cite{Junca} established the relaxation convergence from the isothermal equation \eqref{euler} to the heat equation for arbitrarily large initial data in $BV(\mathbb{R})$ that are bounded away from vacuum.

As for \eqref{GEQSYM} in several dimensions,
Coulombel, Goudon and Lin \cite{CoulombelGoudon,CoulombelLin} employed the classical energy approach and constructed uniform-in-$\varepsilon$ smooth solutions to the isothermal Euler equations \eqref{euler} and then they justiﬁed the weak relaxation limit in the Sobolev spaces $H^{s}(\mathbb{R}^{d})\:(s>1+d/2,\ s\in \mathbb{Z})$. The fourth author and Wang \cite{XuWang} improved their works to the setting of critical Besov space $B^{\frac{d}{2}+1}$. More precisely, it is shown that the density converges towards the solution of the porous medium equation, as $\var\rightarrow0$. Peng and Wasiolek \cite{Peng16AIHP} proposed structural stability conditions and constructed an approximate solution using a formal asymptotic expansion with initial layer corrections. It allowed to establish the  uniform local existence with respect to $\var$ and the convergence of \eqref{GEQSYM} to parabolic-type equations as $\var\rightarrow0$. Subsequently, under the Shizuta–Kawashima stability condition, they \cite{PengWasiolek} established the uniform global existence and the global-in-time convergence from \eqref{GEQSYM} to second-order nonlinear parabolic systems by using Aubin-Lions compactness arguments. In the spirit of the stream function approach of \cite{Junca}, Li, Peng and Zhao \cite{LiPengZhao1d} obtained explicit convergence rates for this relaxation process for $d=1$. Recently, the first author and Danchin \cite{CBD3,danchinnoterelaxation} observed that the partially dissipative hyperbolic system \eqref{GEQSYM} can be decomposed into a parabolic part and a damped part in the frequency region $|\xi|\lesssim \var^{-1}$ and justified the strong relaxation limit of diffusively rescaled solutions of \eqref{GEQSYM} globally in time in homogeneous critical Besov spaces with the explicit convergence rate.

\vspace{1mm}
However, the parabolic relaxation theory developed in \cite{CBD3,danchinnoterelaxation} is only applicable to  \eqref{GEQSYM}
with {\emph{symmetric}} relaxation matrices, where the Shizuta-Kawashima condition is well satisfied. In the present manuscript, we analyze the compressible Euler-Maxwell system \eqref{EMvar} and develop the corresponding theory for hyperbolic systems with non-symmetric relaxation.

\subsection{A first look at our strategy}\label{subsectionmotivation}
First, we characterize 
the dissipation structures of the system \eqref{EMvar} 
with respect to $\var$. We denote by $U_{L,\varepsilon}=(\rho-\bar{\rho},\varepsilon u,E,B-\bar{B})$ the solution to the linearization \eqref{EMvarl} of \eqref{EMvar}. In Proposition \ref{proppointwise}, it is shown that 
\begin{equation}\label{pointwise0}
\begin{aligned}
|U_{L,\varepsilon}(t,\xi)|^2\lesssim e^{\lambda_{\var}(|\xi|)t}|U_{L,\varepsilon}(0,\xi)|^2\quad\text{where}\quad \lambda_{\var}(|\xi|):=-\frac{c_0|\xi|^2}{(1+\var^2|\xi|^2)(1+|\xi|^2)}.
\end{aligned}
\end{equation}
Compared to \eqref{pointwisevar},  the pointwise estimate \eqref{pointwise0} allows us to 
keep track of the parameter $\var$. Consequently, the spectral behaviour of the solutions depending on the frequency-regions can be depicted as follows:
\begin{itemize}
    \item 
   $\lambda_{\var}(|\xi|)\sim -c_{0}|\xi|^2$, for $|\xi|\lesssim 1$ (the low-frequency region);
   \item $\lambda_{\var}(|\xi|)\sim -c_{0}$, \hspace{4mm} for  $1\lesssim |\xi|\lesssim 1/\var$ (the medium-frequency region); 
   \item $\lambda_{\var}(|\xi|)\sim -\frac{c_{0}}{ \var^2 |\xi|^2}$, \hspace{2mm} for $|\xi|\gtrsim 1/\var$  (the high-frequency region).
\end{itemize}
That is, the solutions behave like the heat kernel in low frequencies, undergo a damping effect in the medium frequencies and, in high frequencies, a loss of regularity occurs. The precise behaviour of each component is drawn in Table \ref{table:1}, see also \eqref{pointwise02}.
\begin{table}[h!]
    \centering
    \begin{tabular}{|c|c|c|c|}
       \hline & $|\xi|\leq 1$ & $1\leq |\xi|\leq \frac{C}{\var}$ & $|\xi|\geq \frac{C}{\var}$ \\
     \hline   $\rho^{\var}-\bar{\rho}$ &  Damped & Heat & Damped\\
      \hline   $u^{\var}$ &  Damped & Damped & Damped\\
     \hline   $E^{\var}$ &  Damped & Damped & Regularity-loss\\
     \hline   $B^{\var}-\bar{B}$ & Heat & Damped & Regularity-loss\\
     \hline
    \end{tabular} 
    \caption{Behaviours of each component of the Euler-Maxwell system \eqref{EMvarl}.}
    \label{table:1}
\end{table}

The above spectral analysis suggests us to split
the frequency space into three regimes: low, medium and high frequencies.  
This contrasts with \cite{CBD3,danchinnoterelaxation}, where only two frequency regimes employed. Moreover, due to the different behaviour observed in each regime, one must develop different {\emph{hypocoercivity}} methods in each regime to recover the expected dissipation properties stated in  Table \ref{table:1}. We design a functional framework allowing us to obtain uniform estimates with respect to $\varepsilon$. The framework we employ is depicted in Figure \ref{fig:1}.
\begin{figure}[!h]
    \centering
	\begin{tikzpicture}[xscale=0.4,yscale=0.4, thick]
	\draw [-latex] (-1,0) -- (21,0) ;
\draw  (21,0)  node [below]  {$|\xi|$} ;

\draw  (14,-0.2)  node [below]  {$\frac{C}{\var}$} ;
\draw  (7,-0.4)  node [below]  {$1$};
\draw  (-1,-0.2)  node [below]  {$0$} ;

\draw (14,0) node {$|$};
\draw (7,0) node {$|$};
\draw (-1,0) node {$|$};

\draw (17.5,1) node [above] {High};
\draw (17.5,0) node [above] {frequencies};
\draw (17.5,-0.4) node [below] {$\dot{B}^{\frac d2+1}$};

\draw (10.8,0.9) node [above] {Medium};
\draw (10.8,0) node [above] {frequencies};
\draw (10.8,-0.4) node [below] {$\dot{B}^{\frac d2}$};

\draw (3.3,1) node [above] {Low};
\draw (3.35,0) node [above] {frequencies};
\draw (3.35,-0.4) node [below] {$\dot{B}^{\frac d2-1}$};

	\end{tikzpicture}
    \caption{Frequency splitting for the Euler-Maxwell system \eqref{EM1}.}
        \label{fig:1}
\end{figure}
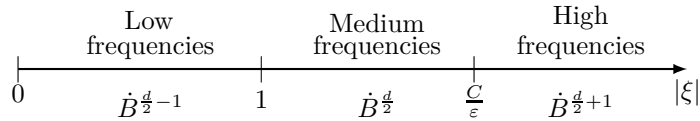

As $\var\rightarrow0$, we observe that the high-frequency regime disappears and the medium-frequency regime becomes the new high-frequency regime, see Figure \ref{fig:2}. This is coherent as the density $\rho^{\var}$ in the low and medium frequencies behaves like the solution $\rho^{*}$ of the limit drift-diffusion system \eqref{DD} (cf. Figure \ref{fig:2}). 

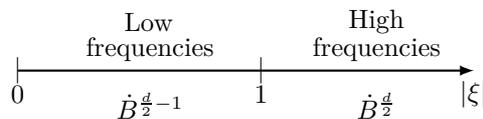
\begin{figure}[!h]
    \centering
	\begin{tikzpicture}[xscale=0.4,yscale=0.4, thick]
	\draw [-latex] (-1,0) -- (14,0) ;
\draw  (14,0)  node [below]  {$|\xi|$} ;
\draw  (7,-0.2)  node [below]  {$1$};
\draw  (-1,-0.2)  node [below]  {$0$} ;

\draw (7,0) node {$|$};
\draw (-1,0) node {$|$};

\draw (10.8,0.9) node [above] {High};
\draw (10.8,0) node [above] {frequencies};
\draw (10.8,-0.4) node [below] {$\dot{B}^{\frac d2}$};

\draw (3.3,1) node [above] {Low};
\draw (3.35,0) node [above] {frequencies};
\draw (3.35,-0.4) node [below] {$\dot{B}^{\frac d2-1}$};
	\end{tikzpicture}
    \caption{Frequency splitting for the drift-diffusion model \eqref{DD}.}
        \label{fig:2}
\end{figure}
Moreover, such a functional setting allows us to derive quantitative error estimates of solutions between \eqref{EMvar} and \eqref{DD}. A key ingredient  is the introduction of a new unknown: the effective velocity 
\begin{align}
z^{\var}:=u^{\var}+ \nabla h(\rho^{\var})+ E^{\var}+\var u^{\var}\times \bar{B},\label{dampmode}
\end{align}
which is associated with Darcy's law \eqref{darcy}. The unknown $z^\varepsilon$ satisfies stronger dissipative properties compared to the other components and exhibits a $\mathcal{O}(\varepsilon)$-bound. This is crucial to establish a global-in-time strong relaxation result in the whole space and derive the sharp convergence rate $\var$.




\section{Preliminaries and main results}\label{sec:mainr}

\subsection{Notations}
Before stating our main results, we explain the notations and definitions employed throughout the paper. $C>0$ denotes a constant independent of $\var$ and time, $f\lesssim g~(\text{resp.}\;f\gtrsim g)$ means $f\leq Cg~(\text{resp.}\;f\geq Cg)$, and $f\sim g$ stands for $f\lesssim g$ and $f\gtrsim g$. For any Banach space $X$ and functions $f,g\in X$, $\|(f,g)\|_{X}:=\|f\|_{X}+\|g\|_{X}$. For any $T>0$ and $1\leq \varrho\leq\infty$, we denote by $L^{\varrho}(0,T;X)$ the set of measurable functions $g:[0,T]\rightarrow X$ such that $t\mapsto \|g(t)\|_{X}$ is in $L^{\varrho}(0,T)$ and write $\|\cdot\|_{L^{\varrho}(0,T;X)}:=\|\cdot\|_{L^{\varrho}_{T}(X)}$.  $\mathcal{F}$ and $\mathcal{F}^{-1}$ stand for the Fourier transform and its inverse, respectively. In addition, we define $\Lambda^{\sigma}f:=\mathcal{F}^{-1}(|\xi|^{\sigma}\mathcal{F}f)$ for $\sigma\in \mathbb{R}$. 


\smallbreak

Following the pre-analysis in Subsection \ref{subsectionmotivation}, we introduce the threshold $J_{\var}$ between medium and high frequencies as
\begin{align}
    J_{\var}:=-[\log_{2}\var]+1,\label{J}
\end{align}
such that $2^{J_{\var}}\sim 1/\var$. 
The Littlewood-Paley decomposition and homogeneous Besov spaces emerge as natural tools for decomposing the analysis of our system in each frequency regime.
 We define the frequency-restricted Besov semi-norms corresponding to the three-regime decomposition:
\begin{equation}\nonumber
\begin{aligned}
&\|u\|_{\dot{B}^{s} }^{\ell}:=\sum_{j\leq 0} 2^{js}\|u_{j}\|_{L^{2}},\quad \|u\|_{\dot{B}^{s} }^{m}:=\sum_{-1\leq j\leq J_{\var}} 2^{js}\|u_{j}\|_{L^{2}}\quad \text{and}\quad\|u\|_{\dot{B}^{s} } ^{h}:=\sum_{j\geq  J_{\var}-1} 2^{js}\|u_{j}\|_{L^{2}},
\end{aligned}
\end{equation}
where 
$u_{j}:=\dot{\Delta}_{j}u$ and $\dot{\Delta}_{j}$ is the classical homogeneous Littlewood-Paley frequency-localization operator, see \cite[Chapter 2]{HJR}.  Analogously, we decompose $u=u^{\ell}+u^{m}+u^{h}$ as 
\begin{equation}\nonumber
\begin{aligned}
&u^{\ell}:=\sum_{j\leq -1} u_{j},\quad\quad u^{m}:=\sum_{0\leq j\leq J_{\var}-1} u_{j}\quad\text{and}\quad u^{h}:=\sum_{j\geq J_{\var}} u_{j}.
\end{aligned}
\end{equation}
Note that using Young's and Bernstein's inequalities, we have
$$
\|u^{\ell}\|_{\dot{B}^{s} }\lesssim\|u\|_{\dot{B}^{s} }^{\ell},\quad \|u^{m}\|_{\dot{B}^{s} }\lesssim\|u\|_{\dot{B}^{s} }^{m},\quad \|u^{h}\|_{\dot{B}^{s} }\lesssim \|u^{h}\|_{\dot{B}^{s} },
$$
and for any $s'>0$, the following inequalities hold true 
\begin{equation}\label{bernstein}
\left\{
\begin{aligned}
&\|u\|_{\dot{B}^{s} }^{\ell}\lesssim \|u\|_{\dot{B}^{s-s'} }^{\ell},\quad\quad\quad\quad  \quad \|u\|_{\dot{B}^{s} }^{m}\lesssim \|u\|_{\dot{B}^{s+s'} }^{m},\\
&\|u\|_{\dot{B}^{s} }^{m}\lesssim \var^{-s'}\|u\|_{\dot{B}^{s-s'} }^{m}, \quad\quad \quad~  \|u\|_{\dot{B}^{s} }^{h}\lesssim \var^{s'}\|u\|_{\dot{B}^{s+s'} }^{h}.
\end{aligned}
\right.
\end{equation}
 
\smallbreak

To justify the relaxation limit and analyse the drift-diffusion model \eqref{DD}, we also introduce the (independent of $\varepsilon$) hybrid Besov spaces
$$
\dot{B}^{s_{1},s_{2}}:=\{u\in \mathcal{S}_{h}'~: ~\|u\|_{\dot{B}^{s_{1},s_{2}}}:=\sum_{j\leq 0} 2^{js_{1}}\|u_{j}\|_{L^{2}}+\sum_{j\geq -1} 2^{js_{2}}\|u_{j}\|_{L^{2}}<\infty\},
$$
which verify the following properties: 
\begin{equation}\nonumber
\begin{aligned}
\dot{B}^{s_{1},s_{2}}=\dot{B}^{s_{1}}\quad&\text{if}\quad s_{1}=s_{2},\\
\dot{B}^{s_{1},s_{2}}=\dot{B}^{s_{1}}\cap\dot{B}^{s_{2}}\quad&\text{if}\quad s_{1}<s_{2},\\
\dot{B}^{s_{1},s_{2}}=\dot{B}^{s_{1}}+\dot{B}^{s_{2}}\quad&\text{if}\quad s_{1}>s_{2}.
\end{aligned}
\end{equation}
Furthermore, we denote the Chemin-Lerner spaces $\widetilde{L}^{\varrho}(0,T;\dot{B}^{s}_{p,r})$ by the function set in $L^{\varrho}(0,T;\mathcal{S}'_{h})$ endowed with the norm
\begin{eqnarray}
\nonumber\|u\|_{\widetilde{L}^{\varrho}_{T}(\dot{B}^{s})}:=
\begin{cases}
\underset{j\in\mathbb{Z}}{\sum}2^{js}\|u_{j}\|_{L^{\varrho}_{T}(L^{p})}<\infty,
& \mbox{if $1\leq \varrho<\infty,$ } \\
\underset{j\in\mathbb{Z}}{\sum}2^{js}\sup\limits_{t\in[0,T]}\|u_{j}\|_{L^{p}}<\infty,
& \mbox{if $\varrho=\infty.$}
\end{cases}
\end{eqnarray}
Using Minkowski's inequality, we have
\begin{equation}\nonumber
\begin{aligned}
&\|u\|_{\widetilde{L}^{1}_{T}(\dot{B}^{s})}= \|u\|_{L^{1}_{T}(\dot{B}^{s})}\quad \text{and}\quad \|u\|_{\widetilde{L}^{\varrho}_{T}(\dot{B}^{s})}\geq \|u\|_{L^{\varrho}_{T}(\dot{B}^{s})} \:\:\text{for}\:\: \varrho>1.
\end{aligned}
\end{equation}



\subsection{Main results}\label{subsectionmain}
To state our results, it is convenient to define the energy functional
\begin{equation}\label{Et}
\begin{aligned}
\mathcal{E}(a,u,E,H)&:=\|(a,\var u,E,H)\|_{\widetilde{L}^{\infty}_{t}( \dot{B}^{\frac{1}{2}} )}^{\ell}+\|(a,\var u,E,H)\|_{\widetilde{L}^{\infty}_{t}( \dot{B}^{\frac{3}{2}} )}^{m}+\var\|(a,\var u,E, H)\|_{\widetilde{L}^{\infty}_{t}(\dot{B}^{\frac{5}{2}} )}^{h},
\end{aligned}
\end{equation}
and the corresponding dissipation functional 
\begin{equation}\label{Dt}
\begin{aligned}
\mathcal{D}(a,u,E,H)&=\|a\|_{\widetilde{L}^{2}_{t}( \dot{B}^{\frac{1}{2}})}^{\ell}+\|u\|_{\widetilde{L}^{2}_{t}( \dot{B}^{\frac{1}{2}})}^{\ell}+\|E\|_{\widetilde{L}^{2}_{t}( \dot{B}^{\frac{1}{2}})}^{\ell}+\| H\|_{\widetilde{L}^{2}_{t}(\dot{B}^{\frac{3}{2}} )}^{\ell}\\
&~+\|a\|_{\widetilde{L}^{2}_{t}( \dot{B}^{\frac{5}{2}} )}^{m}+\|u\|_{\widetilde{L}^{2}_{t}( \dot{B}^{\frac{3}{2}})}^{m}+\|E\|_{\widetilde{L}^{2}_{t}( \dot{B}^{\frac{3}{2}})}^{m}+\| H\|_{\widetilde{L}^{2}_{t}(\dot{B}^{\frac{3}{2}} )}^{m}\\
&~+\|a\|_{\widetilde{L}^2_{t}(\dot{B}^{\frac{5}{2}} )}^{h}+\var\|u\|_{\widetilde{L}^{2}_{t}(\dot{B}^{\frac{5}{2}} )}^{h}+\|E\|_{\widetilde{L}^{2}_{t}( \dot{B}^{\frac{3}{2}})}^{h}+\| H\|_{\widetilde{L}^{2}_{t}(\dot{B}^{\frac{3}{2}} )}^{h}
\end{aligned}
\end{equation}
for $t>0$. The initial energy is denoted as follows:
\begin{equation}
\begin{aligned}
\mathcal{E}^{\var}_{0}:&=\|(\rho_{0}-\bar{\rho}, u_{0},E_{0},B_{0}-\bar{B})\|_{ \dot{B}^{\frac{1}{2}}}^{\ell}+\|(\rho_{0}-\bar{\rho},  u_{0},E_{0},B_{0}-\bar{B})\|_{\dot{B}^{\frac{3}{2}}}^{m}\\
&\quad+\var\|(\rho_{0}-\bar{\rho}, u_{0},E_{0},B_{0}-\bar{B})\|_{\dot{B}^{\frac{5}{2}}}^{h}.\label{E0}
\end{aligned}
\end{equation}

Our first result provides the global existence and uniqueness of classical solutions to \eqref{EMvar}-\eqref{EMdvar}, uniformly with respect to the relaxation parameter $\varepsilon$.

\begin{theorem}\label{theorem1}Let $0<\varepsilon\leq 1$.
There exists a constant $\alpha_{0}$ independent of $\var$
such that if 
\begin{equation}\label{a1}
\begin{aligned}
&\mathcal{E}^{\var}_{0}\leq \alpha_{0},
\end{aligned}
\end{equation}
then the Cauchy problem \eqref{EMvar}-\eqref{EMdvar} admits a unique global-in-time classical solution $(\rho^{\var},u^{\var},E^{\var},B^{\var})$ fulfilling $(\rho^{\var}-\bar{\rho},u^{\var},E^{\var}, B^{\var}-\bar{B})\in \mathcal{C}(\mathbb{R}^{+}; \dot{B}^{\frac{1}{2},\frac{5}{2}} )$. Moreover, the following uniform estimate holds{\rm:}
\begin{equation}\label{r1}
\begin{aligned}
&\mathcal{E}(\rho^{\var}-\bar{\rho},u^{\var},E^{\var},B^{\var}-\bar{B})+\mathcal{D}(\rho^{\var}-\bar{\rho},u^{\var},E^{\var},B^{\var}-\bar{B})\leq C\mathcal{E}^{\var}_{0},
\end{aligned}
\end{equation}
for all $t\in\mathbb{R}_{+}$,
where
$C>0$ is a uniform constant independent of $\var$ and $t$.
\end{theorem}

\begin{remark}As observed in Table \ref{table:1}, the non-symmetric relaxation term induces a one-regularity loss phenomenon in the high-frequency regime. To deal with this difficulty, we partition the frequency space into three distinct regimes associated with the different behaviour of the solution. In addition, Theorem \ref{theorem1} provides a larger regularity framework for the well-posedness of classical solutions of \eqref{EMvar}-\eqref{EMdvar}. This can be observed in the following chain of Sobolev embeddings
$$
H^{s}(s>\frac{5}{2})\hookrightarrow B^{\frac{5}{2}}\hookrightarrow \dot{B}^{\frac{1}{2},\frac{5}{2}}\hookrightarrow \mathcal{C}^1\cap W^{1,\infty}.
$$ 
The left space corresponds to the classical Sobolev theory, see for instance \cite{Duan2011,PYJEMSIAM,UedaKawa2011,UedaWangKawa2012}. Compared to the result in the inhomogeneous Besov space $B^{\frac{5}{2}}$, see \cite{XuEM,XuEMtwo}, the result of the present paper -- $\dot{B}^{\frac{1}{2},\frac{5}{2}}$-- allows to assume less regularity on the low frequencies of the initial data, i.e. $B^{1/2}$ rather than $L^2$.

\end{remark}






\begin{remark}
Theorem \ref{theorem1} provides a sharp smallness condition \eqref{a1} for the global existence of the Euler-Maxwell system. Notice that we only assume the low and medium-frequency norms of initial data to be small, the high-frequency norm, actually, can be arbitrarily large when $\var$ is suitably small. This comes from the fact that as $\var\rightarrow0$,  the high-frequency regime disappears and the medium-frequency regime becomes the new high-frequency one. See Figure \ref{fig:2}.
\end{remark}




Next, we establish quantitative error estimates for ill-prepared initial data, which leads to   the strong relaxation limit from (\ref{EMvar})-\eqref{EMdvar} to (\ref{DD}). 

\begin{theorem}\label{theorem4}
Let $0<\var\leq 1$ and $(\rho^{\var},u^{\var},E^{\var},B^{\var})$ be the solution of \eqref{EMvar}-\eqref{EMdvar} from Theorem \ref{theorem1}. Let $\rho^{*}$ be the solution of \eqref{DD} associated to the initial datum $\rho^{*}_{0}$ given by Theorem \ref{theorem2}.
Define the effective velocity
\[ z^{\var}:=u^{\var}+\nabla h(\rho^{\var})+E^{\var}+\var u^{\var}\times\bar{B} \]
and its initial datum
\[  z^{\var}_{0}:=\frac{1}{\var}u_{0}+\nabla h(\rho_{0})+E_0+u_0\times\bar{B}. \]
Then, it holds that
\begin{equation}
\begin{aligned}
& \|z^{\var}-z^{\var}_{L}\|_{\widetilde{L}^2_{t}(\dot{B}^{\frac{1}{2}} )}\leq C\var,\quad t\in\mathbb{R}_{+},\label{enhancez0}
\end{aligned}
\end{equation}
where $C>0$ is a constant independent of $\var$ and time, and $z^{\var}_{L}:=e^{-\frac{t}{\var^2}}z^{\var}_0$ is the initial layer correction of $z^{\var}$.

Let $E^{*}=\nabla(-\Delta)^{-1}(\rho^{*}-\bar{\rho})$ associated with its initial datum $E_{0}^{*}=\nabla(-\Delta)^{-1}(\rho_{0}^{*}-\bar{\rho})$ and set $B^{*}=\bar{B}$. 
If we assume $\rho_0^*-\bar{\rho}\in \dot{B}^{\frac{1}{2}}$ and 
\begin{align}
\|(\rho_{0}-\rho_{0}^*,E_{0}  -E_{0}^{*},B_{0}-\bar{B})\|_{\dot{B}^{\frac{1}{2}}}\leq \var,\label{aerror2}
\end{align}
then, for all $t\in \R_+$, 
\begin{equation}\label{convergence2}
\begin{aligned}
&\|\rho^{\var}  -\rho^{*}\|_{\widetilde{L}^{\infty}_{t}(\dot{B}^{\frac{1}{2}} )\cap \widetilde{L}^{2}_{t}( \dot{B}^{\frac{1}{2},\frac{3}{2}} )}+\|u^{\var}  -u^{*}-u^{\var}_{L}\|_{\widetilde{L}^{2}_{t}(\dot{B}^{\frac{1}{2}} )}\\
&\quad\quad\quad\quad+\|E^{\var} -E^{*}\|_{\widetilde{L}^{\infty}_{t}(\dot{B}^{\frac{1}{2}})\cap \widetilde{L}^{2}_{t}(\dot{B}^{\frac{1}{2}})}+\|B^{\var}-B^{*}\|_{\widetilde{L}^{\infty}_{t}(\dot{B}^{\frac{1}{2}})\cap\widetilde{L}^{2}_{t}(\dot{B}^{\frac{3}{2},\frac{1}{2}})} \leq C \var,
\end{aligned}
\end{equation}
where $u_{L}^{\var}:=e^{-\frac{t}{\var^2}}\frac{1}{\var}u_0$ is the initial layer correction of $u^{\var}$. 
\end{theorem}

\begin{remark}







Theorem \ref{theorem4} is, to the best of our knowledge,  the first result providing global-in-time convergence rates of the compressible Euler-Maxwell system towards the drift-diffusion system in $\mathbb{R}^3$. Thanks to the initial layer corrections $z_{L}^{\var}$ and $u_{L}^{\var}$ in \eqref{enhancez0} and \eqref{convergence2}, the strong convergence can hold for general ill-prepared initial data.


\end{remark}

\subsection{Strategy to derive error estimates}
We now explain the strategies for establishing error estimates of the relaxation limit from \eqref{EMvar} to \eqref{DD}. 
Our first step is the introduction of the effective velocity $z^{\var}$ which reveals the convergence of $u^{\var}+\nabla h(\rho^{\var})+E^{\var}$ towards Darcy's law \eqref{darcy}. Let $(\delta \rho, \delta u,\delta E,\delta B):=(\rho^{\var}-\rho^{*},u^{\var}-u^{*},E^{\var} -E^{*},B^{\var} -B^{*})$ be the error unknowns. We observe that $\delta \rho$ satisfies
\begin{equation}\label{errorrho}
\begin{aligned}
&\partial_{t}\delta\rho-P'(\bar{\rho})\Delta\delta\rho+\bar{\rho}\delta\rho=-\bar{\rho}\div  z^{\var}+\var\bar{\rho} \div(u^{\var}\times\bar{B})+\text{nonlinear terms},
\end{aligned}
\end{equation}
where the left-hand side of \eqref{errorrho} presents a priori estimates of the linearized drift-diffusion system, and the term $\var\bar{\rho} \div(u^{\var}\times\bar{B})$ give an $\mathcal{O}(\var)$-bound due to \eqref{r1}. Hence, one has to establish the decay-in-$\varepsilon$ of the remainder term $-\bar{\rho}\div  z^{\var}$. On the one hand, we find that  $z^{\var}$ 
satisfies
\begin{align}
\partial_{t} z^{\var}+\frac{1}{\var^2 }z^{\var}-\frac{1}{\var}z\times \bar{B}=\text{higher-order linear and nonlinear terms}.\nonumber
\end{align}
The above damping structure enables us to derive $\mathcal{O}(\varepsilon)$-bounds for $z^{\var}$ (see Proposition \ref{propenhancez}) and thus to control $-\bar{\rho}\div  z^{\var}$.
On the other hand, we reformulate the system of $(\delta E,\delta B)$ in terms of the effective velocity $z^{\var}$:
\begin{equation}\label{deltaBH}
\left\{
\begin{aligned}
&\partial_{t}\delta E-\frac{1}{\var}\nabla\times \delta B+\bar{\rho}\delta E-P'(\bar{\rho})\nabla\div \delta E=\nabla\times B^{1,*}+\bar{\rho}z^{\var}-\var\bar{\rho}u^{\var}\times\bar{B}+\text{nonlinear terms},\\
&\partial_{t}\delta B+\frac{1}{\var}\nabla\times \delta E=0,\\
&\div \delta E=-\delta\rho,\quad\quad \div \delta B=0,
\end{aligned}
\right.
\end{equation}
with  $B^{1,*}=-(-\Delta)^{-1}\nabla\times(\rho^{*}u^{*})$. One can see that the  dissipative structure of \eqref{deltaBH} share similarities with that of the compressible Euler system with damping. Consequently, we derive qualitative estimates for $(\delta E,\delta B)$ by  employing a {\emph{hypocoercivity}} argument as in \cite{BZ,CBD3}. Nevertheless,  there is an additional difficulty arising from the term $B^{1,*}$, which lacks the $\mathcal{O}(\varepsilon)$-bound in  \eqref{deltaBH} in fact. To overcome it, we employ the \textit{auxiliary} unknown
$$
\delta\mathcal{B}:=\delta B+\var B^{1,*}
$$
which allows us to rewrite \eqref{deltaBH} in terms of $(\delta \mathcal{B}, \delta E)$ without the term $\nabla\times B^{1,*}$ and establish desired convergence estimates. See Subsection \ref{subsectionill} for more details.

\vspace{2mm}
\subsection{Outline of the paper}
The rest of the paper unfolds as follows. 
In Section \ref{sectiontheorem1}, we derive uniform a priori estimates for \eqref{EM1} and prove the global well-posedness of the Cauchy problem \eqref{EMvar}-\eqref{EMdvar} (Theorem \ref{theorem1}). Section \ref{sectionrelaxation} is dedicated to the justification of the strong relaxation limit from \eqref{EMvar}-\eqref{EMdvar} to \eqref{DD} (Theorem \ref{theorem4}). 
  Finally, technical lemmas that are used throughout the manuscript are presented in Appendix \ref{app1}.

\section{Global well-posedness for the Euler-Maxwell system}\label{sectiontheorem1}
In this section, we focus on the proof of Theorem \ref{theorem1}. We simplify the notations of unknowns by omitting the superscript $\var$. Define 
\begin{align}
n :=h(\rho )-h(\Bar{\rho} ),\quad\quad H: =B -\bar{B},\label{enth}
\end{align}
where $h(\rho)$ is the enthalpy satisfying $h'(\rho)=P'(\rho)/\rho$. The system \eqref{EMvar} for $(t,x)\in [0,\infty)\times\mathbb{R}^3$ can be rewritten as
\begin{equation}\label{EM1}
\left\{
\begin{aligned}
&\partial_{t}n+P'(\bar{\rho})\div u  =-u \cdot\nabla n -G(n )\div u ,\\
&\var^2(\partial_{t}u +u \cdot\nabla u )+\nabla n +E + u +\var u \times \bar{B}=-\var u \times H ,\\
&\var\partial_{t}E -\nabla\times H -\var\bar{\rho} u =\var F(n )u ,\\
&\var\partial_{t}H +\nabla\times E =0,\\
&\div E =-Kn -\Phi(n ),\\
&\div H =0,\\
&(n ,u ,E ,H )(0,x)=(n_{0}, \frac{1}{\var}u_0,E_{0},H_{0})(x),
\end{aligned}
\right.
\end{equation}
with 
\begin{equation}
\left\{
\begin{aligned}
&n_{0}:=h(\rho_{0} )-h(\Bar{\rho} ), \quad H_{0}:=B_{0}-\bar{B},\\
&K:=\rho'(0)=\frac{\bar{\rho}}{P'(\bar{\rho})},\\
&G(n):=P'(\rho)-P'(\bar{\rho}),\quad\quad F(n):=\rho(n)-\bar{\rho},\quad\quad \Phi(n):= \rho(n)-\bar{\rho}- K n .\label{GFPHI}
\end{aligned}
\right.
\end{equation}
Note that $\Phi(n)=\mathcal{O}(|n|^2)$ if $n$ is uniformly bounded. For clarity, we split the proof into several subsections.
\subsection{Pointwise estimates for the linearized Euler-Maxwell system}\label{sectionpointwise}

In order to get a priori estimates with optimal regularity, we first derive pointwise estimates for the following linearized Euler-Maxwell system
\begin{equation}\label{EMvarl}
\left\{
\begin{aligned}
&\partial_{t}n+P'(\bar{\rho})\div u =0,\\
&\var^2\partial_{t}u+\nabla n+E+u+\var u\times \bar{B}=0,\\
&\var\partial_{t}E-\nabla\times H-\var\bar{\rho} u=0,\\
&\var\partial_{t}H+\nabla\times E=0,\\
&\div E=-Kn,\quad\quad \div H=0.\\
\end{aligned}
\right.
\end{equation}
In what follows, we  employ a hypocoercivity argument and deduce uniform-in-$\varepsilon$ pointwise estimates for  \eqref{EMvarl}, which   provide us an insight into the evolution of the dissipation rates with respect to $\var$.


\begin{proposition}\label{proppointwise}
For $0<\var\leq1$, let $(n,u,E,H)$ be the solution to the system \eqref{EMvarl}. Then there exists a functional $\mathcal{L}_{\xi}(t)\sim |(\widehat{n}, \var \widehat{u},\widehat{E},\widehat{H})(t,\xi)|^2$ and a constant $c_{0}=c_{0}(\bar{\rho},\bar{B},P'(\bar{\rho}))>0$ such that
\begin{equation}\label{pointwise02}
\begin{aligned}
& \frac{d}{dt}\mathcal{L}_{\xi}(t)+c_{0} |\widehat{u}|^2 + \frac{c_{0}( 1+|\xi|^2)}{1+\var^2|\xi|^2}|\widehat{n}|^2\\
&\quad\quad\quad\quad+\frac{ c_{0}}{1+\var^2|\xi|^2}|\widehat{E}|^2+\frac{c_{0} |\xi|^2}{(1+\var^2|\xi|^2)(1+|\xi|^2)}|\widehat{H}|^2\leq 0.
\end{aligned}
\end{equation}
Furthermore, we have 
\begin{equation}\label{pointwise}
\begin{aligned}
&|(\widehat{n}, \var \widehat{u},\widehat{E},\widehat{H})(t,\xi)|^2\lesssim e^{\lambda_{\var}(|\xi|)t}(\widehat{n}, \var \widehat{u},\widehat{E},\widehat{H})(0,\xi)|^2,\quad\quad t>0,\quad \xi\in\mathbb{R}^{d},
\end{aligned}
\end{equation}
where $\lambda_{\var}(|\xi|)$ is given by
\begin{align}
\lambda_{\var}(|\xi|)=-\frac{c_{0} |\xi|^2}{(1+\var^2|\xi|^2)(1+|\xi|^2)}.\nonumber
\end{align}
\end{proposition}
\begin{proof}
   Applying the Fourier transform to \eqref{EMvarl} gives
\begin{equation}\label{EMvarll}
\left\{
\begin{aligned}
&\partial_{t}\widehat{n }+P'(\bar{\rho})i\xi \widehat{u } =0,\\
&\var^2\partial_{t}\widehat{u }+i \xi \widehat{n }+\widehat{E }+\widehat{u }+ \var \widehat{u }\times \bar{B}=0,\\
&\var \partial_{t}\widehat{E }-i\xi\times  \widehat{H }-\var \bar{\rho} \widehat{u }=0,\\
&\var \partial_{t}\widehat{H }+i\xi \times \widehat{E }=0,\\
&i\xi \widehat{E }=-K\widehat{n },\quad\quad i\xi \widehat{H }=0,
\end{aligned}
\right.
\end{equation}
where we recall that $K=\dfrac{\bar{\rho}}{P'(\bar{\rho})}$.
Performing the  inner product of \eqref{EMvarll} with $(\widehat{n }, P'(\bar{\rho})\widehat{u },\frac{1}{K}\widehat{E},\frac{1}{K}\widehat{H})^T $ and taking the real part, we obtain
\begin{equation}\label{p1}
\begin{aligned}
&\frac{1}{2}\frac{d}{dt}\Big(|\widehat{n }|^2+P'(\bar{\rho})\var^2|\widehat{u}|^2+\frac{1}{K} |\widehat{E}|^2+\frac{1}{K} |\widehat{H}|^2\Big)+P'(\bar{\rho})|\widehat{u}|^2=0.
\end{aligned}
\end{equation}
To capture dissipation for $n$, we have
\begin{equation}\label{p21}
\begin{aligned}
-\var^2\frac{d}{dt}{\rm Re}< \widehat{u}, i\xi \widehat{n}>+|\xi|^2 |\widehat{n}|^2+K |\widehat{n}|^2
&={\rm Re}<\widehat{u }+ \var \widehat{u }\times \bar{B}, i\xi \widehat{n}>+P'(\bar{\rho})\var^2|\xi\cdot \widehat{u}|^2\\
&\leq \frac{1}{2}|\xi|^2 |\widehat{n}|^2+C(1+\var^2|\xi|^2)|\widehat{u }|^2.
\end{aligned}
\end{equation}
Then, multiplying \eqref{p21} by $\frac{1}{1+\var^2|\xi|^2}$, we obtain
\begin{equation}\label{p2}
\begin{aligned}
&-\frac{d}{dt}\frac{\var^2 {\rm Re}< \widehat{u}, i\xi \widehat{n}>}{1+\var^2|\xi|^2}+\frac{|\xi|^2}{2(1+\var^2|\xi|^2)} |\widehat{n}|^2+ \frac{K }{1+\var^2|\xi|^2}|\widehat{n}|^2\leq C|\widehat{u}|^2.
\end{aligned}
\end{equation}
Performing the inner scalar product of the second equation in \eqref{EMvarll} with $\widehat{E}$ (associated with the skew-symmetric part of the relaxation matrix), and then using the third equation in  \eqref{EMvarll} implies that 
\begin{equation}\label{p31}
\begin{aligned}
&\var^2\frac{d}{dt}{\rm Re}< \widehat{u},  \widehat{E}>+|\widehat{E}|^2+\frac{1}{K}|\xi\cdot \widehat{E}|^2\\
&=-{\rm Re}<\widehat{u }+ \var\widehat{u }\times \bar{B}, \widehat{E}>+\var{\rm Re}<i\xi\times\widehat{H},\widehat{u}>+ \var^2\bar{\rho}|\widehat{u}|^2\\
&\leq \frac{1}{2}|\widehat{E}|^2+\frac{C(1+\var^2 |\xi|^2)}{\sqrt{\eta} } |\widehat{u}|^2+\frac{C\sqrt{\eta}|\xi|^2}{1+|\xi|^2} |\widehat{H}|^2
\end{aligned}
\end{equation}
for  $\eta\in(0,1)$ to be chosen later. In order to be consistent with the dissipation of $u$ in \eqref{p1}, we multiply both sides of \eqref{p31} by $\frac{1}{1+\var^2|\xi|^2}$ and obtain
\begin{equation}\label{p3}
\begin{aligned}
&\frac{d}{dt}\frac{\var^2 {\rm Re}< \widehat{u},  \widehat{E}>}{1+\var^2|\xi|^2}+\frac{ 1}{2(1+\var^2|\xi|^2)}|\widehat{E}|^2\leq \frac{C}{\sqrt{\eta}}|\widehat{u}|^2+\frac{C\sqrt{\eta} |\xi|^2}{(1+\var^2|\xi|^2)(1+|\xi|^2)} |\widehat{H}|^2.
\end{aligned}
\end{equation}
To derive the dissipation of $\widehat{H}$, using $|\xi|^2 |\widehat{H}|^2=|\xi\times \widehat{H}|^2$ due to $\xi\cdot \widehat{H}=0$, it follows that
\begin{equation}
\begin{aligned}
\var\frac{d}{dt}{\rm Re}< \widehat{E}, -i\xi\times   \widehat{H}>+|\xi|^2 |\widehat{H}|^2
&~= |\xi\times \widehat{E}|^2-\bar{\rho}\var{\rm Re} <\widehat{u},i\xi\times   \widehat{H}>\\
&~\leq \frac{1}{2}|\xi|^2 |\widehat{H}|^2+C|\xi|^2 |\widehat{E}|^2+C|\widehat{u}|^2.
\end{aligned}
\end{equation}
In view of the dissipation of $\widehat{E}$ in \eqref{p3}, we have 
\begin{equation}\label{p4}
\begin{aligned}
&\frac{d}{dt}\frac{\var {\rm Re}< \widehat{E}, -i\xi\times   \widehat{H}>}{(1+\var^2|\xi|^2)(1+|\xi|^2)}+\frac{ |\xi|^2}{2(1+\var^2|\xi|^2)(1+|\xi|^2)}|\widehat{H}|^2\\
&\leq C|\widehat{u}|^2+\frac{ |\xi|^2}{(1+\var^2|\xi|^2)(1+|\xi|^2)} |\widehat{E}|^2.
\end{aligned}
\end{equation}
Then, we define the Lyapunov functional
\begin{equation}
\begin{aligned}
\mathcal{L}_{\xi}(t)\triangleq&\frac{1}{2}\bigg(|\widehat{n }|^2+P'(\bar{\rho})\var^2|\widehat{u}|^2+\frac{1}{K} |\widehat{E}|^2+\frac{1}{K} |\widehat{H}|^2\bigg)\\
&\quad-\eta \frac{\var^2 {\rm Re}< \widehat{u}, i\xi \widehat{n}>}{1+\var^2|\xi|^2}+\eta\frac{\var^2 {\rm Re}< \widehat{u},  \widehat{E}>}{1+\var^2|\xi|^2}+\eta^{\frac{5}{4}}\frac{\var {\rm Re}< \widehat{E}, -i\xi\times   \widehat{H}>}{(1+\var^2|\xi|^2)(1+|\xi|^2)}.
\end{aligned}
\end{equation}
It follows from \eqref{p1}, \eqref{p2}, \eqref{p3} and \eqref{p4} that
\begin{equation}
\begin{aligned}
&\frac{d}{dt}\mathcal{L}_{\xi}(t)+(P'(\bar{\rho})-C\eta-C\sqrt{\eta})|\widehat{u}|^2+ \frac{ \eta (1+|\xi|^2)}{1+\var^2|\xi|^2}|\widehat{n}|^2\\
&\quad+(\frac{1}{2}-\eta^{\frac{1}{4}})\eta\frac{ 1}{1+\var^2|\xi|^2}|\widehat{E}|^2+\eta^{\frac{5}{4}}(\frac{1}{2}-\eta^{\frac{1}{4}})(\frac{ |\xi|^2}{(1+\var^2|\xi|^2)(1+|\xi|^2)}|\widehat{H}|^2\leq 0.
\end{aligned}
\end{equation}
Choosing a suitable small constant $\eta$, we have $\mathcal{L}_{\xi}(t)\sim |(\widehat{n}, \var \widehat{u},\widehat{E},\widehat{H})|^2$ and the inequality \eqref{pointwise02} is proved. In particular, it holds that
\begin{equation}\label{pmmm}
\begin{aligned}
&\frac{d}{dt}\mathcal{L}_{\xi}(t)+\lambda_{\var}(\xi)\mathcal{L}_{\xi}(t)\leq 0,
\end{aligned}
\end{equation}
which leads to \eqref{pointwise} by Gr\"onwall's inequality.
\end{proof}

\subsection{Uniform a priori estimates and global well-posedness}\label{sectionpriori}
In this section, our central task is to derive uniform a priori estimates in the spirit of  Proposition \ref{proppointwise} and the work of Beauchard and Zuazua \cite{BZ}. This enables us to achieve the global existence of classical solutions to the Cauchy problem \eqref{EM1}. Denote 
\begin{align}
\mathcal{X}(t):=\mathcal{E}(n,u,E,H)+\mathcal{D}(n,u,E,H)\label{Xt}
\end{align} for $t>0$ and $0<\var\leq 1$, 
where $\mathcal{E}$ and $\mathcal{D}$ are defined by \eqref{Et} and \eqref{Dt}.

\begin{proposition}\label{propapriori} Assume that $(n,u,E,H)$ is a classical solution to \eqref{EM1} on the time interval $[0,T]$. There exist positive constants $\delta_0$ and $C_0$ independent of $\var$
such that for $t\in [0,T]$, if
\begin{equation}\label{prioribound}
\begin{aligned}
&\|n\|_{L^{\infty}_{t}(L^{\infty})}\leq  \delta_0,
\end{aligned}
\end{equation} then
it holds that
\begin{equation}
\begin{aligned}
&\mathcal{X}(t)\leq C_{0}\left(\mathcal{E}^{\var}_0+\mathcal{X}(t)^2+\mathcal{X}(t)^3\right),\label{uniformapriori}
\end{aligned}
\end{equation}
where the initial energy norm $\mathcal{E}_0^{\var}$ is given by \eqref{E0}.
\end{proposition}
The proof of the proposition \ref{propapriori} is a direct consequence of Lemmas \ref{lemmalow}-\ref{lemmahigh},
which are closely linked with the dissipation analysis (on  three distinct regimes) addressed in Section \ref{subsectionmotivation}. 

\begin{lemma}[Low-frequency estimates] \label{lemmalow}
If $(n,u,E,H)$ is a classical solution to \eqref{EM1} on the time interval $[0,T]$, then the following estimate holds: 
\begin{equation}\label{low}
\begin{aligned}
&\|(n,\var u,E,H)\|_{\widetilde{L}^{\infty}_{t}(\dot{B}^{\frac{1}{2}} )}^{\ell}+\|(n,u,E)\|_{\widetilde{L}^{2}_{t}( \dot{B}^{\frac{1}{2}})}^{\ell}+\|H\|_{\widetilde{L}^{2}_{t}(\dot{B}^{\frac{3}{2}})}^{\ell}\lesssim \|(n_{0}, u_{0},E_{0},H_{0})\|_{\dot{B}^{\frac{1}{2}} }^{\ell}+\mathcal{X}(t)^2
\end{aligned}
\end{equation}
for $t\in [0,T]$ and $0<\var\leq 1$.
\end{lemma}

\begin{proof}
Applying the frequency-localization operator $\dot{\Delta}_j$ to \eqref{EM1}, we obtain
\begin{equation}\label{EM1j}
\left\{
\begin{aligned}
&\partial_{t}n_{j}+P'(\bar{\rho}) \div u_{j}=-\dot{\Delta}_{j}(u\cdot\nabla n)-\dot{\Delta}_{j}(G(n) \div u),\\
&\var^2\partial_{t}u_{j}+\nabla n_{j}+E_{j}+ u_{j}+\var u_{j}\times \bar{B}=-\dot{\Delta}_{j}(\var^2 u\cdot\nabla u)-\dot{\Delta}_{j}(\var u\times H),\\
&\var \partial_{t}E_{j}-\nabla\times H_{j}-\var\bar{\rho} u_{j}=\dot{\Delta}_{j}(\var F(n)u),\\
&\var \partial_{t}H_{j}+\nabla\times E_{j}=0,\\
&\div E_{j}=-Kn_{j}-\dot{\Delta}_{j}\Phi(n),\quad\quad \div H_{j}=0.
\end{aligned}
\right.
\end{equation}
Taking the $L^2$-inner product of $\eqref{EM1j}_{1}$ with $n_{j}$, we have
\begin{equation}\label{E11}
\begin{aligned}
\frac{1}{2}\frac{d}{dt}\int |n_{j}|^2\,dx+P'(\bar{\rho}) \int \div u_{j} n_{j}  \,dx\leq (\|\dot{\Delta}_{j}(u\cdot\nabla n)\|_{L^2}+\|\dot{\Delta}_{j}(G(n)\div u)\|_{L^2})\|n_{j}\|_{L^2}.
\end{aligned}
\end{equation}
To cancel the second term on the left-hand side of \eqref{E11}, we take the $L^2$-inner product of $\eqref{EM1j}_{2}$ with $P'(\bar{\rho}) u_{j}$ to get
\begin{equation}\label{E12}
\begin{aligned}
&\frac{P'(\bar{\rho})\var^2}{2}\frac{d}{dt}\int |u_{j}|^2 \,dx+P'(\bar{\rho})  \int  \nabla n_{j} \cdot u_{j} \,dx+P'(\bar{\rho})\int |u_{j}|^2 \,dx+P'(\bar{\rho}) \int E_{j}\cdot u_{j}\,dx\\
&~\leq P'(\bar{\rho})\|\dot{\Delta}_{j}(\var u\cdot\nabla u, u\times H)\|_{L^2}\var\|u_{j}\|_{L^2},
\end{aligned}
\end{equation}
where the fact that $(u_{j}\times \bar{B})\cdot u_{j}=(u_{j}\times u_{j})\cdot\bar{B}=0$ was used. In addition, it follows from $\eqref{EM1j}_{3}$-$\eqref{EM1j}_{4}$ that
\begin{equation}\label{E13}
\begin{aligned}
&\frac{P'(\bar{\rho})}{2\bar{\rho}}\frac{d}{dt}\|(E_{j},H_{j})\|_{L^2}^2-P'(\bar{\rho})\int u_{j}\cdot E_{j}\,dx \leq \frac{1}{K}\|\dot{\Delta}_{j}(F(n)u)\|_{L^2}\|E_{j}\|_{L^2},
\end{aligned}
\end{equation}
where we have used 
$$
\int (\nabla\times f)\cdot g- (\nabla \times g)\cdot f \,dx=\int \div (f\times g)\,dx=0,\quad\quad \forall f,g\in \mathcal{S}'(\R^3).
$$
Combining \eqref{E11}-\eqref{E13} together, we have
\begin{equation}\label{E1}
\begin{aligned}
&\frac{1}{2}\frac{d}{dt}\int\Big(|n_{j}|^2+ P'(\bar{\rho})\var^2|u_{j}|^2+\frac{1}{K}|E_{j}|^2+\frac{1}{K} |H_{j}|^2\Big)\,dx+P'(\bar{\rho})\int |u_{j}|^2\,dx\\
&~\leq  (\|\dot{\Delta}_{j}(u\cdot\nabla n)\|_{L^2}+\|\dot{\Delta}_{j}(G(n)\div u)\|_{L^2})\|n_{j}\|_{L^2} \\
&~\quad+P'(\bar{\rho})\|\dot{\Delta}_{j}(\var u\cdot\nabla u,u\times H)\|_{L^2}\var \|u_{j}\|_{L^2}+\frac{1}{K}\|\dot{\Delta}_{j}(F(n)u)\|_{L^2}\|E_{j}\|_{L^2}.
\end{aligned}
\end{equation}
In order to obtain some dissipation rate for $n_{j}$, we multiply $\eqref{EM1j}_{2}$ by $\nabla n_{j}$ and integrate the resulting equality over $\mathbb{R}^{d}$. 
Since $\div E_{j}=-Kn_{j}-\dot{\Delta}_{j}\Phi(n)$, we see that $n_j$ satisfies
$$
\int E_{j}\cdot\nabla n_{j} \,dx =-\int \div E_{j} n_{j} \,dx=K\|n_{j}\|_{L^2}^2+\int \dot{\Delta}_{j}\Phi(n) n_{j} \,dx.
$$
Furthermore, with the help of the Cauchy-Schwarz inequality, we get
\begin{equation}\label{E2}
\begin{aligned}
&\var^2\frac{d}{dt}\int u_{j}\cdot \nabla n_{j} \,dx+\int \Big( |\nabla n_{j}|^2+K|n_{j}|^2-P'(\bar{\rho})\var^2|\div u_{j}|^2+u_{j}\cdot \nabla n_{j}\Big)\,dx\\
&\leq \var \|\dot{\Delta}_{j}(\var u\cdot\nabla u,u\times H)\|_{L^2}\|\nabla n_{j}\|_{L^2}+\var\|\nabla\dot{\Delta}_{j}(u\cdot\nabla n,G(n)\div u)\|_{L^2}\var\|u_{j}\|_{L^2}\\
&\quad+\|\dot{\Delta}_{j}\Phi(n)\|_{L^2} \|n_{j}\|_{L^2}.
\end{aligned}
\end{equation}
The nice ``div-curl" construction of Maxwell's equation in $\eqref{EM1}$ enables us to get dissipation for $(E,H)$. Concerning $E$, it comes from the interaction between the symmetric and skew-symmetric part of the zero-order dissipation matrix. Indeed, taking the inner product of $\eqref{EM1j}_{2}$ with $E_{j}$, using $\eqref{EM1j}_{3}$, $\eqref{EM1j}_{5}$ and that $n_{j}=-\frac{1}{K}\div E_{j}-\frac{1}{K}\dot{\Delta}_{j}\Phi(n)$, we arrive at
\begin{equation}\label{E3}
\begin{aligned}
&\var^2\frac{d}{dt}\int  u_{j}\cdot E_{j} \,dx+\int (|E_{j}|^2+\frac{1}{K}|\div E_{j}|^2)\,dx\\
&\quad+\int \Big(  u_{j}\cdot E_{j}+\var(u_{j}\times \bar{B}) \cdot E_{j} - \var u_{j}\cdot (\nabla\times H_{j})- \var^2\bar{\rho}|u_{j}|^2 \Big)\,dx\\
&\leq \var\|\dot{\Delta}_{j}(\var u\cdot\nabla u,u\times H)\|_{L^2}\|E_{j}\|_{L^2}+\var\|\dot{\Delta}_{j}(F(n)u)\|_{L^2}\var\|u_{j}\|_{L^2}\\
&\quad+\frac{1}{K}\|\dot{\Delta}_{j}\Phi(n)\|_{L^2}\|\div E_{j}\|_{L^2}.
\end{aligned}
\end{equation}
On the other hand,  taking the inner product of $\eqref{EM1j}_{3}$ with $-\nabla\times H_{j}$ and using $\eqref{EM1j}_{4}$, we get the dissipation for $H$: 
\begin{equation}\label{E4}
\begin{aligned}
&-\var\frac{d}{dt}\int E_{j}\cdot\nabla \times H_{j} \,dx+\int (|\nabla\times H_{j}|^2 + \var\bar{\rho}u_{j}\cdot \nabla \times H_{j})\,dx  \\&\leq \int|\nabla\times E_{j}|^2dx+ \var \|\dot{\Delta}_{j}(F(n)u)\|_{L^2}\|\nabla \times H_{j} \|_{L^2}.
\end{aligned}
\end{equation}
Let $\eta_{1}\in(0,1)$. We denote by $\mathcal{L}_{\ell,j}$
and $D_{\ell,j}$ the low-frequency energy functional and dissipation functional:
\begin{equation}\nonumber
\begin{aligned}
\mathcal{L}_{\ell,j}(t):&=\frac{1}{2}\int\Big(|n_{j}|^2+ P'(\bar{\rho})\var^2|u_{j}|^2+\frac{1}{K}|E_{j}|^2+\frac{1}{K} |H_{j}|^2\Big)\,dx\\
&\quad+\var^2\eta_{1}\int u_{j}\cdot \nabla n_{j} \,dx+ \var^2\eta_{1} \int  u_{j}\cdot E_{j} \,dx-\eta_{1}^{\frac{5}{4}} \var \int E_{j}\cdot\nabla \times H_{j} \,dx
\end{aligned}
\end{equation}
and
\begin{equation}\nonumber
\begin{aligned}
D_{\ell,j}(t):&=P'(\bar{\rho})\int |u_{j}|^2+\eta_{1}\int \Big( |\nabla n_{j}|^2 +K|n_{j}|^2-P'(\bar{\rho})|\div u_{j}|^2+u_{j}\cdot \nabla n_{j}\Big)\,dx \\
&\quad+\eta_{1}\int \Big(|E_{j}|^2+\frac{1}{K}|\div E_{j}|^2+ u_{j}\cdot E_{j}+\var(u_{j}\times \bar{B}) \cdot E_{j} - \var u_{j}\cdot (\nabla\times H_{j})- \var^2\bar{\rho}|u_{j}|^2 \Big)\,dx\\
&\quad+\eta_{1}^{\frac{5}{4}}  \int (|\nabla\times H_{j}|^2- \var\bar{\rho}u_{j}\cdot \nabla \times H_{j} -|\nabla\times E_{j}|^2 )\,dx.
\end{aligned}
\end{equation}
Combining \eqref{prioribound} with \eqref{E1}-\eqref{E4}, Bernstein's inequality and $j\leq J_{0}$ leads to
\begin{equation}\label{Elow}
\begin{aligned}
&\frac{d}{dt}\mathcal{L}_{\ell,j}(t)+D_{\ell,j}(t)\lesssim  \|G_{j}^{\ell}(t)\|_{L^2}\|(n_{j},\var u_{j},E_{j},H_{j})\|_{L^2},
\end{aligned}
\end{equation}
with
$$
G_{j}^{\ell}(t):=\|\dot{\Delta}_{j}\big(u\cdot\nabla n, G(n)\div u,\var u\cdot\nabla u, u\times H, \var F(n) u,\Phi(n)\big)\|_{L^2}.
$$
Hence, we claim that for $\var\in(0,1]$, there exists a suitable small constant $\eta_{1}>0$ independent of $\var$ such that
\begin{equation}\label{Esimlow}
\left\{
\begin{aligned}
\mathcal{L}_{\ell,j}(t)&\sim \|(n_{j},\var u_{j},E_{j},H_{j})\|_{L^2}^2,\\
D_{\ell,j}(t)&\gtrsim \|(n_{j},u_{j},E_{j})\|_{L^2}^2+ 2^{2j}\|H_{j}\|_{L^2}^2.
\end{aligned}
\right.
\end{equation}
Indeed, it follows from $\supp\:(\widehat{\Delta_{j}\cdot})\subset \{\frac{3}{4} 2^{j}\leq |\xi|\leq \frac{8}{3} 2^{j}\}$ and  $2^{j}\leq 1 $ that
\begin{equation}\nonumber
\begin{aligned}
\mathcal{L}_{\ell,j}(t)&\leq \frac{1}{2}\int\Big((1+\frac{8}{3}\eta_{1})|n_{j}|^2+ (P'(\bar{\rho})+\frac{11}{3}\eta_{1})\var^2|u_{j}|^2+(\frac{1}{K}+\frac{11}{3}\eta_{1})|E_{j}|^2+(\frac{1}{K}+\frac{8}{3}\eta_{1}^{\frac{3}{2}})|H_{j}|^2\Big)\,dx,\\
\mathcal{L}_{\ell,j}(t)&\geq \frac{1}{2}\int\Big((1-\frac{8}{3}\eta_{1})|n_{j}|^2+ (P'(\bar{\rho})-\frac{11}{3}\eta_{1})\var^2|u_{j}|^2+(\frac{1}{K}-\frac{11}{3}\eta_{1})|E_{j}|^2+(\frac{1}{K}-\frac{8}{3}\eta_{1}^{\frac{3}{2}})|H_{j}|^2\Big)\,dx.
\end{aligned}
\end{equation}
Since $\div H_j=0$, the div-curl lemma implies that
\begin{equation}
\begin{aligned}
 \|\nabla \times H_{j}\|_{L^2}^2=\|\nabla H_{j}\|_{L^2}^2\geq \frac{9}{16}2^{2j}\|H_{j}\|_{L^2}^2.
\end{aligned}
\end{equation}
Furthermore, we have
\begin{equation}\nonumber
\begin{aligned}
D_{\ell,j}(t)&\geq P'(\bar{\rho})\int |u_{j}|^2 \,dx+\eta_{1} \int \Big( \frac{1}{2} |\nabla n_{j}|^2 \,dx+K|n_{j}|^2-P'(\bar{\rho})| \div u_{j}|^2-\frac{1}{2}|u_{j}|^2\Big)\,dx \\
&\quad+\eta_{1} \int \Big(\frac{1}{2} |E_{j}|^2-(1+\bar{B}^2+\frac{1}{2\eta_{1}^{\frac{1}{4}}}+\bar{\rho})|u_{j}|^2-\frac{1}{2}\eta_{1}^{\frac{1}{4}} |\nabla\times H_{j}|^2\Big)\,dx\\
&\quad+\eta_{1}^{\frac{5}{4}}  \int (\frac{1}{2}|\nabla\times H_{j}|^2- \frac{\bar{\rho}^2}{2}|u_{j}|^2 -|\nabla\times E_{j}|^2 )\,dx\\
&\geq \int \Big( (P'(\bar{\rho})-\frac{64P'(\bar{\rho})}{9}\eta_{1}-\bar{\rho}\eta_{1}-\frac{\bar{\rho}^2}{2}\eta_{1}^{\frac{3}{4}})|u_{j}|^2+\eta_{1}K |n_{j}|^2)\Big)dx\\
&\quad\quad+\int \Big( \frac{1}{2}\eta_{1}\var(1-\frac{32}{9}\eta_{1}^{\frac{1}{4}})|E_{j}|^2+\frac{9}{32}\eta_{1}^{\frac{5}{4}}  2^{2j}| H_{j}|^2)\,dx.
\end{aligned}
\end{equation}
Taking $\eta_{1}$ sufficiently small yields \eqref{Esimlow} immediately. Together, \eqref{Elow} and  \eqref{Esimlow} yield
\begin{equation}\label{L1}
\begin{aligned}
&\frac{d}{dt}\mathcal{L}^{\ell}_{j}(t)+\|(n_{j},u_{j},E_{j})\|_{L^2}^2+ 2^{2j}  \| H_{j}\|_{L^2}^2\lesssim G_{j}^{\ell}(t) \sqrt{\mathcal{L}^{\ell}_{j}(t)}.
\end{aligned}
\end{equation}
Applying Lemma \ref{lemmaL1L2} to \eqref{L1} and  \eqref{Esimlow} leads to
\begin{equation}\label{nnn}
\begin{aligned}
&\|(n_{j},\var u_{j},E_{j},H_{j})\|_{L^{\infty}_{t}(L^2)}+\|(n_{j},u_{j},E_{j})\|_{L^2_{t}(L^2)}+ 2^{j}\|H_{j}\|_{L^2_{t}(L^2)}\\
&\quad\lesssim \|(n_{j},\var u_{j},E_{j},H_{j})(0)\|_{L^2}+\|G_j^{\ell}\|_{L^1_t(L^2)}.
\end{aligned}
\end{equation}
 Multiplying $\eqref{nnn}$ by the factor $2^{(\frac{d}{2}-1)j}$ and summing over $j\leq 0$, we get
\begin{equation}\label{nuehell}
\begin{aligned}
&\|(n,\var u,E,H)\|_{\widetilde{L}^{\infty}_{t}(\dot{B}^{\frac{1}{2}} )}^{\ell}+\|(n,u,E)\|_{\widetilde{L}^{2}(\dot{B}^{\frac{1}{2}})}^{\ell}+\|H\|_{\widetilde{L}^{2}_{t}(\dot{B}^{\frac{3}{2}})}^{\ell}\\
&\quad\quad\quad\lesssim \|(n_{0}, u_{0},E_{0},H_{0})\|_{\dot{B}^{\frac{1}{2}} }^{\ell}+\|(u\cdot\nabla n, G(n)\div u,\var u\cdot\nabla u, \var u\times H, \var F(n) u,\Phi(n))\|_{L^1_{t}(\dot{B}^{\frac{1}{2}} )}^{\ell}.
\end{aligned}
\end{equation}
Before bounding the nonlinear terms on the right-hand side of \eqref{nuehell}, we claim that the standard Besov norms of $(n,u,E,B)$ can be bounded by $\mathcal{X}(t)$. Indeed,  owing to \eqref{J} and \eqref{bernstein}, one has
\begin{equation}\label{d2}
\left\{
\begin{aligned}
\|u\|_{\widetilde{L}^2_{t}(\dot{B}^{\frac{1}{2}}\cap \dot{B}^{\frac{3}{2}} )}&\lesssim \|u\|_{\widetilde{L}^2_{t}(\dot{B}^{\frac{1}{2}} )}^{\ell}+\|u\|_{\widetilde{L}^2_{t}(\dot{B}^{\frac{3}{2}} )}^{m}+\var \|u\|_{\widetilde{L}^2_{t}(\dot{B}^{\frac{5}{2}} )}^{h},\\
\|n\|_{\widetilde{L}^2_{t}(\dot{B}^{\frac{1}{2}}\cap\dot{B}^{\frac{5}{2}} )}&\lesssim \|n\|_{\widetilde{L}^2_{t}(\dot{B}^{\frac{1}{2}} )}^{\ell}+\|n\|_{\widetilde{L}^2_{t}(\dot{B}^{\frac{5}{2}} )}^{m}+\|n\|_{\widetilde{L}^2_{t}(\dot{B}^{\frac{5}{2}} )}^{h},\\
\|E\|_{\widetilde{L}^2_{t}(\dot{B}^{\frac{1}{2}}\cap\dot{B}^{\frac{3}{2}} )}&\lesssim \|E\|_{\widetilde{L}^2_{t}(\dot{B}^{\frac{1}{2}} )}^{\ell}+\|E\|_{\widetilde{L}^2_{t}(\dot{B}^{\frac{3}{2}} )}^{m}+\|E\|_{\widetilde{L}^2_{t}(\dot{B}^{\frac{3}{2}} )}^{h},\\
\|B\|_{\widetilde{L}^2_{t}(\dot{B}^{\frac{3}{2}} )}&\lesssim \|B\|_{\widetilde{L}^2_{t}(\dot{B}^{\frac{3}{2}} )}^{\ell}+\|B\|_{\widetilde{L}^2_{t}(\dot{B}^{\frac{3}{2}} )}^{m}+\|B\|_{\widetilde{L}^2_{t}(\dot{B}^{\frac{3}{2}} )}^{h}.
\end{aligned}
\right.
\end{equation}
Then, it follows from \eqref{d2} and the product law $\dot{B}^{\frac{1}{2}}\hookrightarrow \dot{B}^{\frac{3}{2}}\times \dot{B}^{\frac{1}{2}}$ in \eqref{uv2} that
\begin{equation}\label{unablanlow}
\begin{aligned}
&\|u\cdot\nabla n\|_{L^1_{t}(\dot{B}^{\frac{1}{2}} )}\lesssim \|u\|_{\widetilde{L}^2_{t}(\dot{B}^{\frac{3}{2}})}\|\nabla n\|_{\widetilde{L}^2_{t}(\dot{B}^{\frac{1}{2}} )}\lesssim \|u\|_{\widetilde{L}^2_{t}(\dot{B}^{\frac{3}{2}} )}\|n\|_{\widetilde{L}^2_{t}(\dot{B}^{\frac{3}{2}} )}\lesssim \mathcal{X}(t)^2.
\end{aligned}
\end{equation}
Similarly, as $0<\var\leq 1$, we get
\begin{equation}\label{unablaulow}
\begin{aligned}
&\var\|u\cdot\nabla u\|_{L^1_{t}(\dot{B}^{\frac{1}{2}} )}+\|u\times H\|_{L^1_{t}(\dot{B}^{\frac{1}{2}} )}\lesssim \|u\|_{\widetilde{L}^2_{t}(\dot{B}^{\frac{1}{2}} )}\|(u,H)\|_{\widetilde{L}^2_{t}(\dot{B}^{\frac{3}{2}} )}\lesssim \mathcal{X}(t)^2.
\end{aligned}
\end{equation}
In accordance with the bound \eqref{prioribound}, the product law \eqref{uv2} and the composition estimate \eqref{F1}, it also holds that
\begin{equation}\label{321}
\begin{aligned}
\var\|F(n) u\|_{L^1_{t}(\dot{B}^{\frac{1}{2}} )}\lesssim \|F(n)\|_{\widetilde{L}^2_{t}(\dot{B}^{\frac{3}{2}} )}\|u\|_{\widetilde{L}^2_{t}(\dot{B}^{\frac{1}{2}} )}\lesssim \|n\|_{\widetilde{L}^2_{t}(\dot{B}^{\frac{3}{2}} )}\|u\|_{\widetilde{L}^2_{t}(\dot{B}^{\frac{1}{2}} )}\lesssim \mathcal{X}(t)^2.
\end{aligned}
\end{equation}
Recall that $\Phi(n)$ is  quadratic with respect to $n$, so it follows from \eqref{prioribound}, Lemma \ref{compositionlp} and the embedding $\dot{B}^{\frac{3}{2}} \hookrightarrow L^{\infty}$ that 
\begin{equation}\label{Philow}
\begin{aligned}
\|\Phi(n)\|_{L^1_{t}(\dot{B}^{\frac{1}{2}} )}^{\ell}&\lesssim \|n\|_{\widetilde{L}^2_{t}(\dot{B}^{\frac{3}{2}} )} ( \|n\|_{\widetilde{L}^2_{t}(\dot{B}^{\frac{1}{2}} )}^{\ell}+\|n\|_{\widetilde{L}^2_{t}(\dot{B}^{\frac{3}{2}} )}^{m}+\var^2\|n\|_{\widetilde{L}^2_{t}(\dot{B}^{\frac{5}{2}} )}^{h})\lesssim \mathcal{X}(t)^2.
\end{aligned}
\end{equation}
Inserting   the above estimates \eqref{unablanlow}-\eqref{Philow} into \eqref{nuehell}, we obtain  \eqref{low}. Hence,  the proof of Lemma \ref{lemmalow} is complete.
\end{proof}

\begin{lemma}[Medium-frequency estimates]\label{lemmamedium}
If $(n,u,E,H)$ is a classical solution to \eqref{EM1} on the time interval $[0,T]$, then the following estimate holds: 
\begin{equation}\label{medium}
\begin{aligned}
&\|(n,\var u,E,H)\|_{\widetilde{L}^{\infty}_{t}(\dot{B}^{\frac{3}{2}} )}^{m}+\|n\|_{\widetilde{L}^{2}_{t}(\dot{B}^{\frac{5}{2}} )}^{m}+\|(u,E,H)\|_{\widetilde{L}^{2}_{t}(\dot{B}^{\frac{3}{2}} )}^{m}\lesssim  \|(n_{0}, u_{0},E_{0},H_{0})\|_{\dot{B}^{\frac{3}{2}} }^{m}+\mathcal{X}(t)^2
\end{aligned}
\end{equation}
for $t\in [0,T]$ and $0<\var\leq 1$.
\end{lemma}

\begin{proof}
As in the proof of Lemma \ref{lemmalow}, we construct a Lyapunov functional to capture the dissipation effects for $(n,u,E,H)$ in medium frequencies. Here, $n$ behaves like heat kernel and the other components are damped. In that case, one cannot treat $\dot{\Delta}_j(G(n)\div u)$ as a source term, since it 
will cause a loss of one derivative with respect to $u$.  To overcome the difficulty, we rewrite $\eqref{EM1j}_{1}$ as
\begin{equation}\label{EM1j1}
\begin{aligned}
&\partial_{t}n_{j}+(P'(\bar{\rho})+G(n))\div u_{j}=\mathcal{R}_{1,j}-\dot{\Delta}_{j}(u\cdot\nabla n),
\end{aligned}
\end{equation}
where 
the commutator is given by $\mathcal{R}_{1,j}:=[G(n),\dot{\Delta}_{j}]\div u$.
Taking the inner product of \eqref{EM1j1} with $n_{j}$, we obtain
\begin{equation}\label{mmm}
\begin{aligned}
&\frac{1}{2}\frac{d}{dt}\|n_{j}\|_{L^2}^2+\int (P'(\bar{\rho})+G(n))\div u_{j} n_{j} \,dx\leq(\|\mathcal{R}_{1,j}\|_{L^2}+\|\dot{\Delta}_{j}(u\cdot\nabla n)\|_{L^2})\|n_{j}\|_{L^2}.
\end{aligned}
\end{equation}
In order to cancel the second term on the left-hand side of \eqref{mmm}, we multiply $\eqref{EM1j}_{2}$ by $(P'(\bar{\rho})+G(n))u_{j}$ and integrate the resulting equality over $\mathbb{R}^{3}$. Performing an integration by parts and using Cauchy-Schwarz inequality implies that   
\begin{equation}\label{mmm1}
\begin{aligned}
&\frac{\var^2}{2}\frac{d}{dt}\int (P'(\bar{\rho})+G(n))|u_{j}|^2\,dx-\int (P'(\bar{\rho}+G(n))\div u_{j} n_{j} \,dx\\
&\quad~+\int\Big( (P'(\bar{\rho})+G(n)) |u_{j}|^2+(P'(\bar{\rho})+G(n))E_{j}\cdot u_{j}\Big)\,dx\\
&~\leq \frac{\var^2}{2}\|\partial_{t}G(n)\|_{L^{\infty}}\|u_{j}\|_{L^2}^2+\|\nabla G(n)\|_{L^{\infty}}\|u_{j}\|_{L^2}\|n_{j}\|_{L^2}\\
&~\quad+(P'(\bar{\rho})+\|G(n)\|_{L^{\infty}})\|\dot{\Delta}_{j}(\var u\cdot\nabla u, u\times H)\|_{L^2}\var \|u_{j}\|_{L^2}.
\end{aligned}
\end{equation}
Combining  \eqref{mmm}-\eqref{mmm1} and \eqref{E13}, we arrive at
\begin{equation}\label{E11f}
\begin{aligned}
&\frac{1}{2}\frac{d}{dt}\int\Big(|n_{j}|^2+ (P'(\bar{\rho})+G(n))\var^2|u_{j}|^2+\frac{1}{K}|E_{j}|^2+\frac{1}{K} |H_{j}|^2\Big)\,dx\\
&\quad~+\int\Big( (P'(\bar{\rho})+G(n)) |u_{j}|^2+G(n)E_{j}\cdot u_{j}\Big)\,dx\\
&~\leq  (\|\mathcal{R}_{1,j}\|_{L^2}+\|\dot{\Delta}_{j}(u\cdot\nabla n)\|_{L^2})\|n_{j}\|_{L^2}+\frac{\var^2}{2}\|\partial_{t}G(n)\|_{L^{\infty}}\|u_{j}\|_{L^2}^2+\|\nabla G(n)\|_{L^{\infty}}\|u_{j}\|_{L^2}\|n_{j}\|_{L^2}\\
&~\quad+(P'(\bar{\rho}+\|G(n)\|_{L^{\infty}})\|\dot{\Delta}_{j}(\var u\cdot\nabla u, u\times h)\|_{L^2}\var \|u_{j}\|_{L^2}+\frac{\var }{K}\|\dot{\Delta}_{j}(F(n)u)\|_{L^2}\|E_{j}\|_{L^2}.
\end{aligned}
\end{equation}
As \eqref{E2}, it follows from  $\eqref{EM1j}_{2}$ and \eqref{EM1j1} that, for $\eta_{2}\in(0,1)$,
\begin{equation}\label{E12f}
\begin{aligned}
&\var^2\frac{d}{dt}\int u_{j}\cdot \nabla n_{j} \,dx+\int \Big( |\nabla n_{j}|^2 +K|n_{j}|^2-(P'(\bar{\rho}+G(n))\var^2|\div u_{j}|^2+u_{j}\cdot \nabla n_{j}\Big)\,dx\\
&\leq \|\dot{\Delta}_{j}(\var u\cdot\nabla u,u\times H)\|_{L^2}\var\|\nabla n_{j}\|_{L^2}+\var\|\nabla\dot{\Delta}_{j}(u\cdot\nabla n)\|_{L^2}\var\|u_{j}\|_{L^2}+\var\|\nabla\mathcal{R}_{1,j}\|_{L^2}\var\| u_{j}\|_{L^2}\\
&\quad+\|\dot{\Delta}_{j}\Phi(n)\|_{L^2}\|n_{j}\|_{L^2}.
\end{aligned}
\end{equation} In view of \eqref{E3}-\eqref{E4} and  \eqref{E11f}-\eqref{E12f}, we denote
\begin{equation}\nonumber
\begin{aligned}
\mathcal{L}_{m,j}(t):&=\frac{1}{2}\int\Big(|n_{j}|^2+ (P'(\bar{\rho})+G(n))\var^2|u_{j}|^2+\frac{1}{K}|E_{j}|^2+\frac{1}{K} |H_{j}|^2\Big)\,dx\\
&\quad+\eta_{2}\var^2\int u_{j}\cdot \nabla n_{j} \,dx+ \eta_{2}\var^2 \int u_{j}\cdot E_{j} \,dx-\eta^{\frac{5}{4}}_{2}\var 2^{-2j}\int E_{j}\cdot\nabla \times H_{j} \,dx
\end{aligned}
\end{equation}
and
\begin{equation}\nonumber
\begin{aligned}
\mathcal{D}_{m,j}(t):&=\int\Big( (P'(\bar{\rho})+G(n)) |u_{j}|^2+G(n)E_{j}\cdot u_{j}\Big)\,dx\\
&\quad+\eta_{2}  \int \Big( |\nabla n_{j}|^2 +K|n_{j}|^2-(P'(\bar{\rho})+G(n))\var^2|\div u_{j}|^2+u_{j}\cdot \nabla n_{j}\Big)\,dx\\
&\quad+\eta_{2} \int \Big(|E_{j}|^2+\frac{1}{K}|\div E_{j}|^2+  u_{j}\cdot E_{j}+\var (u_{j}\times \bar{B}) \cdot E_{j} - \var u_{j}\cdot (\nabla\times H_{j})- \var^2\bar{\rho}|u_{j}|^2 \Big)\,dx\\
&\quad+\eta^{\frac{5}{4}}_{2} 2^{-2j} \int (|\nabla\times H_{j}|^2- \var\bar{\rho}u_{j}\cdot \nabla \times H_{j} -|\nabla\times E_{j}|^2 )\,dx.
\end{aligned}
\end{equation}
Let $\delta_0\leq \frac{P'(\bar{\rho})}{2(1+\|P''\|_{L^{\infty}})}$. It follows from  \eqref{prioribound} that
\begin{align}
\frac{1}{2}P'(\bar{\rho})\leq P'(\bar{\rho})+G(n) \leq \frac{3}{2}P'(\bar{\rho}).\label{nonon}
\end{align}
Furthermore, as \eqref{Esimlow}, it is not difficult to check that   
\begin{equation}\label{EEEE}
\left\{
\begin{aligned}
\mathcal{L}_{m,j}(t)&\sim \|(n_{j},\var u_{j},E_{j},H_{j})\|_{L^2}^2,\\
\mathcal{D}_{m,j}(t)&\gtrsim 2^{2j}\|n_{j}\|_{L^2}^2+\|(u_{j},E_{j},H_{j})\|_{L^2}^2
\end{aligned}
\right.
\end{equation}
for $-1\leq j\leq J_{\var}$, provided that we take the constant $\eta_{2}$ (independent of $\var$) small enough. Therefore, together with
\eqref{prioribound}, \eqref{E3}-\eqref{E4}, \eqref{E11f}-\eqref{E12f} and \eqref{EEEE}, one can get 
the following localized Lyapunov inequality:
\begin{equation}\label{L2}
\begin{aligned}
&\frac{d}{dt}\mathcal{L}^{m}_{j}(t)+ 2^{2j}\|n_{j}\|_{L^2}^2+\|(u_{j},E_{j},H_{j})\|_{L^2}^2\lesssim  G_{j}^{m}(t)\sqrt{\mathcal{L}^{m}_{j}(t)},
\end{aligned}
\end{equation}
with
$$
G_{j}^{m}(t):=\|\dot{\Delta}_{j}(u\cdot\nabla n,\var u\cdot\nabla u,\var F(n)u, \var u\times H,\Psi(n))\|_{L^2}+\|\partial_{t}n\|_{L^{\infty}}\var\|u_{j}\|_{L^2}+\|\mathcal{R}_{1,j}\|_{L^2}.
$$
Then it follows from Lemma \ref{lemmaL1L2} that 
\begin{equation}\nonumber
\begin{aligned}
&\|(n_{j},\var u_{j},E_{j},H_{j})\|_{L^{\infty}_{t}(L^2)}+2^{j}\|n_{j}\|_{L^2_{t}(L^2)}+\|(u_{j},E_{j},H_{j})\|_{L^2_{t}(L^2)}\\
&\quad\lesssim \|(n_{j},\var u_{j},E_{j},H_{j})(0)\|_{L^2}+\|G^{m}_{j}\|_{L^1_{t}(L^2)}
\end{aligned}
\end{equation}
 for $-1\leq j\leq J_{\var}$, 
which implies that
\begin{equation}\label{medium11}
\begin{aligned}
&\|(n,\var u,E,H)\|_{\widetilde{L}^{\infty}_{t}(\dot{B}^{\frac{3}{2}} )}^{m}+\|(u,E,H)\|_{\widetilde{L}^{2}_{t}(\dot{B}^{\frac{3}{2}} )}^{m}\\
&\quad\lesssim \|(n_{0}, u_{0},E_{0},H_{0})\|_{\dot{B}^{\frac{3}{2}} }^{m}+\|(u\cdot\nabla n,\var u\cdot\nabla u,u\times H,\var F(n)u,\Phi(n))\|_{L^1_{t}(\dot{B}^{\frac{3}{2}} )}^{m}\\
&\quad\quad+\var\|\partial_{t}n\|_{L^2_{t}(L^{\infty})}\|u\|_{\widetilde{L}^{2}_{t}(\dot{B}^{\frac{3}{2}} )}^{m}+
\sum_{j\in\mathbb{Z}} 2^{\frac{d}{2}j}\|\mathcal{R}_{1,j}\|_{L^1_{t}(L^2)}.
\end{aligned}
\end{equation}
In what follows, we estimate the  nonlinear terms on the right-hand side of \eqref{medium11}. Similarly to \eqref{d2}, it follows from \eqref{bernstein} that \begin{align}
\var\|u\|_{\widetilde{L}^2_{t}(\dot{B}^{\frac{5}{2}} )}\lesssim \|u\|_{\widetilde{L}^2_{t}(\dot{B}^{\frac{1}{2}} )}^{\ell}+\|u\|_{\widetilde{L}^2_{t}(\dot{B}^{\frac{3}{2}} )}^{m}+\var \|u\|_{\widetilde{L}^2_{t}(\dot{B}^{\frac{5}{2}} )}^{h}.\label{d20}
\end{align}
Hence, by  \eqref{d2}, \eqref{d20} and \eqref{uv2}, we  have
\begin{equation}\label{unablaumed}
\begin{aligned}
&\|(u\cdot\nabla n,  \var u\cdot\nabla u)\|_{L^1_{t}(\dot{B}^{\frac{3}{2}} )}\lesssim \|u\|_{\widetilde{L}^2_{t}(\dot{B}^{\frac{3}{2}} )}(\|n\|_{\widetilde{L}^2_{t}(\dot{B}^{\frac{5}{2}} )}+\var\|u\|_{\widetilde{L}^2_{t}(\dot{B}^{\frac{5}{2}} )})\lesssim \mathcal{X}(t)^2.
\end{aligned}
\end{equation}
Similarly,
\begin{equation}\label{uHmed}
\begin{aligned}
\|u\times H\|_{L^1_{t}(\dot{B}^{\frac{3}{2}} )}\lesssim \|u\|_{\widetilde{L}^2_{t}(\dot{B}^{\frac{3}{2}} )}\|H\|_{\widetilde{L}^2_{t}(\dot{B}^{\frac{3}{2}} )}\lesssim \mathcal{X}(t)^2.
\end{aligned}
\end{equation}
By using \eqref{prioribound}, \eqref{d2}, \eqref{uv2} and \eqref{F1}, we get
\begin{equation}\label{Fuumed}
\begin{aligned}
\|F(n)u\|_{L^1_{t}(\dot{B}^{\frac{3}{2}} )} \lesssim \|u\|_{\widetilde{L}^2_{t}(\dot{B}^{\frac{3}{2}} )}\|n\|_{\widetilde{L}^2_{t}(\dot{B}^{\frac{3}{2}} )}\lesssim \mathcal{X}(t)^2.
\end{aligned}
\end{equation}
In addition, employing  the composition law in Lemma \ref{compositionlp} once again leads to
\begin{equation}\label{Psimed}
\begin{aligned}
\|\Phi(n)\|_{L^1_{t}(\dot{B}^{\frac{3}{2}} )}^{m}&\lesssim  \|n\|_{\widetilde{L}^2_{t}(\dot{B}^{\frac{3}{2}} )}^2\lesssim \mathcal{X}(t)^2.
\end{aligned}
\end{equation}
According to  \eqref{EM1}, \eqref{d2}, \eqref{d20} and $\dot{B}^{\frac{3}{2}}\hookrightarrow L^{\infty}$, it holds that
\begin{equation}\label{partialtnmed}
\begin{aligned}
\var\|\partial_{t}n\|_{L^2_{t}(L^{\infty})}&\lesssim \var\|u\|_{L^{\infty}_{t}(L^{\infty})}\|\nabla n\|_{\widetilde{L}^{2}_{t}(L^{\infty})}+(\bar{\rho}+\|G(n)\|_{L^{\infty}_{t}(L^{\infty})})\var\|\div u\|_{L^2_{t}(L^{\infty})}\\
&\lesssim \|n\|_{\widetilde{L}^2_{t}(\dot{B}^{\frac{5}{2}} )}+\var\|u\|_{\widetilde{L}^2_{t}(\dot{B}^{\frac{5}{2}} )}\lesssim \mathcal{X}(t).
\end{aligned}
\end{equation}
To bound the commutator term
$\mathcal{R}_{1,j}$, using \eqref{prioribound}, \eqref{commutator} and \eqref{F1}, we have
\begin{equation}\label{R1med}
\begin{aligned}
\sum_{j\in\mathbb{Z}} 2^{\frac{d}{2}j}\|\mathcal{R}_{1,j}\|_{L^1_{t}(L^2)}&\lesssim\|G(n)\|_{\widetilde{L}^2_{t}(\dot{B}^{\frac{5}{2}} )}\|\div u\|_{\widetilde{L}^2_{t}(\dot{B}^{\frac{1}{2}} )}\lesssim\|n\|_{\widetilde{L}^2_{t}(\dot{B}^{\frac{3}{2}} )}\|u\|_{\widetilde{L}^2_{t}(\dot{B}^{\frac{3}{2}} )}\lesssim \mathcal{X}(t)^2.
\end{aligned}
\end{equation}
Finally, substituting the above estimates \eqref{unablaumed}-\eqref{R1med} into \eqref{medium11}, we arrive at \eqref{medium}. This completes the proof of Lemma \ref{lemmamedium}.
\end{proof}

\begin{lemma}[High-frequency estimates]\label{lemmahigh}
If $(n,u,E,H)$ is a classical solution to \eqref{EM1} on the time interval $[0,T]$, then the following estimate holds: 
\begin{equation}\label{high}
\begin{aligned}
&\var\|(n,\var u,E,H)\|_{\widetilde{L}^{\infty}_{t}(\dot{B}^{\frac{5}{2}} )}^{h}+\|(n,\var u)\|_{\widetilde{L}^{2}_{t}(\dot{B}^{\frac{5}{2}} )}^{h}+\|(E,H)\|_{\widetilde{L}^{2}_{t}(\dot{B}^{\frac{3}{2}} )}^{h}\\
&\lesssim \var\|(n_{0}, u_{0},E_{0},H_{0})\|_{\dot{B}^{\frac{5}{2}} }^{h}+\mathcal{X}(t)^2+\mathcal{X}(t)^3
\end{aligned}
\end{equation}
for $t\in [0,T]$ and $0<\var\leq 1$.
\end{lemma}

\begin{proof}
As emphasized before, a regularity-loss phenomenon for $E$ and $H$ occurs in the high-frequency regime. This is the main   difference in comparison with recent efforts \cite{CBD3,danchinnoterelaxation} concerning hyperbolic systems with symmetric relaxation. To avoid the loss of one derivative arising from
the nonlinear terms involving the components $(n,u)$, we shall introduce some commutators and rewrite \eqref{EM1j} as
\begin{equation}\label{EM1jj}
\left\{
\begin{aligned}
&\partial_{t}n_{j}+u\cdot\nabla n_{j}+(P'(\bar{\rho})+G(n))\div u_{j}=\mathcal{R}_{1,j}+\mathcal{R}_{2,j},\\
&\var^2\partial_{t}u_{j}+\var^2 u\cdot\nabla u_{j}+\nabla n_{j}+E_{j}+ u_{j}+\var u_{j}\times \bar{B}=-\var\dot{\Delta}_{j}(u\times H)-\var^2\mathcal{R}_{3,j},\\
&\var\partial_{t}E_{j}-\nabla\times H_{j}-\bar{\rho}\var u_{j}=\var\dot{\Delta}_{j}(F(n)u),\\
&\var\partial_{t}H_{j}+\nabla\times E_{j}=0,\\
&\div E_{j}=-Kn_{j}-\dot{\Delta}_{j}\Phi(n),\quad\quad \div H_{j}=0
\end{aligned}
\right.
\end{equation}
with
\[  \mathcal{R}_{1,j}=[G(n),\dot{\Delta}_{j}]\div u, \quad
\mathcal{R}_{2,j}=[u,\dot{\Delta}_{j}]\nabla a\quad \mbox{and}\quad
\mathcal{R}_{3,j}:=[u,\dot{\Delta}_{j}]\nabla n.  \]
Similarly to \eqref{mmm1}-\eqref{E11f}, through a direct computation, we are able to get 
\begin{equation}\label{E12hh}
\begin{aligned}
&\frac{1}{2}\frac{d}{dt}\int\Big(|n_{j}|^2+ (P'(\bar{\rho})+G(n))\var^2 |u_{j}|^2+\frac{1}{K}|E_{j}|^2+\frac{1}{K} |H_{j}|^2\Big)\,dx\\
&\quad~+\int\Big( (P'(\bar{\rho})+G(n)) |u_{j}|^2+G(n)E_{j}\cdot u_{j}+\var G(n)(u_{j}\times \bar{B})\cdot u_{j} \Big)\,dx\\
&~\leq  (P'(\bar{\rho})+\|G(n)\|_{L^{\infty}})\var\|\dot{\Delta}_{j}(u\times H)\|_{L^2}\|u_{j}\|_{L^2}+ \frac{1}{K}\var \|\dot{\Delta}_{j}(F(n)u)\|_{L^2}\|E_{j}\|_{L^2}\\
&\quad~+\frac{1}{2}\|\div u\|_{L^{\infty}}\|n_{j}\|_{L^2}^2+ \frac{1}{2}\Big(P'(\bar{\rho})+\|G(n)\|_{L^{\infty}})\|\div u\|_{L^{\infty}}\var^2\|u_{j}\|_{L^2}^2\\
&\quad~+\|\nabla G(n)\|_{L^{\infty}}\| u\|_{L^{\infty}}\var^2 \|u_{j}\|_{L^2}^2+\frac{\var^2}{2}\|\partial_{t}G(n)\|_{L^{\infty}}\|u_{j}\|_{L^2}^2\\
&\quad~+\Big(P'(\bar{\rho})+\|G(n)\|_{L^{\infty}})\|(\cR_{1,j},\cR_{2,j},\var \cR_{3,j})\|_{L^2}\| (n_{j},\var u_{j})\|_{L^2}.
\end{aligned}
\end{equation}
In order to get the dissipation for $n_j$, we perform the following cross estimate
\begin{equation}\label{cross}
\begin{aligned}
&\var^2\frac{d}{dt}\int u_{j}\cdot \nabla n_{j} \,dx+\int \Big( |\nabla n_{j}|^2 +K|n_{j}|^2-(P'(\bar{\rho}+G(n))\var^2|\div u_{j}|^2+u_{j}\cdot \nabla n_{j}\Big)\,dx\\
&~\leq 2\var^2 \|u\|_{L^{\infty}}\|\nabla u_{j}\|_{L^2}\|\nabla n_{j}\|_{L^2}+\var\|\dot{\Delta}_{j}(u\times H)\|_{L^2}\|\nabla n_{j}\|_{L^2}\\&\hspace{5mm}+\|(\cR_{1,j},\cR_{2,j},\var \cR_{3,j})\|_{L^2}\|\nabla (\var u_{j},n_{j})\|_{L^2}.
\end{aligned}
\end{equation}
Let $\eta_3\in(0,1)$. With aid of \eqref{E3}-\eqref{E4} and \eqref{E12hh}-\eqref{cross}, we denote 
\begin{equation}\nonumber
\begin{aligned}
\mathcal{L}_{h,j}(t):&=\frac{1}{2}\int\Big(|n_{j}|^2+ (P'(\bar{\rho})+G(n))|u_{j}|^2+\frac{1}{K}|E_{j}|^2+\frac{1}{K} |H_{j}|^2\Big)\,dx\\
&\quad+\eta_{3} 2^{-2j}\int u_{j}\cdot \nabla n_{j} \,dx+\eta_{3}2^{-2j} \int u_{j}\cdot E_{j} \,dx-\eta_{3}^{\frac{5}{4}} \frac{1}{\var} 2^{-4j}\int E_{j}\cdot\nabla \times H_{j} \,dx,
\end{aligned}
\end{equation}
and
\begin{equation*}
\begin{aligned}
\mathcal{D}_{h,j}(t):&=\int\Big( (P'(\bar{\rho})+G(n)) |u_{j}|^2+G(n)E_{j}\cdot u_{j}+\var G(n)(u_{j}\times \bar{B})\cdot u_{j} \Big)\,dx\\
&\quad+\eta_{3}\frac{1}{\var^2} 2^{-2j} \int \Big( |\nabla n_{j}|^2 +K|n_{j}|^2-(P'(\bar{\rho}+G(n))\var^2|\div u_{j}|^2+u_{j}\cdot \nabla n_{j}\Big)\,dx\\
&\quad+\eta_{3}\frac{1}{\var^2} 2^{-2j}\int \Big(|E_{j}|^2+\frac{1}{K}|\div E_{j}|^2+  u_{j}\cdot E_{j}+\var (u_{j}\times \bar{B}) \cdot E_{j} - \var u_{j}\cdot (\nabla\times H_{j})- \var^2\bar{\rho}|u_{j}|^2 \Big)\,dx\\
&\quad+\eta_{3}^{\frac{5}{4}} \frac{1}{\var^2} 2^{-4j} \int (|\nabla\times H_{j}|^2- \var\bar{\rho}u_{j}\cdot \nabla \times H_{j} -|\nabla\times E_{j}|^2 )\,dx
\end{aligned}
\end{equation*}
for $j\geq J_{\var}-1$. Recalling \eqref{nonon} and the fact that $2^{-j}\lesssim \var$, one can verify that
\begin{equation}\label{Ehsim}
\left\{
\begin{aligned}
\mathcal{L}_{h,j}(t)&\sim \|(n_{j},\var u_{j},E_{j},H_{j})\|_{L^2}^2,\\
\mathcal{D}_{h,j}(t)&\gtrsim \frac{1}{\var^2}\|n_{j}\|_{L^2}^2+\|u_{j}\|_{L^2}^2 +\frac{1}{\var^2}2^{-2j}  \| (E_{j},H_{j})\|_{L^2}^2,
\end{aligned}
\right.
\end{equation}
if $\eta_3$ is chosen to be small enough. With the help of \eqref{E3}-\eqref{E4}, \eqref{nonon}, \eqref{E12hh}-\eqref{Ehsim}, we obtain for $j\leq J_{\var}-1,$ 
\begin{equation}\label{L3}
\begin{aligned}
&\frac{d}{dt}\mathcal{L}^{h}_{j}(t)+\frac{1}{\var^2}\|n_{j}\|_{L^2}^2+\|u_{j}\|_{L^2}^2+\frac{1}{\var^2}2^{-2j}\|(E_{j},H_{j})\|_{L^2}^2\\
&\lesssim G_{1,j}^{h}(t)\sqrt{\mathcal{L}^{h}_{j}(t)}+G_{1,j}^{h}(t) (\|u_{j}\|_{L^2}+\frac{1}{\var}\|n_{j}\|_{L^2}^2), 
\end{aligned}
\end{equation}
where
\begin{equation}\nonumber
\begin{aligned}
G_{1,j}^{h}(t):&=\|\dot{\Delta}_{j}(\var F(n)u,\Phi(n))\|_{L^2}+(\|\div u\|_{L^{\infty}}+\|\partial_{t}n\|_{L^{\infty}}) \|(n_{j},\var u_{j})\|_{L^2}\\
&\quad+(1+\var\|\nabla n\|_{L^{\infty}})\|u\|_{L^{\infty}}\|u_{j}\|_{L^2}+\|(\mathcal{R}_{1,j},\mathcal{R}_{2,j},\var \mathcal{R}_{3,j})\|_{L^2},\\
G_{2,j}^{h}(t):&=\var\|\dot{\Delta}_{j}(u\times B)\|_{L^2}.
\end{aligned}
\end{equation}
Furthermore, it follows from Lemma \ref{lemmaL1L2} and \eqref{L3} that
\begin{equation}\label{dggg}
\begin{aligned}
& \var\|(n_{j},\var u_{j},E_{j},H_{j})\|_{L^{\infty}_{t}(L^2)}+\|n_{j}\|_{L^2_{t}(L^2)}+ \var\|u_{j}\|_{L^2_{t}(L^2)}+2^{-j}\|(E_{j},H_{j})\|_{L^2_{t}(L^2)}\\
&\quad\lesssim \var\|(n_{j},u_{j},E_{j},H_{j})(0)\|_{L^2}+\var\|G^{h}_{1,j}\|_{L^1_{t}(L^2)}+\var\|G^{h}_{2,j}\|_{L^2_{t}(L^2)}.
\end{aligned}
\end{equation}
Multiplying \eqref{dggg} by $2^{j(\frac{d}{2}+1)}$ and summing the resulting inequality over $j\geq J_{\var}-1$, we get
\begin{equation}\label{fgddgg}
\begin{aligned}
&\var\|(n,\var u,E,H)\|_{\widetilde{L}^{\infty}_{t}(\dot{B}^{\frac{5}{2}} )}^{h}+\|n\|_{\widetilde{L}^{2}_{t}(\dot{B}^{\frac{5}{2}} )}^{h}+\var \|u\|_{\widetilde{L}^{2}_{t}(\dot{B}^{\frac{5}{2}} )}^{h}+\|(E,H)\|_{\widetilde{L}^{2}_{t}(\dot{B}^{\frac{3}{2}} )}^{h}\\
&\quad\lesssim \var\|(n_{0}, u_{0},E_{0},H_{0})\|_{\dot{B}^{\frac{5}{2}} }^{h}+\var\|(F(n)u,\Phi(n))\|_{L^{1}_{t}(\dot{B}^{\frac{5}{2}} )}^{h}\\
&\quad\quad+\var(\|\div u\|_{L^2_{t}(L^{\infty})}+\|\partial_t n\|_{L^2_{t}(L^{\infty})})\|(n,\var u)\|_{\widetilde{L}^{2}_{t}(\dot{B}^{\frac{5}{2}} )}^{h}\\
&\quad\quad+(1+\var\|\nabla n\|_{L^{\infty}_{t}(L^{\infty})})\|u\|_{L^2_{t}(L^{\infty})}\var\|u\|_{\widetilde{L}^{2}_{t}(\dot{B}^{\frac{5}{2}} )}^{h}\\
&\quad\quad+\var\sum_{j\geq J_{\var}-1}2^{(\frac{d}{2}+1)j}\|(\mathcal{R}_{1,j},\cR_{2,j},\cR_{3,j})\|_{L^1_{t}(L^2)}+\var^2\|u\times H\|_{\widetilde{L}^2_{t}(\dot{B}^{\frac{5}{2}} )}^{h}.
\end{aligned}
\end{equation}
It follows from \eqref{uv1} and  \eqref{F1} that
\begin{equation}\nonumber
\begin{aligned}
\var\| F(n)u\|_{L^{1}_{t}(\dot{B}^{\frac{5}{2}} )}^{h}&\lesssim \var\|F(n)\|_{L^2_t(L^{\infty})} \|u\|_{\widetilde{L}^{2}_{t}(\dot{B}^{\frac{5}{2}})}+\var \|F(n)\|_{\widetilde{L}^{2}_{t}(\dot{B}^{\frac{5}{2}})} \|u\|_{L^2_t(L^{\infty})}\\
&\lesssim  \|n\|_{L^2_t(L^{\infty})}\var \|u\|_{\widetilde{L}^{2}_{t}(\dot{B}^{\frac{5}{2}})}+\|n\|_{\widetilde{L}^{2}_{t}(\dot{B}^{\frac{5}{2}})}\|u\|_{\widetilde{L}^{2}_{t}(\dot{B}^{\frac{3}{2}})}.
\end{aligned}
\end{equation}
Noting that \eqref{d2} and \eqref{d20}, we get 
\begin{equation}\nonumber
\begin{aligned}
\var\| F(n)u\|_{L^{1}_{t}(\dot{B}^{\frac{5}{2}} )}^{h}\lesssim \mathcal{X}(t)^2.
\end{aligned}
\end{equation}
As $\Phi(0)=\Phi'(0)=0$, employing \eqref{q2} with $(s,\sigma)=(\frac{5}{2},\frac{3}{2})$ yields
\begin{equation}\nonumber
\begin{aligned}
\var\|\Phi(n)\|_{L^1_{t}(\dot{B}^{\frac{5}{2}} )}^{m}&\lesssim  \|n\|_{\widetilde{L}^2_{t}(\dot{B}^{\frac{3}{2}} )}(\|n\|_{\widetilde{L}^2_{t}(\dot{B}^{\frac{1}{2}} )}^{\ell}+\|n\|_{\widetilde{L}^2_{t}(\dot{B}^{\frac{3}{2}} )}^{m}+\var \|n\|_{\widetilde{L}^2_{t}(\dot{B}^{\frac{5}{2}} )}^{h})\lesssim \mathcal{X}(t)^2.
\end{aligned}
\end{equation}
In addition,  by \eqref{bernstein}, it is easy to see that
\begin{equation}\nonumber
\begin{aligned}
\var\|\nabla n\|_{L^{\infty}_{t}(L^{\infty})}&\lesssim \|n\|_{\widetilde{L}^{\infty}_{t}(\dot{B}^{\frac{1}{2}})}^{\ell}+\|n\|_{\widetilde{L}^{\infty}_{t}(\dot{B}^{\frac{3}{2}})}^{m}+\var\|n\|_{\widetilde{L}^{\infty}_{t}(\dot{B}^{\frac{5}{2}})}^{h}\lesssim \mathcal{X}(t),
\end{aligned}
\end{equation}
and
\begin{equation}\nonumber
\begin{aligned}
\|u\|_{L^2_{t}(L^{\infty})}&\lesssim \|u\|_{\widetilde{L}^{2}_{t}(\dot{B}^{\frac{3}{2}})}\lesssim \mathcal{X}(t).
\end{aligned}
\end{equation}
In view of \eqref{d2}, \eqref{d20}, \eqref{commutator} and  \eqref{F1}, it follows that
\begin{equation}\nonumber
\begin{aligned}
&\var\sum_{j\in\mathbb{Z}}2^{(\frac{d}{2}+1)j}\|(\mathcal{R}_{1,j},\cR_{2,j}, \cR_{3,j})\|_{L^1_{t}(L^2)}\lesssim\|(n,\var u)\|_{\widetilde{L}^2_{t}(\dot{B}^{\frac{5}{2}} )}^2.
\end{aligned}
\end{equation}
Finally, by employing \eqref{uv2} and $\var\leq 1$, we have
\begin{equation}\nonumber
\begin{aligned}
\var^2\|u\times H\|_{\widetilde{L}^2_{t}(\dot{B}^{\frac{5}{2}} )}^{h}
&\lesssim \|u\|_{\widetilde{L}^2_{t}(\dot{B}^{\frac{3}{2}} )}\var \|H\|_{\widetilde{L}^{\infty}_{t}(\dot{B}^{\frac{5}{2}} )}+\var\|u\|_{\widetilde{L}^2_{t}(\dot{B}^{\frac{5}{2}} )} \|H\|_{\widetilde{L}^{\infty}_{t}(\dot{B}^{\frac{3}{2}} )}.
\end{aligned}
\end{equation}
Likewise, one can use \eqref{bernstein} again and deduce that 
\begin{equation}\nonumber
\begin{aligned}
\var \|H\|_{\widetilde{L}^{\infty}_{t}(\dot{B}^{\frac{5}{2}} )}+\|H\|_{\widetilde{L}^{\infty}_{t}(\dot{B}^{\frac{3}{2}} )}&\lesssim \|H\|_{\widetilde{L}^{\infty}_{t}(\dot{B}^{\frac{1}{2}})}^{\ell}+\|H\|_{\widetilde{L}^{\infty}_{t}(\dot{B}^{\frac{3}{2}})}^{m}+\var\|n\|_{\widetilde{L}^{\infty}_{t}(\dot{B}^{\frac{5}{2}})}^{h},
\end{aligned}
\end{equation}
which yields
\begin{equation}\nonumber
\begin{aligned}
\var^2\|u\times H\|_{\widetilde{L}^2_{t}(\dot{B}^{\frac{5}{2}} )}^{h}
&\lesssim \mathcal{X}(t)^2.
\end{aligned}
\end{equation}
Combining \eqref{fgddgg} and the above estimates  gives rise to \eqref{high}. Hence,  the proof of Lemma \ref{lemmahigh} is finished.
\end{proof}

\noindent
\textbf{Proof of Theorem \ref{theorem1}}. In what follows, we give the proof of Theorem \ref{theorem1}. First, we recall a local existence of classical solutions to the Cauchy problem \eqref{EMvar}-\eqref{EMdvar} in the framework of Besov space, which has been shown by prior works \cite{XuEM,XuEMtwo}. 
\begin{proposition} \label{proplocal}
Assume that the initial datum $(\rho_0, u_0,E_0, B_0)$ satisfies $\inf_{x\in \mathbb{R}^3}\rho_0(x)>0$ and $(\rho_0-\bar{\rho}, u_0,E_0, B_0-\bar{B})\in B^{\frac{5}{2}}$. Then, for any fixed $0<\var\leq1$, there exists a maximal time $T_0>0$ such that the Cauchy problem \eqref{EMvar}-\eqref{EMdvar} has a unique classical solution $(\rho,u,E,B)$ satisfying
\begin{equation}\label{localprop}
\begin{aligned}
\inf_{(t,x)\in [0,T_0)\times\mathbb{R}^3}\rho(t,x)>0,\quad (\rho-\bar{\rho},u,E,B-\bar{B})\in \mathcal{C}([0,T_0);B^{\frac{5}{2}})\cap \mathcal{C}^1([0,T_0);B^{\frac{3}{2}}),
\end{aligned}
\end{equation}
where the inhomogeneous Besov space $B^{s}(s>0)$ is defined by the subset of $\mathcal{S}'$ endowed with the norm
\begin{align}
\|\cdot \|_{B^{s}}:=\|\cdot \|_{L^2}+\|\cdot \|_{\dot{B}^{s}}.\nonumber
\end{align}
\end{proposition}

Owing to Proposition \ref{proplocal}, one can  construct a sequence of approximate solutions and show its convergence to the global solution with required regularities. For clarity, we divide the procedure into several steps.

\begin{itemize}
\item\emph{Step 1: Construction of the approximate sequence}
\end{itemize}

Set $(n_{0},u_{0},E_{0},H_{0})$  with $n_{0}=h(\rho_{0})-h(\bar{\rho})$ and $H_{0}=B_{0}-\bar{B}$. Assume that  $(\rho_{0}-\bar{\rho},u_{0},E_{0},B_{0}-\bar{B})$ satisfies \eqref{a1}. For any $k=1,2,...,$ we regularize $(n_{0},u_{0},E_{0},H_{0})$ as follows
\begin{equation}
\begin{aligned}
(n^{k}_{0},u^{k}_{0},E^{k}_{0},H^{k}_{0}):=\sum_{|j'|\leq k} \dot{\Delta}_{j'}(n_{0},u_{0},E_{0},H_{0}).\nonumber
\end{aligned}
\end{equation}
Then, Bernstein's lemma implies that $(n^{k}_{0},u^{k}_{0},E^{k}_{0},H^{k}_{0})\in B^{\frac{5}{2}}$. Furthermore, for suitable large $k$,
$(n^{k}_{0},u^{k}_{0},E^{k}_{0},H^{k}_{0})$ has the uniform bound
\begin{equation}\label{appdata}
\begin{aligned}
&\|(n^{k}_{0},u^{k}_{0},E^{k}_{0},H^{k}_{0})\|_{ \dot{B}^{\frac{1}{2}}}^{\ell}+\|(n^{k}_{0},u^{k}_{0},E^{k}_{0},H^{k}_{0})\|_{\dot{B}^{\frac{3}{2}}}^{m}+\var\|(n^{k}_{0},u^{k}_{0},E^{k}_{0},H^{k}_{0})\|_{\dot{B}^{\frac{5}{2}}}^{h}\leq C_1 \mathcal{E}_0^{\var},
\end{aligned}
\end{equation}
where $C_1$ is a constant independent of $\var$ and $k$, and $\mathcal{E}_0^{\var}$ is given by \eqref{E0}. It suffices to show the above estimate for $n_{0}^{k}$. Indeed,  choosing $k$ large enough such $k\geq J_{\var}+1$, it follows from  Lemma \ref{lemma63} and \eqref{bernstein} that
\begin{equation}\nonumber
\begin{aligned}
\var \|n^{k}_{0}\|_{\dot{B}^{\frac{5}{2}}}^{h}&\lesssim \var \|n_0\|_{\dot{B}^{\frac{5}{2}}}^{h}\lesssim \var\|\rho_0-\bar{\rho}\|_{ \dot{B}^{\frac{5}{2}}}\lesssim \|\rho_0-\bar{\rho}\|_{ \dot{B}^{\frac{1}{2}}}^{\ell}+\|\rho_0-\bar{\rho}\|_{ \dot{B}^{\frac{3}{2}}}^{m}+\var\|\rho_0-\bar{\rho}\|_{ \dot{B}^{\frac{5}{2}}}^{h}.
\end{aligned}
\end{equation}
Similarly,  
\begin{equation}\nonumber
\begin{aligned}
\|n^{k}_{0}\|_{\dot{B}^{\frac{1}{2}}}^{\ell}+\|n^{k}_{0}\|_{\dot{B}^{\frac{3}{2}}}^{m}\lesssim \|n_{0}\|_{\dot{B}^{\frac{1}{2}}\cap \dot{B}^{\frac{3}{2}}}\lesssim \|\rho_0-\bar{\rho}\|_{\dot{B}^{\frac{1}{2}}\cap \dot{B}^{\frac{3}{2}}}\lesssim \|\rho_0-\bar{\rho}\|_{ \dot{B}^{\frac{1}{2}}}^{\ell}+\|\rho_0-\bar{\rho}\|_{ \dot{B}^{\frac{3}{2}}}^{m}+\var\|\rho_0-\bar{\rho}\|_{ \dot{B}^{\frac{5}{2}}}^{h}.
\end{aligned}
\end{equation}
On the other hand, we see  that $(n^{k}_{0},u^{k}_{0},E^{k}_{0},H^{k}_{0})$ converges to $(n_{0},u_{0},E_{0},H_{0})$ strongly as $k\rightarrow \infty$ in the topology associated with $\mathcal{E}_{0}^{\var}$. Actually, \eqref{a1} implies that $\|n_0\|_{\dot{B}^{\frac{1}{2}}}^{\ell}+\var \|n_0\|_{\dot{B}^{\frac{5}{2}}}^{h}<\infty$, so it is not difficult to check that, for $k\geq J_{\var}+1$,
\begin{equation}\nonumber
\begin{aligned}
&\|n^{k}_{0}-n_{0}\|_{\dot{B}^{\frac{1}{2}}}^{\ell}+\|n_0^{k}-n_0\|_{\dot{B}^{\frac{3}{2}}}^{m}+\var\|n_0^{k}-n_0\|_{\dot{B}^{\frac{5}{2}}}^{h}\\
&\quad\quad\quad \lesssim  \sum_{j<-k} 2^{\frac{1}{2}j}\|\dot{\Delta}_jn_{0}\|_{L^2}+\var\sum_{j\geq k} 2^{\frac{5}{2}j}\|\dot{\Delta}_jn_{0}\|_{L^2}\rightarrow 0.
\end{aligned}
\end{equation}
Therefore, according to Proposition \ref{proplocal}, there exists a maximal time $T_k>0$ such that the problem \eqref{EM1} supplemented with the
initial datum $(n^{k}_{0},\frac{1}{\var}u^{k}_{0},E^{k}_{0},H^{k}_{0})$, admits a unique classical solution $(n^{k},u^{k},E^{k},H^{k})$ with $\rho^{k}=\bar{\rho}+Kn^{k}+\Psi(n^k)$ and $B^{k}=H^{k}+\bar{B}$ satisfying \eqref{localprop}. 


\begin{itemize}
\item \emph{Step 2: The continuation argument}
\end{itemize}

Define 
\begin{equation}\label{T*}
\begin{aligned}
&T_{k}^{*}:=\sup\big{\{}t\in[0,T_{k}): ~\mathcal{X}^{k}(t)\leq  2C_{0}C_{1}\mathcal{E}_{0}^{\var}   \big{\}},
\end{aligned}
\end{equation}
where $\mathcal{X}^{k}(t)$ denotes the same functional as $\mathcal{X}(t)$ (see \eqref{Xt}) for $(n^{k},u^{k},E^{k},H^{k})$. Here $T_{k}^{*}$ is well-defined and fulfills $0<T_{k}^{*}\leq T_{k}$. We claim $T_{k}^{*}=T_{k}$. 
Let $\delta_0>0$ be given by Proposition \ref{propapriori}. Due to \eqref{appdata}, (\ref{T*}) and the embedding $\dot{B}^{\frac{3}{2}}\hookrightarrow L^{\infty}$, we choose a generic constant $C_{2}$ such that
\begin{align}
\|n^{k}\|_{L^{\infty}}\leq C_{2} \mathcal{X}^{k}(t)\leq 2C_0C_1C_2 \mathcal{E}_{0}^{\var}\leq \delta_0,\nonumber
\end{align}
provided that
$$
\mathcal{E}_{0}^{\var}\leq \alpha_{0}^{*}:=\frac{\delta_0}{2C_0C_1C_2}.
$$
Therefore, it follows from \eqref{appdata} and \eqref{uniformapriori} in Proposition \ref{propapriori} that
\begin{equation}\nonumber
\begin{aligned}
\mathcal{X}^{k}(t)\leq C_{0}\Big(C_1\mathcal{E}_{0}^{\var}+\mathcal{X}^{k}(t)^2+\mathcal{X}^{k}(t)^{3}\Big),\quad 0< t<T_{k}.
\end{aligned}
\end{equation}
Furthermore, we take $\alpha_{0}$ small enough such that
$$
\mathcal{E}_{0}^{\var}\leq \alpha_{0}:=\min\bigg\{\alpha_0^*,\frac{1}{2C_0C_{1}},\frac{1}{16 C_0^2 C_1}\bigg\}, 
$$
which leads to
\begin{equation}\label{X22}
\begin{aligned}
\mathcal{X}^{k}(t)&\leq C_{0}\Big(C_1\mathcal{E}_{0}^{\var}+2(2C_0C_1\mathcal{E}_0^{\var})^2\Big)\leq \frac{3}{2} C_{0}C_1  \mathcal{E}_{0}^{\var},\quad 0< t<T_{k}.
\end{aligned}
\end{equation}
Thus, the claim follows by using the standard continuity argument. 


Next, we show that $T^*_{k}=+\infty$.
For that end, we use a contradiction argument and assume that $T_{k}^{*}<\infty$. Since $(\rho^k,u^k,E^k,B^k)$ is the classical solution to \eqref{EMvar}, we have
\begin{equation}\nonumber
\begin{aligned}
&\int \Big(\frac{\var^2}{2}\rho^k |u^k|^2+\rho^k \int^{\rho^k}_{\bar{\rho}}\frac{P'(s)-P'(\bar{\rho})}{s^2}\,ds+\frac{1}{2}|E^{k}|^2+\frac{1}{2}|B^{k}-\bar{B}|^2\Big)\,dx+\int_{0}^{t}\int \rho^k |u^k|^2\,dx\\
&\quad=\int \Big(\frac{\var^2}{2}\rho_0^k |u_0^k|^2+\rho_0^k \int^{\rho_0^k}_{\bar{\rho}}\frac{P'(s)-P'(\bar{\rho})}{s^2}\,ds+\frac{1}{2}|E_0^{k}|^2+\frac{1}{2}|B_0^{k}-\bar{B}|^2\Big)\,dx.
\end{aligned}
\end{equation}
The above energy equality gives the $L^2$-norm estimate for $(n^k,u^k,E^k,H^k)$, which is independent of time but depends on $k$. Together with (\ref{X22}), we deduce that $(n^k,u^k,E^k,H^k)\in B^{\frac{5}{2}}$. Hence, let $(n^{k},u^{k},E^{k},H^{k})(t)$ be the new initial datum at some $t$ sufficiently close to $T_{k}^{*}$. 
Applying Proposition \ref{proplocal} once again implies that the existence interval can be extended from $[0,t]$ to $[0,t+\eta^{*}]$ with $t+\eta^{*}>T_{k}^{*}$, which contradicts the definition of $T_{k}^{*}$. Therefore, we conclude that $T_{k}^{*}=\infty$ and $(n^{k},u^{k},E^{k},H^{k})$ is the global-in-time solution to (\ref{EM1}).

\begin{itemize}

\item \emph{Step 3: Compactness and Convergence}
\end{itemize}
 
From the uniform estimate $\mathcal{X}^{k}(t)\lesssim \mathcal{E}_{0}^{\var}$ and (\ref{EM1}), one can deduce that  $(\partial_{t}n^{k},\partial_{t}u^{k},\partial_{t}E^{k}, \partial_{t}H^{k})$ is uniformly bounded. Note that $\dot{B}^{\frac{1}{2},\frac{5}{2}}$ is a Banach space (see Lemma \ref{lemma62}). Thus, by applying the Aubin-Lions lemma and the Cantor diagonal process, there exists a limit $(n,u,E,H)$ such that $(n^{k},u^{k},E^{k},H^{k})$ converges to $(n,u,E,H)$ strongly in $L^2_{loc}(\mathbb{R}_{+};H^{2}_{loc})$, as $k\rightarrow\infty$ (up to a subsequence). Furthermore, the limit $(n,u,E,H)$ solves  (\ref{EM1}) in the sense of distributions. Thanks to Fatou's property $\mathcal{X}(t)\lesssim \liminf\limits_{k\rightarrow\infty} \mathcal{X}^{k}(t)$, we conclude that $\mathcal{X}(t)\lesssim \mathcal{E}_{0}^{\var}$ for all $t>0$. Denote $\rho$ and $B$ by
\begin{align}
&\rho:=\bar{\rho}+K n+\Phi(n),\quad\quad B:=H+\bar{B},\nonumber
\end{align}
where $\Phi(n)$ is given by \eqref{GFPHI}. Consequently, one can show that  $(\rho,u,E,B)$ is the classical solution to the original system (\ref{EMvar})-\eqref{EMdvar} subject to $(\rho_{0},\frac{1}{\var}u_{0},E_{0},B_{0})$. By  standard product laws and composition estimates,  $(\rho,u,E,B)$ satisfies the energy inequality (\ref{r1}). In addition, following a similar argument as in  \cite[Page 196]{BHN}, one has $(\rho-\bar{\rho},u,E,B-\bar{B})\in \mathcal{C}(\mathbb{R}_{+};\dot{B}^{\frac{1}{2},\frac{5}{2}})$. 

\begin{itemize}
\item \emph{Step 4: Uniqueness}
\end{itemize}

For any time $T>0$,
let $(\rho_{1},u_{1},E_{1},H_{1})$ and $(\rho_{2},u_{2},E_{2},H_{2})$ be two solutions of the system \eqref{EMvar} with the same initial data, such that $(\rho_{i}-\bar{\rho},u_{i},E_{i},B_{i}-\bar{B})\in L^{\infty}(0,T;\dot{B}^{\frac{1}{2}}\cap \dot{B}^{\frac{5}{2}})($i=1,2$)$ and $\rho_{-}\leq \rho_{i}\leq \rho_{+}$ for $0<\rho_{-}\leq\rho_{+}$. Without loss of generality, we set $\var=1$. Let
$$
(\delta \rho,\delta u,\delta E,\delta B)=(\rho_{1}-\rho_{2},u_{1}-u_{2},E_{1}-E_{2},B_{1}-B_{2}).
$$
The unknown $(\delta \rho,\delta u,\delta E,\delta B)$ solves the error system
\begin{equation}\label{EMdelta}
\left\{
\begin{aligned}
&\partial_{t}\delta \rho+u_{1}\cdot \nabla \delta \rho+\rho_{1}\div \delta u=\delta F^{1},\\
&\partial_{t}\delta u+u_{1}\cdot \nabla \delta u+M(\rho_{1})\nabla\delta\rho +\delta u+\delta E+ \delta u\times \bar{B}= \delta F^{2},\\
&\partial_{t}\delta E-\nabla\times \delta B-\bar{\rho}\delta u=\delta F^{3},\\
&\partial_{t}\delta B+\nabla\times \delta E=0,\\
&\div \delta E=-\delta \rho,\quad\quad \div \delta B=0,
\end{aligned}
\right.
\end{equation}
with $M(s)=P'(s)/s$ and
\begin{equation}\nonumber
\begin{aligned}
&\delta F^{1}=-\delta u\cdot\nabla\rho_{2}-\delta \rho \div u_{2},\\
&\delta F^{2}=-\delta u\cdot\nabla u_{2}-(M(\rho_{1})-M(\rho_{2}))\nabla \rho_{2}-u_{1}\times \delta B_{2}-\delta u\times (B_{2}-\bar{B}),\\
&\delta F^{3}=\delta \rho u_{1}+(\rho_{2}-\bar{\rho})\delta u.
\end{aligned}
\end{equation}
Applying $\dot{\Delta}_{j}$ to \eqref{EMdelta} leads to
\begin{equation}\label{EMdeltaj}
\left\{
\begin{aligned}
&\partial_{t}\delta \rho_{j}+u_{1}\cdot \nabla \delta \rho_{j}+\rho_{1}\div \delta u_{j}=\delta F^{1}_{j}+\delta R_{1,j}+\delta R_{2,j},\\
&\partial_{t}\delta u_{j}+u_{1}\cdot \nabla \delta u_{j}+M(\rho_{1})\nabla\delta\rho_{j} +\delta u_{j}+\delta E_{j}+ \delta u_{j}\times \bar{B}= \delta F^{2}_{j}+\delta R_{3,j}+\delta R_{4,j},\\
&\partial_{t}\delta E_{j}-\nabla\times \delta B_{j}-\bar{\rho}\delta u_{j}=\delta F^{3}_{j},\\
&\partial_{t}\delta B_{j}+\nabla\times \delta E_{j}=0,\\
&\div \delta E_{j}=-\delta \rho_{j},\quad\quad \div \delta B_{j}=0,
\end{aligned}
\right.
\end{equation}
where commutator terms are defined by
\[  \delta R_{1,j}=[u_{1},\dot{\Delta}_{j}]\nabla\delta\rho, \quad
\delta R_{2,j}=[\rho_{1},\dot{\Delta}_{j}]\nabla\delta u, \quad
\delta R_{3,j}=[u_{1},\dot{\Delta}_{j}]\nabla\delta u\quad\mbox{and}\quad
\delta R_{4,j}=[M(\rho_{1}),\dot{\Delta}_{j}]\nabla\delta\rho.  \]

Direct computations on \eqref{EMdeltaj} give
\begin{equation}\nonumber
\begin{aligned}
&\frac{1}{2}\frac{d}{dt}\int \Big( \frac{1}{\rho_{1}}|\delta \rho_{j}|^2+\frac{1}{M(\rho_{1})}|\delta u_{j}|^2+\frac{1}{P'(\bar{\rho})}|E_{j}|^2+\frac{1}{P'(\bar{\rho})}|B_{j}|^2)\,dx+\int \frac{1}{M(\rho_{1})} |u_{j}|^2 dx\\
&~\leq \frac{1}{2}\Big(\|\partial_{t} \frac{1}{\rho_{1}}\|_{L^{\infty}}+\|\nabla \frac{u_{1}}{\rho_{1}}\|_{L^{\infty}} \Big)\|\delta \rho_{j}\|_{L^2}^2+\frac{1}{2}\Big(\|\partial_{t} \frac{1}{M(\rho_{1})}\|_{L^{\infty}}+\|\nabla \frac{u_{1}}{M(\rho_{1})}\|_{L^{\infty}}\Big)\|\delta u_{j}\|_{L^2}^2\\
&\quad~+ \|\frac{1}{M(\rho_{1})}-\frac{1}{M(\bar{\rho})}\|_{L^{\infty}}\|u_{j}\|_{L^2}\|E_{j}\|_{L^2}+\|\frac{1}{\rho_{1}}\|_{L^{\infty}}\|(\delta F^{1}_{j},\delta R_{1,j},\delta R_{2,j})\|_{L^2}\|\delta\rho_{j}\|_{L^2}\\
&\quad~+\|\frac{1}{M(\rho_{1})}\|_{L^{\infty}}\|(\delta F^{2}_{j},\delta R_{3,j},\delta R_{4,j})\|_{L^2}\|\delta u_{j}\|_{L^2}+\frac{1}{P'(\bar{\rho})}\|\delta F^{3}_{j}\|_{L^2}\|\delta E_{j}\|_{L^2},
\end{aligned}
\end{equation}
which leads to
\begin{equation}\label{1sdg}
\begin{aligned}
&\quad \|(\delta \rho,\delta u,\delta E,\delta B)\|_{\dot{B}^{\frac{3}{2}} }\\
&\lesssim \int_{0}^{T}(1+\|(\partial_{t}\rho_{1},\nabla \rho_{1}, \nabla u_{1})\|_{L^{\infty}})(\|(\delta \rho,\delta u)\|_{\dot{B}^{\frac{3}{2}} }dt\\
&\quad+\int_{0}^{T}\Big(\|(\delta F_{1},\delta F_{2},\delta F_{2})\|_{\dot{B}^{\frac{3}{2}}}+\sum_{j\in\mathbb{Z}}2^{\frac{d}{2}j}\|(\delta R_{1,j},\delta R_{2,j},\delta R_{3,j},\delta R_{4,j})\|_{L^2} \Big)d\tau.
\end{aligned}
\end{equation}
Using the product law \eqref{uv2} and the composition estimates \eqref{F1} and \eqref{F3}, we arrive at
\begin{equation}\label{1sdg1}
\begin{aligned}
\|(\delta F_{1},\delta F_{2},\delta F_{2})\|_{\dot{B}^{\frac{3}{2}}}\lesssim (\|\nabla(\rho_{2},u_{2})\|_{\dot{B}^{\frac{3}{2}}}+\|(\rho_{2}-\bar{\rho},u_{1},B_{2}-\bar{B})\|_{\dot{B}^{\frac{3}{2}}})\|(\delta \rho,\delta u)\|_{\dot{B}^{\frac{3}{2}}}.
\end{aligned}
\end{equation}
It follows from the composition estimate \eqref{commutator} that
\begin{equation}\label{1sdg2}
\begin{aligned}
\sum_{j\in\mathbb{Z}}2^{\frac{d}{2}j}\|(\delta R_{1,j},\delta R_{2,j},\delta R_{3,j},\delta R_{4,j})\|_{L^2}\lesssim \|\nabla(\rho_{1},u_{2})\|_{\dot{B}^{\frac{3}{2}}}\|(\delta\rho,\delta  u)\|_{\dot{B}^{\frac{3}{2}}}.
\end{aligned}
\end{equation}
Inserting \eqref{1sdg1}-\eqref{1sdg2} into \eqref{1sdg} and then taking advantage of Gr\"onwall's inequality leads to $(\rho_{1},u_{1},E_{1},H_{1})=(\rho_{2},u_{2},E_{2},H_{2})$ for $(x,t)\in\mathbb{R}^{d}\times [0,T]$.
Hence, the proof of the uniqueness of Theorem \ref{theorem1} is finished.

\section{Strong relaxation limit for the compressible Euler-Maxwell system}\label{sectionrelaxation}



In this section, we prove 
Theorem \ref{theorem4}. As a preliminary result, we would like to give the global well-posedness for the following drift-diffusion system (\ref{DD}) first 
\begin{equation*}
\left\{
\begin{aligned}
&\partial_{t}\rho^{*}-\Delta P(\rho^{*})-\div (\rho^{*} \nabla \phi^{*})=0,\\
&\Delta \phi^{*}=\bar{\rho}-\rho^{*}.
\end{aligned}
\right.
\end{equation*}

\begin{theorem}\label{theorem2}
There exists a generic constant $\alpha_{1}>0$  such that if
\begin{equation}\label{a2}
\begin{aligned}
&\|\rho^{*}_{0}-\bar{\rho}\|_{ \dot{B}^{\frac{1}{2},\frac{3}{2}} }\leq \alpha_{1},
\end{aligned}
\end{equation}
then the Cauchy problem \eqref{DD} has a unique global solution $\rho^{*}$ fulfilling $\rho^{*}-\bar{\rho}\in \mathcal{C}(\mathbb{R}^{+}; \dot{B}^{\frac{1}{2},\frac{3}{2}} )$ and
\begin{equation}\label{r2}
\begin{aligned}
\|\rho^{*}-\bar{\rho}\|_{\widetilde{L}^{\infty}_{t}(\dot{B}^{\frac{1}{2},\frac{3}{2}})}+\|\rho^{*}-\bar{\rho}\|_{\widetilde{L}^2_{t}( \dot{B}^{\frac{1}{2},\frac{5}{2}} )}
\leq C\|\rho^{*}_{0}-\bar{\rho}\|_{ \dot{B}^{\frac{1}{2},\frac{3}{2}} }.
\end{aligned}
\end{equation}
\end{theorem}

The proof of Theorem \ref{theorem2} can be given by the maximal regularity estimate and the standard fixed point argument (see \cite{c2,Lemarié23}). Here, we feel free to omit the similar details for brevity. Let us mention that the regularity of $\rho^*$ in \eqref{r2} is the exactly same as that of $\rho^{\var}$ in the low-frequency regime $j\leq 0$ and in the medium-frequency regime $-1\leq j\leq J_{\var}$, respectively. We give a little explanation on the choice of  $\dot{B}^{\frac{1}{2},\frac{3}{2}}$ for the initial datum $\rho^{*}_{0}$. Indeed, one can rewrite \eqref{DD} as
\begin{align}
&\partial_{t}\rho^{*}-P'(\bar{\rho})\Delta \rho^{*}+\bar{\rho}\rho^{*}=\div ( (P'(\rho^{*})-P'(\bar{\rho}))\nabla\rho^{*})+\div ((\rho^{*}-\bar{\rho})\nabla(-\Delta)^{-1}\rho^{*}).\label{DD11}
\end{align}
Clearly, there are two dissipation effects in \eqref{DD11}: the heat diffusion and damping. In order to handle the second lower-order term, we need the $\dot{B}^{\frac{1}{2}}$-regularity for low frequencies, and to control the composite function  $P'(\rho^{*})-P'(\bar{\rho})$, the $\dot{B}^{\frac{3}{2}}$-regularity is required for high frequencies owing to the embedding $\dot{B}^{\frac{3}{2}}\hookrightarrow L^{\infty}$. 

\vspace{3mm}
Let $(n^{\var},u^{\var},E^{\var},B^{\var})$, with $n^{\var}=h(\rho^{\var})-h(\bar{\rho})$ and $H^{\var}=B^{\var}-\bar{B}$, be the global solution to \eqref{EMvar}-\eqref{EMdvar} in Theorem \ref{theorem1}. As mentioned in Subsection \ref{subsectionmain}, it is convenient to introduce the effective velocity
$$
z^{\var}:=u^{\var}+ \nabla n^{\var}+ E^{\var}+ \var u^{\var} \times \bar{B},
$$
which plays a key role in justifying the strong relaxation limit from \eqref{EMvar} to \eqref{DD}. Indeed, observe that 
$$
\partial_{t}u^{\var}=- \frac{1}{\var^2}z^{\var}-u^{\var}\cdot \nabla u^{\var}-\frac{1}{\var}u^{\var}\times H^{\var},
$$
in which one can deduce that  $z^{\var}$ satisfies a damping equation with high-order terms
\begin{align}
 \partial_{t} z^{\var}+\frac{1}{\var^2 }z^{\var}+\frac{1}{\var}z^{\var}\times\bar{B}=\nabla\partial_{t}n^{\var}+ \partial_{t}E^{\var}+ F^{\var},\label{zeq}
\end{align}
with
$$
F^{\var}=- u^{\var}\cdot\nabla u^{\var}-\frac{1}{\var}u^{\var}\times H^{\var}-\var (u^{\var}\cdot\nabla u^{\var})\times\bar{B}- (u^{\var}\times H^{\var})\times\bar{B}.
$$
The equation \eqref{zeq} indicates that $z^{\var}$ possesses a better property compared with the velocity $u^{\var}$. We establish the decay estimates of $z^{\var}$ as follows.

\subsection{Regularity estimates of the effective velocity}\label{subsectiondampmode}

\begin{proposition}\label{propenhancez}
Under the assumptions of Theorem \ref{theorem1}, it holds that
\begin{equation}
\begin{aligned}
&\|z_{L}^{\var}\|_{L^1_{t}(\dot{B}^{\frac{1}{2},\frac{3}{2}})}+\|z^{\var}-z_{L}^{\var} \|_{\widetilde{L}^2_{t}(\dot{B}^{\frac{1}{2}} )}\leq C \var \mathcal{E}_{0}^{\var},\label{enhancez}
\end{aligned}
\end{equation}
where initial layer correction $z_{L}^{\var}:=e^{-\frac{t}{\var^2}t}z^{\var}_0$ is the solution to
\begin{align}
&\partial_{t}z^{\var}_{L}+\frac{1}{\var^2}z^{\var}_{L}=0,\quad\quad z^{\var}_{L}|_{t=0}=z^{\var}_0:=\frac{1}{\var}u_0+\nabla h(\rho_0)+E_0+u_0\times \bar{B},
\end{align}
and $C>0$ is a constant independent of $\var$.
\end{proposition}

\begin{remark}
If we aim to establish the convergence rate of $z^{\var}$ in $\widetilde{L}^2_{t}(\dot{B}^{\frac{1}{2}})$ directly, then one has to require the well-prepared condition $\|u_0\|_{\dot{B}^{\frac{1}{2}}}=\mathcal{O}(\var)$. Indeed, we have
$$
\|z_{L}^{\var}\|_{\widetilde{L}^2_{t}(\dot{B}^{\frac{1}{2}})}\lesssim \|u_0\|_{\dot{B}^{\frac{1}{2}}}+\var\|(\rho_0-\bar{\rho},E_0)\|_{\dot{B}^{\frac{1}{2}}}.
$$
\end{remark}

\begin{proof}
We first deal with the initial layer correction  $z^{\var}_{L}=e^{-\frac{t}{\var^2}}z_{0}^{\var}$. According to the definition of $z_0^{\var}$, we have
\begin{equation}\label{47}
\begin{aligned}
\|z_{L}^{\var}\|_{L^1_{t}(\dot{B}^{\frac{1}{2},\frac{3}{2}})}&=\int_0^{t}e^{-\frac{\tau}{\var^2}} \,d\tau\, \|z_0^{\var}\|_{\dot{B}^{\frac{1}{2},\frac{3}{2}}}\\
&\leq \var^2 \Big(\Big\| \frac{1}{\var} u_0^{\var}\Big\|_{\dot{B}^{\frac{1}{2},\frac{3}{2}}}+\|\nabla h(\rho^{\var})\|_{\dot{B}^{\frac{1}{2},\frac{3}{2}}}+\|E_0\|_{\dot{B}^{\frac{1}{2},\frac{3}{2}}}+\|u_0^{\var}\times \bar{B}\|_{\dot{B}^{\frac{1}{2},\frac{3}{2}}} \Big)\lesssim \var \mathcal{E}_0^{\var}.
\end{aligned}
\end{equation}
Denote
$$
\widetilde{z}^{\var}:=z^{\var}-z_{L}^{\var},
$$
which solves
\begin{align}
 \partial_{t} \widetilde{z}^{\var}+\frac{1}{\var^2 }\widetilde{z}^{\var}+\frac{1}{\var}\widetilde{z}^{\var}\times\bar{B}=\frac{1}{\var}z_{L}^{\var}\times\bar{B}+\nabla\partial_{t}n^{\var}+ \partial_{t}E^{\var}+ F^{\var},\quad \widetilde{z}^{\var}|_{t=0}=0.\label{zj}
\end{align}
Applying $\dot{\Delta}_j$ to \eqref{zj}, taking the $L^2$ inner product of the resulting equation with $z_j$ and noticing that $(\widetilde{z}_j^{\var}\times\bar{B})\cdot \widetilde{z}_j=0$, yields 
\begin{equation}\nonumber
\begin{aligned}
&\frac{1}{2}\frac{d}{dt}\|\widetilde{z}^{\var}_j\|_{L^2}^2+\frac{1}{\var^2}\|\widetilde{z}^{\var}_j\|_{L^2}^2\\
&\leq (\var^{-1} \|(z_{L}^{\var})_{j}\times\bar{B}\|_{L^2}+\|\nabla\partial_{t}n_j^{\var}\|_{L^2}+ \|\partial_{t}E_j^{\var}\|_{L^2}+\|F_j^{\var}\|_{L^2})\|z^{\var}_j\|_{L^2}\\
&\leq \frac{1}{2\var^2}\|z^{\var}_j\|_{L^2}^2+2\var^2(\|\nabla\partial_{t}n_j^{\var}\|_{L^2}^2+ \|\partial_{t}E_j^{\var}\|_{L^2}^2+\|F_j^{\var}\|_{L^2}^2)+\var^{-1} \|(z_{L}^{\var})_{j}\times\bar{B}\|_{L^2}\|z^{\var}_j\|_{L^2},
\end{aligned}
\end{equation}
from which we infer that
\begin{equation}\nonumber
\begin{aligned}
&\|\widetilde{z}^{\var}_j\|_{L^{\infty}_{t}(L^2)}+\frac{1}{\var}\|\widetilde{z}^{\var}_j\|_{L^1_t(L^2)}\\
&\quad\lesssim \var\|\nabla\partial_{t}n_j^{\var}\|_{L^2_t(L^2)}+\var \|\partial_{t}E_j^{\var}\|_{L^2_t(L^2)}+\var \|F_j^{\var}\|_{L^2_t(L^2)}+\var^{-\frac{1}{2}} \|(z_{L}^{\var})_{j}\times\bar{B}\|_{L^1_{t}(L^2)}^{\frac{1}{2}}\|z^{\var}_j\|_{L^{\infty}_{t}(L^2)}^{\frac{1}{2}}.
\end{aligned}
\end{equation}
Therefore, summing the resulting inequality over $j\in\mathbb{Z}$ with the factor $2^{(\frac{d}{2}-1)j}$ after taking advantage of Young's inequality for the last term, we obtain
\begin{equation}\label{qwer}
\begin{aligned}
&\|\widetilde{z}^{\var}\|_{\widetilde{L}^{\infty}_{t}(\dot{B}^{\frac{1}{2}} )}+\frac{1}{\var}\|\widetilde{z}^{\var}\|_{\widetilde{L}^{2}(\dot{B}^{\frac{1}{2}} )}\\
&\quad\lesssim \var\|(\partial_{t}\nabla n^{\var},\partial_{t} E^{\var})\|_{\widetilde{L}^{2}(\dot{B}^{\frac{1}{2}} )}+\var\|F^{\var}\|_{\widetilde{L}^{2}(\dot{B}^{\frac{1}{2}} )}+\var^{-1}\|z_{L}^{\var}\|_{L^1_{t}(\dot{B}^{\frac{1}{2}})}.
\end{aligned}
\end{equation}
It follows from $\eqref{EM1}_{1}$, \eqref{uv2} and \eqref{F1} that
\begin{equation}\nonumber
\begin{aligned}
\var\|\partial_{t}\nabla n^{\var}\|_{\widetilde{L}^{2}(\dot{B}^{\frac{1}{2}} )}&\lesssim \var \|u\|_{\widetilde{L}^{2}(\dot{B}^{\frac{5}{2}} )}+\var \|u\cdot\nabla n\|_{\widetilde{L}^{2}(\dot{B}^{\frac{3}{2}} )}+\var\|G(n)\div u\|_{\widetilde{L}^{2}(\dot{B}^{\frac{3}{2}} )}\\
&\lesssim (1+\|n\|_{\widetilde{L}^{\infty}(\dot{B}^{\frac{3}{2}} )}) \var \|u\|_{\widetilde{L}^{2}(\dot{B}^{\frac{5}{2}} )}+ \|u\|_{\widetilde{L}^{2}(\dot{B}^{\frac{3}{2}} )}\|n\|_{\widetilde{L}^{2}(\dot{B}^{\frac{5}{2}})}.
\end{aligned}
\end{equation}
Together with  \eqref{r1}, \eqref{d2} and \eqref{d20}, we arrive at
\begin{equation}\nonumber
\begin{aligned}
\var\|\partial_{t}\nabla n^{\var}\|_{\widetilde{L}^{2}(\dot{B}^{\frac{1}{2}} )}\lesssim (1+\mathcal{E}_0^{\var})\mathcal{E}_0^{\var}.
\end{aligned}
\end{equation}
Hence, it follows from $\eqref{EM1}_{3}$, \eqref{r1}, \eqref{d2} and \eqref{uv2} that
\begin{equation}\nonumber
\begin{aligned}
\var\|\partial_{t}E^{\var}\|_{\widetilde{L}^{2}(\dot{B}^{\frac{1}{2}} )}&\lesssim  \| H^{\var}\|_{\widetilde{L}^{2}(\dot{B}^{\frac{3}{2}} )}+\|u^{\var}\|_{\widetilde{L}^{2}(\dot{B}^{\frac{1}{2}} )}(1+\|H^{\var}\|_{\widetilde{L}^{\infty}(\dot{B}^{\frac{3}{2}} )})\lesssim    (1+\mathcal{E}_0^{\var})\mathcal{E}_0^{\var}.
\end{aligned}
\end{equation}
Similarly, it holds that
\begin{equation}\nonumber
\begin{aligned}
\var\|F^{\var}\|_{\widetilde{L}^{2}(\dot{B}^{\frac{1}{2}} )}&\lesssim \var \|u^{\var}\cdot\nabla u^{\var}\|_{\widetilde{L}^{2}(\dot{B}^{\frac{1}{2}} )}+\|u^{\var}\times H^{\var}\|_{\widetilde{L}^{2}(\dot{B}^{\frac{1}{2}} )}\\
&\lesssim \|u^{\var}\|_{\widetilde{L}^{2}(\dot{B}^{\frac{1}{2}} )} (\|u^{\var}\|_{\widetilde{L}^{2}(\dot{B}^{\frac{3}{2}} )}+\|H^{\var}\|_{\widetilde{L}^{2}(\dot{B}^{\frac{3}{2}} )})\lesssim (\mathcal{E}_0^{\var})^2.
\end{aligned}
\end{equation}
Substituting the above estimates into \eqref{qwer} and using \eqref{47}, we end up with
$$
\|\widetilde{z}^{\var}\|_{\widetilde{L}^{2}(\dot{B}^{\frac{1}{2}} )}\leq C \var\mathcal{E}_0^{\var}.
$$
This completes the proof of Proposition \ref{propenhancez}.
\end{proof}

\subsection{Proof of Theorem \ref{theorem4}}\label{subsectionill}
Let $(\rho^{\var},u^{\var},E^{\var},B^{\var})$ and $\rho^{*}$ be the solutions to \eqref{EMvar}-\eqref{EMdvar} and \eqref{DD} from Theorems \ref{theorem1} and \ref{theorem2} associated with the initial data $(\rho^{\var}_{0},u^{\var}_{0},E^{\var}_{0},B^{\var}_{0})$ and $\rho^{*}_{0}$, respectively. Denote $E^{*}=\nabla(-\Delta)^{-1}(\rho^{*}-\bar{\rho})$ and $B^{*}=\bar{B}$. Now we begin with the proof of Theorem \ref{theorem4}. To that matter, we define the error unknowns as
\begin{equation}
(\delta \rho, \delta u,\delta E,\delta B):=(\rho^{\var}-\rho^{*},u^{\var}-u^{*},E^{\var} -E^{*},B^{\var} -B^{*}).\nonumber
\end{equation}
We will split the proof into two steps.

\begin{itemize}
\item \textbf{Step 1: Convergence estimates for the Euler part $(\delta \rho,\delta u)$}.
\end{itemize}

Recall that the effective velocity $z_{\var}$ is  given by \eqref{dampmode}, and the initial layer correction $z_{L}^{\var}$ is given by Proposition \eqref{propenhancez}. 
Substituting $$
u^{\var}=z_{L}^{\var}+\widetilde{z}^{\var} -\nabla h(\rho^{\var})-E^{\var}-\var u^{\var}\times \bar{B}
$$
into $\eqref{EMvar}_2$, we have
\begin{equation}\label{EMvar12}
\begin{aligned}
&\partial_{t}\rho^{\var}-P'(\bar{\rho})\Delta\rho^{\var}+\bar{\rho} \rho^{\var}\\
&\quad =\div \Big(-\rho^{\var} z_{L}^{\var}+\rho^{\var} \widetilde{z}^{\var}+\var \rho^{\var} u^{\var}\times \bar{B} + (P'(\rho^{\var})-P'(\bar{\rho}))\nabla \rho^{\var}+(\rho^{\var}-\bar{\rho}) E^{\var} \Big),
\end{aligned}
\end{equation}
where $h(\rho)$ is the enthalpy defined by \eqref{enth1}. According to \eqref{EMlimit} and \eqref{EMvar12}, the equation of $\delta \rho$ reads
\begin{equation}\label{delta1}
\begin{aligned}
&\partial_{t}\delta\rho-P'(\bar{\rho})\Delta\delta\rho+\bar{\rho}\delta\rho=F_{1}^{\var}+F_{2}^{\var},\\
\end{aligned}
\end{equation}
where  
\begin{equation}\nonumber
\begin{aligned}
F_{1}^{\var}:&=-\div (\rho^{\var} z_{L}^{\var}) \quad \mbox{and} \quad F_{2}^{\var}:=\div (-\rho^{\var} \widetilde{z}^{\var}+\var\rho^{\var} u^{\var}\times \bar{B} +\delta F)
\end{aligned}
\end{equation}\label{deltaF}
with
\[ \delta F=(P'(\rho^{\var})-P'(\rho^{*}))\nabla\rho^{\var}+(P'(\rho^{*})-P'(\bar{\rho}))\nabla\delta\rho+\delta\rho E^{\var}+(\rho^{*}-\bar{\rho})\delta E. \]

By applying Lemma \ref{maximalL2} to $\eqref{delta1}_{1}$, we obtain
\begin{equation}\label{sfggggggg}
\begin{aligned}
&\|\delta \rho\|_{\widetilde{L}^{\infty}_{t}(\dot{B}^{\frac{1}{2}} )}+\|\delta \rho\|_{\widetilde{L}^2_{t}(\dot{B}^{\frac{1}{2}} )}+\|\delta \rho\|_{\widetilde{L}^2_{t}(\dot{B}^{\frac{3}{2}} )}\\
&\quad\lesssim\|\rho^{\var}_{0}-\rho_{0}^{*}\|_{\dot{B}^{\frac{1}{2}} }+\|F_1^{\var}\|_{L^1_{t}(\dot{B}^{\frac{1}{2}} )}+\|F_2^{\var}\|_{\widetilde{L}^2_{t}(\dot{B}^{-\frac{1}{2}} )}.
\end{aligned}
\end{equation}
Employing the decay estimate of $z_{L}^{\var}$ in \eqref{enhancez}, together with the uniform bound \eqref{r1},  \eqref{uv2} and \eqref{F3}  leads to
\begin{equation}\label{mmmm1}
\begin{aligned}
\|F_1^{\var}\|_{L^1_{t}(\dot{B}^{\frac{1}{2}} )}&\lesssim \|\rho^{\var}z^{\var}_{L}\|_{L^1_{t}(\dot{B}^{\frac{3}{2}} )}\lesssim (1+\|\rho^{\var}-\bar{\rho}\|_{\widetilde{L}^{\infty}_{t}(\dot{B}^{\frac{3}{2}})})\|z^{\var}_{L}\|_{L^1_{t}(\dot{B}^{\frac{3}{2}} )}\lesssim (1+\alpha_{0})\alpha_{0}\var.
\end{aligned}
\end{equation}
Regarding $F^{\var}_2$, we have
\begin{equation}\label{non}
\begin{aligned}
&\|F_2^{\var}\|_{\widetilde{L}^2_{t}(\dot{B}^{-\frac{1}{2}} )}\lesssim \|\rho^{\var} \widetilde{z}^{\var} \|_{\widetilde{L}^2_{t}(\dot{B}^{\frac{1}{2}} )}+\var\|\rho^{\var}u^{\var}\times\bar{B}\|_{\widetilde{L}^2_{t}(\dot{B}^{\frac{1}{2}} )}+\| \delta F\|_{\widetilde{L}^2_{t}(\dot{B}^{\frac{1}{2}} )}.
\end{aligned}
\end{equation}
The nonlinear terms on the right-hand side of \eqref{non} can be estimated as follows. It follows from \eqref{r1} and \eqref{enhancez} that
\begin{equation}\label{sfggggggg311}
\begin{aligned}
\|\rho^{\var} \widetilde{z}^{\var} \|_{\widetilde{L}^2_{t}(\dot{B}^{\frac{1}{2}} )}\lesssim (\bar{\rho}+\|\rho^{\var}-\bar{\rho}\|_{\widetilde{L}^{\infty}_{t}(\dot{B}^{\frac{3}{2}})})\|\widetilde{z}^{\var} \|_{\widetilde{L}^2_{t}(\dot{B}^{\frac{1}{2}} )}\lesssim (1+\alpha_{0})\alpha_{0}\var,
\end{aligned}
\end{equation}
and
\begin{equation}\label{sfggggggg312}
\begin{aligned}
\var\|\rho^{\var}u^{\var}\times\bar{B}\|_{\widetilde{L}^2_{t}(\dot{B}^{\frac{1}{2}} )}\lesssim (1+\|\rho^{\var}-\bar{\rho}\|_{\widetilde{L}^{\infty}_{t}(\dot{B}^{\frac{3}{2}})})\|u^{\var}\|_{\widetilde{L}^2_{t}(\dot{B}^{\frac{1}{2}} )}\lesssim \alpha_{0}\var.
\end{aligned}
\end{equation}
Moreover, we have 
\begin{equation}\nonumber
\begin{aligned}
\| \delta F\|_{\widetilde{L}^2_{t}(\dot{B}^{\frac{1}{2}} )}&\lesssim \|(P'(\rho^{\var})-P'(\rho^{*}))\nabla \rho^{\var}\|_{\widetilde{L}^2_{t}(\dot{B}^{\frac{1}{2}} )}+\|(P'(\rho^{*})-P'(\bar{\rho}))\nabla\delta\rho\|_{\widetilde{L}^2_{t}(\dot{B}^{\frac{1}{2}} )}\\
&\quad+\|\delta\rho E^{\var}\|_{\widetilde{L}^2_{t}(\dot{B}^{\frac{1}{2}} )} +\|(\rho^{*}-\bar{\rho})\delta E \|_{\widetilde{L}^2_{t}(\dot{B}^{\frac{1}{2}} )}.
\end{aligned}
\end{equation}
It follows from \eqref{uv2} and \eqref{F3} that
\begin{equation}\nonumber
\begin{aligned}
\|(P'(\rho^{\var})-P'(\rho^{*}))\nabla \rho^{\var}\|_{\widetilde{L}^2_{t}(\dot{B}^{\frac{1}{2}} )}&\lesssim \|P'(\rho^{\var})-P'(\rho^{*})\|_{\widetilde{L}^2_{t}(\dot{B}^{\frac{3}{2}})}\|\rho^{\var}-\bar{\rho}\|_{\widetilde{L}^2_{t}(\dot{B}^{\frac{3}{2}} )}\lesssim \alpha_0 \|\delta \rho\|_{\widetilde{L}^2_{t}(\dot{B}^{\frac{3}{2}})}.
\end{aligned}
\end{equation}
Similarly, 
\begin{equation}\nonumber
\begin{aligned}
\|(P'(\rho^{*})-P'(\bar{\rho}))\nabla\delta\rho\|_{\widetilde{L}^2_{t}(\dot{B}^{\frac{1}{2}} )}&\lesssim \|P'(\rho^{*})-P'(\bar{\rho})\|_{\widetilde{L}^2_{t}(\dot{B}^{\frac{3}{2}})}\|\delta\rho\|_{\widetilde{L}^2_{t}(\dot{B}^{\frac{3}{2}})}\lesssim \alpha_1\|\delta \rho\|_{\widetilde{L}^2_{t}(\dot{B}^{\frac{3}{2}})}
\end{aligned}
\end{equation}
and
\begin{equation}\nonumber
\begin{aligned}
&\|\delta\rho E^{\var}\|_{\widetilde{L}^2_{t}(\dot{B}^{\frac{1}{2}} )} +\|(\rho^{*}-\bar{\rho})\delta E \|_{\widetilde{L}^2_{t}(\dot{B}^{\frac{1}{2}} )}\lesssim \|\delta\rho\|_{\widetilde{L}^2_{t}(\dot{B}^{\frac{3}{2}} )}\|E^{\var}\|_{\widetilde{L}^{\infty}_{t}(\dot{B}^{\frac{1}{2}} )}+\|\rho^{*}-\bar{\rho}\|_{\widetilde{L}^{\infty}_{t}(\dot{B}^{\frac{3}{2}} )}\|\delta E\|_{\widetilde{L}^2_{t}(\dot{B}^{\frac{1}{2}} )}.
\end{aligned}
\end{equation}
Gathering \eqref{r1} and \eqref{r2}, we get
\begin{equation}\label{sfggggggg3}
\begin{aligned}
\| \delta F\|_{\widetilde{L}^2_{t}(\dot{B}^{\frac{1}{2}} )}&\lesssim (\alpha_0+\alpha_1)(\|\delta\rho\|_{\widetilde{L}^2_{t}(\dot{B}^{\frac{1}{2}} )}+\|\delta \rho\|_{\widetilde{L}^2_{t}(\dot{B}^{\frac{3}{2}})}+\|\delta E\|_{\widetilde{L}^2_{t}(\dot{B}^{\frac{1}{2}} )}).
\end{aligned}
\end{equation}
Putting the above estimates \eqref{mmmm1}-\eqref{sfggggggg3} and \eqref{sfggggggg} together, we arrive at
\begin{equation}\label{thm41}
\begin{aligned}
&\|\delta \rho\|_{\widetilde{L}^{\infty}_{t}(\dot{B}^{\frac{1}{2}} )}+\|\delta \rho\|_{\widetilde{L}^2_{t}(\dot{B}^{\frac{1}{2},\frac{3}{2}} )}\\
&\quad\lesssim \|\rho_0^{\var}-\rho_0^*\|_{\dot{B}^{\frac{1}{2}}}+\var +(\alpha_{0}+\alpha_{1})(\|\delta \rho\|_{\widetilde{L}^2_{t}(\dot{B}^{\frac{1}{2},\frac{3}{2}} )}+\|\delta E\|_{\widetilde{L}^{2}_{t}(\dot{B}^{\frac{3}{2}} )}).
\end{aligned}
\end{equation}


Next, we turn to bound $\delta u$. Keep in mind that $u^{\var}_{L}=e^{-\frac{t}{\var^2}}\frac{1}{\var}u_0$. The variable $\delta u-u^{\var}_{L}$ can be written in the form of 
\begin{align}
&\delta u-u^{\var}_{L}=z^{\var}_{L}-u^{\var}_{L}+\widetilde{z}^{\var} -\nabla(h(\rho^{\var})-h(\rho^{*}))-\delta E
\end{align}
which implies that
\begin{equation}\nonumber
\begin{aligned}
\|\delta u-u^{\var}_{L}\|_{\widetilde{L}^2_{t}(\dot{B}^{\frac{1}{2}})}&\lesssim \|z^{\var}_{L}-u^{\var}_{L} \|_{\widetilde{L}^2_{t}(\dot{B}^{\frac{1}{2}} )}+\|\widetilde{z}^{\var} \|_{\widetilde{L}^2_{t}(\dot{B}^{\frac{1}{2}} )}+\|h(\rho^{\var})-h(\rho^{*})\|_{\widetilde{L}^2_{t}(\dot{B}^{\frac{3}{2}} )}+\|\delta E\|_{\widetilde{L}^2_{t}(\dot{B}^{\frac{1}{2}} )}.
\end{aligned}
\end{equation}
The first term can be estimated by
\begin{equation}\nonumber
\begin{aligned}
\|z^{\var}_{L}-u^{\var}_{L} \|_{\widetilde{L}^2_{t}(\dot{B}^{\frac{1}{2}} )}\leq \Big(\int_{0}^{t}e^{-\frac{2\tau}{\var^2}}\,d\tau\Big)^{\frac{1}{2}} (\|n(\rho_0)\|_{\dot{B}^{\frac{1}{2}}}+\|E_0\|_{\dot{B}^{\frac{1}{2}}}+ \|u_0\times\bar{B}\|_{\dot{B}^{\frac{1}{2}}})\leq C \alpha_{0} \var.
\end{aligned}
\end{equation}
In view of \eqref{F3}, we get
\begin{equation}\nonumber
\begin{aligned}
&\|h(\rho^{\var})-h(\rho^{*})\|_{\widetilde{L}^2_{t}(\dot{B}^{\frac{3}{2}} )}\lesssim \|\delta \rho\|_{\widetilde{L}^2_{t}(\dot{B}^{\frac{3}{2}} )}.
\end{aligned}
\end{equation}
Thus, together with \eqref{enhancez}, it holds that
\begin{equation}\label{thm42}
\begin{aligned}
\|\delta u-u^{\var}_{L}\|_{\widetilde{L}^2_{t}(\dot{B}^{\frac{1}{2}})}&\lesssim   \var+\|\delta\rho\|_{\widetilde{L}^2_{t}(\dot{B}^{\frac{3}{2}} )}+\|\delta E\|_{\widetilde{L}^2_{t}(\dot{B}^{\frac{1}{2}} )}.
\end{aligned}
\end{equation}

\begin{itemize}
\item \textbf{Step 2: Convergence estimates for the Maxwell part $(\delta E,\delta H)$}.
\end{itemize}
Note that
\[  u^{\var}=z_{L}^{\var}+\widetilde{z}^{\var}
-\nabla h(\rho^{\var})-E^{\var}-\var u^{\var}\times\bar{B}.  \]
We rewrite $\eqref{EMvar}_3$-$\eqref{EMvar}_4$ as follows
\begin{equation}\label{EMvar120}
\left\{
\begin{aligned}
&\partial_{t}E^{\var} -\frac{1}{\var}\nabla\times B^{\var}+ \rho^{\var} E^{\var} =\rho^{\var}(z_{L}^{\var}+\widetilde{z}^{\var})-\var u^{\var}\times \bar{B}-\nabla P(\rho^{\var}),\\
&\partial_{t}B^{\var} +\frac{1}{\var}\nabla\times E^{\var}=0,\\
&\div E^{\var} =\bar{\rho}-\rho^{\var},\quad\quad \div B^{\var} =0.
\end{aligned}
\right.
\end{equation}
Due to $E^{*}=\nabla(-\Delta)^{-1}(\rho^{*}-\bar{\rho})$, Darcy's law \eqref{darcy} and the fact that $\nabla\div =\nabla\times\nabla\times+\Delta$, one has
\begin{equation}\nonumber
\begin{aligned}
 \partial_{t}E^{*}&=-\nabla(-\Delta)^{-1}\div(\rho^{*}u^{*})=\rho^{*}u^{*}+\nabla\times B^{1,*}=-\rho^{*}E^{*}-\nabla P(\rho^{*})+\nabla\times B^{1,*},
\end{aligned}
\end{equation}
with the term 
$$
B^{1,*}=-(-\Delta)^{-1}\nabla\times(\rho^{*}u^{*}).
$$
Hence, recalling $B^{*}=\bar{B}$, we have the equations of $(E^*,B^{*})$ as follows
\begin{equation}\label{deE*}
\left\{
\begin{aligned}
&\partial_{t}E^{*} -\frac{1}{\var}\nabla\times B^{*}+ \rho^{*} E^{*} =-\nabla P(\rho^{*})+\nabla\times B^{1,*},\\
&\partial_{t}B^{*} +\frac{1}{\var}\nabla\times E^{*}=0,\\
&\div E^{*} =\bar{\rho}-\rho^{*},\quad\quad \div B^{*} =0.
\end{aligned}
\right.
\end{equation}
Note that there is no decay property for the last term $\nabla\times B^{1,*}$ with respect to $\var$ on the right-hand side of $\eqref{deE*}_2$. In order to handle this term, we introduce the modified error of  the magnetic induction
$$
\delta \mathcal{B}:=\delta B+\var  B^{1,*}.
$$
Then, by \eqref{EMvar12}, \eqref{deE*}, we obtain the equations of $(\delta E,\delta \mathcal{B})$ as follows
\begin{equation}\label{delta2}
\left\{
\begin{aligned}
&\partial_{t}\delta E-\frac{1}{\var}\nabla\times \delta \mathcal{B}+\bar{\rho}\delta E-P'(\bar{\rho})\nabla\div \delta E=\rho^{\var} (z_{L}^{\var}+\widetilde{z}^{\var})-\var\rho^{\var} u^{\var}\times \bar{B}-\delta F ,\\
&\partial_{t}\delta \mathcal{B}+\frac{1}{\var}\nabla\times \delta E=\var \partial_{t}B^{1,*},\\
&\div \delta E=-\delta\rho,\quad\quad \div \delta \mathcal{B}=0,
\end{aligned}
\right.
\end{equation}
where the nonlinear term $\delta F$ is given by \eqref{deltaF}. 

Then, we perform a hypocoercivity argument for the partially dissipative system \eqref{delta2}. From \eqref{delta2} and $\div\nabla\times=0$, we have the localized energy estimate  
\begin{equation}\label{basicdeltaEB}
\begin{aligned}
&\frac{1}{2}\frac{d}{dt}\|(\delta E_{j},\delta \mathcal{B}_{j})\|_{L^2}^2+\bar{\rho} \| \delta E_{j}\|_{L^2}^2+P'(\bar{\rho})\|\div\delta E_j\|_{L^2}^2\\
&~\leq \|\dot{\Delta}_{j} (\rho^{\var} \widetilde{z}^{\var}-\var\rho^{\var} u^{\var}\times \bar{B}-P'(\bar{\rho})\nabla\delta\rho-\delta F)\|_{L^2} \| \delta E_j\|_{L^2}\\
&\quad+\|\dot{\Delta}_{j} (\rho^{\var} z_{L}^{\var})\|_{L^2}\|\delta E_j\|_{L^2}+\var\|\partial_{t}B^{1,*}_{j}\|_{L^2}\| \delta\mathcal{B}_{j}\|_{L^2},
\end{aligned}
\end{equation}
and the cross estimate
\begin{equation}\label{crossdeltaEB}
\begin{aligned}
&-\frac{d}{dt}\int \var \delta E_j \cdot \nabla\times   \delta\mathcal{B}_j\,dx+\|\nabla\times  \mathcal{B}_j\|_{L^2}^2\\
&\quad+\bar{\rho}\var \int  \delta E_{j} \cdot\nabla\times  \delta\mathcal{B}_j\,dx-\|\nabla\times   E_j\|_{L^2}^2\\
&\leq \var\|\dot{\Delta}_{j}(\rho^{\var} \widetilde{z}^{\var}-\var\rho^{\var} u^{\var}\times \bar{B}-\delta F)\|_{L^2}\|\nabla\times\delta\mathcal{B}_{j}\|_{L^2}\\
&\quad+\var \|\dot{\Delta}_{j} (\rho^{\var} z_{L}^{\var})\|_{L^2}\|\nabla\times \delta \mathcal{B}_j\|_{L^2}+\var^2\|\partial_{t}B^{1,*}_{j}\|_{L^2}\|\nabla\times  \delta E_j\|_{L^2}.
\end{aligned}
\end{equation}
For a suitable small $\eta_{*}>0$, we define the functional
$$
\delta \mathcal{L}_{j}(t):=\frac{1}{2}\|(  \delta E_{j},  \delta \mathcal{B}_{j})\|_{L^2}^2+\eta_{*} \min\{1,2^{-2j}\}\int \var  \delta E_j \cdot \nabla\times  \delta\mathcal{B}_j\,dx\sim \|(  \delta E_{j},  \delta \mathcal{B}_{j})\|_{L^2}^2.
$$
Here, $\min\{1,2^{-2j}\}=1$ for $j\leq0$ and $\min\{1,2^{-2j}\}=2^{-2j}$ for $j\geq 1$. It follows from \eqref{basicdeltaEB} and \eqref{crossdeltaEB} that
\begin{equation}\label{426}
\begin{aligned}
&\frac{d}{dt}\delta \mathcal{L}_{j}(t)+\|  \delta E_{j}\|_{L^2}^2+\min\{1,2^{2j}\}\|  \delta \mathcal{B}_j\|_{L^2}^2\\
&\lesssim (\var\|\partial_{t} B^{1,*}_{j}\|_{L^2}+ \|\dot{\Delta}_{j} (\rho^{\var} z_{L}^{\var})\|_{L^2})\sqrt{\delta \mathcal{L}_{j}(t)}\\
&\quad+(\|\dot{\Delta}_{j}(\rho^{\var} \widetilde{z}^{\var})\|_{L^2}+\var\|\dot{\Delta}_{j}(\rho^{\var} u^{\var}\times \bar{B})\|_{L^2}+ \|\delta\rho_{j}\|_{L^2} +\|\delta F_{j}\|_{L^2})\\
&\quad\quad  \times(\|\delta   E_{j}\|_{L^2}+\min\{1,2^{j}\}\|   \delta\mathcal{B}_j\|_{L^2}).
\end{aligned}
\end{equation}
Therefore, applying Lemma \ref{lemmaL1L2} to \eqref{426}, once again implies that 
\begin{equation}\label{427}
\begin{aligned}
&\|( \delta E_{j},  \delta \mathcal{B}_{j})\|_{L^{\infty}_{t}(L^2)}+\|  \delta E_{j}\|_{L^2_{t}(L^2)}+\min\{1,2^{j}\}\|  \delta \mathcal{B}_j\|_{L^2_{t}(L^2)}\\
&\lesssim\|( \delta E_{j},  \delta \mathcal{B}_{j})(0)\|_{L^2}+ \var\|\partial_{t} B^{1,*}_{j}\|_{L^1_{t}(L^2)}+\|\dot{\Delta}_{j} (\rho^{\var} z_{L}^{\var})\|_{L^1_{t}(L^2)}\\
&\quad+\|\dot{\Delta}_{j}(\rho^{\var} \widetilde{z}^{\var})\|_{L^2_{t}(L^2)}+\var\|\dot{\Delta}_{j}(\rho^{\var} u^{\var}\times \bar{B})\|_{L^2_{t}(L^2)}+\|\delta F_{j}\|_{L^2_{t}(L^2)},
\end{aligned}
\end{equation}
which leads to
\begin{equation}\label{428}
\begin{aligned}
&\|(\delta E,\delta \mathcal{B})\|_{\widetilde{L}^{\infty}_{t}(\dot{B}^{\frac{1}{2}} )}+\|\delta E\|_{\widetilde{L}^2_{t}(\dot{B}^{\frac{1}{2}} )}+\|\delta \mathcal{B}\|_{\widetilde{L}^2_{t}(\dot{B}^{\frac{3}{2},\frac{1}{2}})}\\
&\quad\lesssim \|(E^{\var}_{0}-E_{0}^{*},B^{\var}_{0}-\bar{B},\var B^{1,*}(0))\|_{\dot{B}^{\frac{1}{2}} }+\var\|\partial_{t} B^{1,*}_{j}\|_{L^1_{t}(\dot{B}^{\frac{1}{2}} )}+ \|\rho^{\var} z_{L}^{\var}\|_{L^1_{t}(\dot{B}^{\frac{1}{2}})}+\|\rho^{\var}\widetilde{z}^{\var} \|_{\widetilde{L}^2_{t}(\dot{B}^{\frac{1}{2}} )}\\
&\quad\quad+\var\|\rho^{\var} u^{\var}\times \bar{B}\|_{\widetilde{L}^2_{t}(\dot{B}^{\frac{1}{2}} )}+ \|\delta F\|_{\widetilde{L}^2_{t}(\dot{B}^{\frac{1}{2}} )}.
\end{aligned}
\end{equation}
According to \eqref{enhancez}, we can obtain the decay of $\rho^{\var} z_{L}^{\var}$ as follows
\begin{equation}\label{429}
\begin{aligned}
\|\rho^{\var} z_{L}^{\var}\|_{L^1_{t}(\dot{B}^{\frac{1}{2}})}\lesssim (1+\|\rho^{\var}-\bar{\rho}\|_{\widetilde{L}^{\infty}_{t}(\dot{B}^{\frac{3}{2}})}) \|z_{L}^{\var}\|_{L^1_{t}(\dot{B}^{\frac{1}{2}})}\lesssim \alpha_{0}\var.
\end{aligned}
\end{equation}
Substituting \eqref{mmmm1}-\eqref{sfggggggg3} and \eqref{429} into \eqref{428}, we get
\begin{equation}\label{thm43}
\begin{aligned}
&\|(\delta E,\delta \mathcal{B})\|_{\widetilde{L}^{\infty}_{t}(\dot{B}^{\frac{1}{2}} )}+\|\delta E\|_{\widetilde{L}^2_{t}(\dot{B}^{\frac{1}{2}} )}+\|\delta \mathcal{B}\|_{\widetilde{L}^2_{t}(\dot{B}^{\frac{3}{2},\frac{1}{2}})}\\
&\quad\lesssim \|(E^{\var}_{0}-E_{0}^{*},B^{\var}_{0}-\bar{B})\|_{\dot{B}^{\frac{1}{2}} }+(\alpha_{0}+\alpha_{1})(\|\delta \rho\|_{\widetilde{L}^2_{t}(\dot{B}^{\frac{1}{2},\frac{3}{2}} )}+\|\delta E\|_{\widetilde{L}^{2}_{t}(\dot{B}^{\frac{1}{2}} )})\\
&\quad\quad+\var\|\partial_{t} B^{1,*}\|_{L^1_{t}(\dot{B}^{\frac{1}{2}} )}+\var\| B^{1,*}(0)\|_{\dot{B}^{\frac{1}{2}} },
\end{aligned}
\end{equation}
where we have employed  \eqref{sfggggggg}-\eqref{thm41} which have been obtained in Step 1.

\bigbreak

In order to obtain the convergence rate, one needs to establish uniform bounds for $B^{1,*}(0)$ and $\partial_{t}B^{1,*}$ on the right-hand side of \eqref{thm43}. Then, we shall use uniform bounds of  $B^{1,*}$ to recover error estimates for $\delta B=\delta\mathcal{B}-\var B^{1,*}$. Below, we establish some necessary bounds of $B^{1,*}$.
 
\begin{lemma}\label{lemmaB1}
Let $B^{1,*}=-(-\Delta)^{-1}\nabla\times (\rho^{*}u^{*})$. Assume that $\rho^{*}_0$ satisfies \eqref{a2} and $\rho^{*}_0-\bar{\rho}\in \dot{B}^{-\frac{1}{2}}$. Then, $\rho^*$ satisfies
\begin{equation}\label{rho*additional}
\begin{aligned}
&\|\rho^*-\bar{\rho}\|_{\widetilde{L}^{\infty}(\dot{B}^{-\frac{1}{2}})}+\|\rho^*-\bar{\rho}\|_{\widetilde{L}^{2}(\dot{B}^{-\frac{1}{2}})}\lesssim \|\rho_0^*-\bar{\rho}\|_{\dot{B}^{-\frac{1}{2}}\cap\dot{B}^{\frac{3}{2}}}.
\end{aligned}
\end{equation}
Furthermore, it holds that
\begin{equation}\label{B1e}
\left\{
\begin{aligned}
\|B^{1,*}(0)\|_{\dot{B}^{\frac{1}{2}} }&\lesssim \|\rho_{0}^{*}-\bar{\rho}\|_{\dot{B}^{-\frac{1}{2},\frac{3}{2}}}^2,\\
\|B^{1,*}\|_{\widetilde{L}^{\infty}_{t}(\dot{B}^{\frac{1}{2}})\cap \widetilde{L}^2_{t}(\dot{B}^{\frac{1}{2}})}&\lesssim \|\rho_{0}^{*}-\bar{\rho}\|_{\dot{B}^{-\frac{1}{2},\frac{3}{2}}}^2,\\
\|\partial_{t}B^{1,*}\|_{L^1_{t}(\dot{B}^{\frac{1}{2}})}&\lesssim  \|\rho_{0}^{*}-\bar{\rho}\|_{\dot{B}^{-\frac{1}{2},\frac{3}{2}}}^2.
\end{aligned}
\right.
\end{equation}
\end{lemma}

\begin{proof}
We first show \eqref{rho*additional}. Applying Lemma \ref{maximalL2} to \eqref{DD11} with
\[  f_1=f_2=0,\quad f_3=\div((P'(\rho^{*})-P'(\bar{\rho}))\nabla\rho^{*})
+\div ((\rho^{*}-\bar{\rho})\nabla(-\Delta)^{-1}\rho^{*}),  \]
we obtain
\begin{equation}\nonumber
\begin{aligned}
&\|\rho^*-\bar{\rho}\|_{\widetilde{L}^{\infty}(\dot{B}^{-\frac{1}{2}})}+\|\rho^*-\bar{\rho}\|_{\widetilde{L}^{2}(\dot{B}^{-\frac{1}{2}})}\\
&\quad \lesssim \|\rho_0^*-\bar{\rho}\|_{\dot{B}^{-\frac{1}{2}}}+\|(P'(\rho^{*})-P'(\bar{\rho}))\nabla\rho^{*}\|_{\widetilde{L}^{2}(\dot{B}^{\frac{1}{2}})}+\| (\rho^{*}-\bar{\rho})\nabla(-\Delta)^{-1}\rho^{*}\|_{\widetilde{L}^{2}(\dot{B}^{\frac{1}{2}})}.
\end{aligned}
\end{equation}
In accordance with \eqref{uv2}, \eqref{F1} and \eqref{r2}, we obtain
\begin{equation}\nonumber
\begin{aligned}
\|(P'(\rho^{*})-P'(\bar{\rho}))\nabla\rho^{*}\|_{\widetilde{L}^{2}(\dot{B}^{\frac{1}{2}})}\lesssim \|P'(\rho^{*})-P'(\bar{\rho}\|_{\widetilde{L}^{\infty}(\dot{B}^{\frac{3}{2}})}\|\rho^*-\bar{\rho}\|_{\widetilde{L}^{2}(\dot{B}^{\frac{3}{2}})}\lesssim \|\rho_{0}^{*}-\bar{\rho}\|_{\dot{B}^{\frac{1}{2},\frac{3}{2}}}^2.
\end{aligned}
\end{equation}
Similarly, 
\begin{equation}\nonumber
\begin{aligned}
\| (\rho^{*}-\bar{\rho})\nabla(-\Delta)^{-1}\rho^{*}\|_{\widetilde{L}^{2}(\dot{B}^{\frac{1}{2}})}&\lesssim \|\rho^{*}-\bar{\rho}\|_{\widetilde{L}^{2}(\dot{B}^{\frac{1}{2}})}^2\lesssim \|\rho_{0}^{*}-\bar{\rho}\|_{\dot{B}^{\frac{1}{2},\frac{3}{2}}}^2.
\end{aligned}
\end{equation}
Therefore, we have \eqref{rho*additional}.

Next, it follows from $E^*=\nabla(-\Delta)^{-1}(\rho^{*}-\bar{\rho})$ that
\begin{equation}\nonumber
\begin{aligned}
B^{1,*}&=(-\Delta)^{-1}\nabla\times \Big(\nabla P(\rho^{*})+\rho^{*}E^*\Big)=(-\Delta)^{-1}\nabla\times\Big( (\rho^{*}-\bar{\rho})\nabla(-\Delta)^{-1}\rho^{*}).
\end{aligned}
\end{equation}
Hence, for the initial datum $B^{1,*}(0)$ of $B^{1,*}$, employing the product law \eqref{uv2}n we arrive at
\begin{equation}\nonumber
\begin{aligned}
\|B^{1,*}(0)\|_{\dot{B}^{\frac{1}{2}} }&\lesssim \|(\rho_0^{*}-\bar{\rho})\nabla(-\Delta)^{-1}\rho_0^{*}\|_{\dot{B}^{-\frac{1}{2}} }\\
&\lesssim \|\rho_0^{*}-\bar{\rho}\|_{\dot{B}^{-\frac{1}{2}} } \|\nabla(-\Delta)^{-1}\rho_{0}^{*}\|_{\dot{B}^{\frac{3}{2}} }\lesssim  \|\rho_{0}^{*}-\bar{\rho}\|_{\dot{B}^{-\frac{1}{2}}}\|\rho_{0}^{*}-\bar{\rho}\|_{\dot{B}^{\frac{1}{2}}}.
\end{aligned}
\end{equation}
Concerning the estimate of $B^{1,*}$, a similar computation gives
\begin{equation}\nonumber
\begin{aligned}
\|B^{1,*}\|_{\widetilde{L}^{\infty}_{t}(\dot{B}^{\frac{1}{2}})\cap\widetilde{L}^{2}_{t}(\dot{B}^{\frac{1}{2}})}&\lesssim \|\rho^{*}u^{*}\|_{\widetilde{L}^{\infty}_{t}(\dot{B}^{-\frac{1}{2}})\cap\widetilde{L}^{2}_{t}(\dot{B}^{-\frac{1}{2}})}\\
&\lesssim  \|\rho^{*}-\bar{\rho}\|_{\widetilde{L}^{\infty}_{t}(\dot{B}^{-\frac{1}{2},\frac{1}{2}})\cap\widetilde{L}^{2}_{t}(\dot{B}^{-\frac{1}{2},\frac{1}{2}})}^2\lesssim \|\rho_{0}^{*}-\bar{\rho}\|_{\dot{B}^{-\frac{1}{2},\frac{3}{2}}}^2,
\end{aligned}
\end{equation}
where we have used \eqref{r2} and \eqref{rho*additional}. 
Finally, using $\eqref{DD}_{1}$, the estimate of the time derivative $\partial_{t}\rho^*$ follows
\begin{equation}\nonumber
\begin{aligned}
\|\partial_{t}\rho^{*}\|_{\widetilde{L}^{2}_{t}(\dot{B}^{-\frac{1}{2},\frac{1}{2}})}&\lesssim \|\rho^*-\bar{\rho}\|_{\widetilde{L}^2_{t}(\dot{B}^{\frac{1}{2}}\cap\dot{B}^{\frac{5}{2}})}\lesssim \|\rho_0^*-\bar{\rho}\|_{\dot{B}^{-\frac{1}{2},\frac{3}{2}}}.
\end{aligned}
\end{equation}
Hence, we obtain 
\begin{equation}\nonumber
\begin{aligned}
\|\partial_{t}B^{1,*}\|_{L^1_{t}(\dot{B}^{\frac{1}{2}})}&\lesssim \|\partial_{t}\rho^{*} \nabla(-\Delta)^{-1}\rho^{*}\|_{L^1_{t}(\dot{B}^{-\frac{1}{2}})}+\|( \rho^{*}-\bar{\rho}) \nabla(-\Delta)^{-1}\partial_{t}\rho^{*}\|_{L^1_{t}(\dot{B}^{-\frac{1}{2}})}\\
&\lesssim \|\partial_{t}\rho^{*}\|_{\widetilde{L}^2_{t}(\dot{B}^{-\frac{1}{2}})}\|\nabla(-\Delta)^{-1}\rho^{*}\|_{\widetilde{L}^{2}_{t}(\dot{B}^{\frac{3}{2}})}+ \|\rho^{*}-\bar{\rho}\|_{\widetilde{L}^{2}_{t}(\dot{B}^{-\frac{1}{2}})}\|\nabla(-\Delta)^{-1}\partial_{t}\rho^{*}\|_{\widetilde{L}^2_{t}(\dot{B}^{\frac{3}{2}})}\\
&\lesssim \|\rho^{*}-\bar{\rho}\|_{\widetilde{L}^{2}_{t}(\dot{B}^{-\frac{1}{2},\frac{1}{2}})}\|\partial_{t}\rho^{*}\|_{\widetilde{L}^2_{t}(\dot{B}^{-\frac{1}{2},\frac{1}{2}})}\lesssim \|\rho_{0}^{*}-\bar{\rho}\|_{\dot{B}^{-\frac{1}{2},\frac{3}{2}}}^2,
\end{aligned}
\end{equation}
which concludes the proof of Lemma \ref{lemmaB1}.
\end{proof}

It follows from \eqref{thm43}, $\eqref{B1e}_1$ and $\eqref{B1e}_3$ that
\begin{equation}\label{thm430}
\begin{aligned}
&\|(\delta E,\delta \mathcal{B})\|_{\widetilde{L}^{\infty}_{t}(\dot{B}^{\frac{1}{2}} )}+\|\delta E\|_{\widetilde{L}^2_{t}(\dot{B}^{\frac{1}{2}} )}+\|\delta \mathcal{B}\|_{\widetilde{L}^2_{t}(\dot{B}^{\frac{3}{2},\frac{1}{2}})}\\
&\quad\lesssim \|(E^{\var}_{0}-E_{0}^{*},B^{\var}_{0}-\bar{B})\|_{\dot{B}^{\frac{1}{2}} }+(\alpha_{0}+\alpha_{1})(\|\delta \rho\|_{\widetilde{L}^2_{t}(\dot{B}^{\frac{1}{2},\frac{3}{2}} )}+\|\delta E\|_{\widetilde{L}^{2}_{t}(\dot{B}^{\frac{1}{2}} )})+\alpha_{0}\var.
\end{aligned}
\end{equation}
In view of $\eqref{B1e}_2$, we  recover the estimate of $\delta B$  as follows
\begin{equation}\label{thm44}
\begin{aligned}
&\|\delta B\|_{\widetilde{L}^{\infty}_{t}(\dot{B}^{\frac{1}{2}} )}+\|\delta B\|_{\widetilde{L}^2_{t}(\dot{B}^{\frac{3}{2},\frac{1}{2}})}\\
&\quad\lesssim \|\delta \mathcal{B}\|_{\widetilde{L}^{\infty}_{t}(\dot{B}^{\frac{1}{2}} )}+\|\delta \mathcal{B}\|_{\widetilde{L}^2_{t}(\dot{B}^{\frac{3}{2},\frac{1}{2}})}+\var\|B^{1,*}\|_{\widetilde{L}^{\infty}_{t}(\dot{B}^{\frac{1}{2}} )}+\var\|B^{1,*}\|_{\widetilde{L}^{2}_{t}(\dot{B}^{\frac{1}{2}} )}\\
&\quad\lesssim \|\delta \mathcal{B}\|_{\widetilde{L}^{\infty}_{t}(\dot{B}^{\frac{1}{2}} )}+\|\delta \mathcal{B}\|_{\widetilde{L}^2_{t}(\dot{B}^{\frac{3}{2},\frac{1}{2}})}+\var.
\end{aligned}
\end{equation}
Putting \eqref{thm41} and \eqref{thm430}-\eqref{thm44} together and  using  the smallness of $\alpha_{0}$ and $\alpha_{1}$, we have
\begin{equation}
\begin{aligned}\label{err}
&\|\rho^{\var}  -\rho^{*}\|_{\widetilde{L}^{\infty}_{t}(\dot{B}^{\frac{1}{2}} )\cap \widetilde{L}^{2}_{t}( \dot{B}^{\frac{1}{2},\frac{3}{2}} )}+\|u^{\var}  -u^{*}\|_{\widetilde{L}^{2}_{t}(\dot{B}^{\frac{1}{2}} )}\\
&\quad\quad+\|E^{\var} -E^{*}\|_{\widetilde{L}^{\infty}_{t}(\dot{B}^{\frac{1}{2}})\cap \widetilde{L}^{2}_{t}(\dot{B}^{\frac{1}{2}})}+\|B^{\var}-B^{*}\|_{\widetilde{L}^{\infty}_{t}(\dot{B}^{\frac{1}{2}})\cap\widetilde{L}^{2}_{t}(\dot{B}^{\frac{3}{2},\frac{1}{2}})} \\
&\quad\lesssim \|(\rho_0^{\var}-\rho_0^{*},E^{\var}_{0}-E_{0}^{*},B^{\var}_{0}-\bar{B})\|_{\dot{B}^{\frac{1}{2}} }+\var.
\end{aligned}
\end{equation}
Finally, the inequality \eqref{convergence2} follows by \eqref{thm42} and  \eqref{err}. The proof of Theorem \ref{theorem4} is complete.

\appendix

\section{Technical lemmas}\label{app1}

We recall some basic properties of Besov spaces and product estimates that are repeatedly used in the manuscript. We refer to \cite[Chapters 2-3]{HJR} for more details. Remark that all the properties remain true for the Chemin--Lerner type spaces, up to the modification of the regularity exponent according to Hölder's inequality for the time variable.

The first lemma pertains to the so-called Bernstein inequalities.
\begin{lemma}\label{lemma61}
Let $0<r<R$, $1\leq p\leq q\leq \infty$ and $k\in \mathbb{N}$. For any function $u\in L^p$ and $\lambda>0$, it holds
\begin{equation}\nonumber
\left\{
\begin{aligned}
&{\rm{Supp}}~ \mathcal{F}(u) \subset \{\xi\in\mathbb{R}^{d}~: ~|\xi|\leq \lambda R\}\Rightarrow \|D^{k}u\|_{L^q}\lesssim\lambda^{k+d(\frac{1}{p}-\frac{1}{q})}\|u\|_{L^p},\\
&{\rm{Supp}}~ \mathcal{F}(u) \subset \{\xi\in\mathbb{R}^{d}~: ~ \lambda r\leq |\xi|\leq \lambda R\}\Rightarrow \|D^{k}u\|_{L^{p}}\sim\lambda^{k}\|u\|_{L^{p}}.
\end{aligned}
\right.
\end{equation}
\end{lemma}
Next, we state some properties related to homogeneous Besov spaces.
\begin{lemma}\label{lemma62}
Let $d\ge1$ be the dimension. The following properties hold{\rm:}
\begin{itemize}
\item{} For any $s\in\mathbb{R}$ and $q\geq2$, we have the following continuous embeddings{\rm:}
\begin{equation}\nonumber
\begin{aligned}
 \dot{B}^{s}\hookrightarrow  \dot{H}^{s},\quad \quad \dot{B}^{\frac{d}{2}-\frac{d}{q}}\hookrightarrow  L^{q}.
\end{aligned}
\end{equation}
 \item{} $\dot{B}^{\frac{d}{2}}$ is continuously embedded in the set of continuous functions decaying to $0$ at infinity.
\item{}
For any $\sigma\in \mathbb{R}^{d}$, the operator $\Lambda^{\sigma}$ is an isomorphism from $\dot{B}^{s}$ to $\dot{B}^{s-\sigma}$.
\item{} Let $s_{1}\in\mathbb{R}$ and $ s_{2}\leq \frac{d}{2}$. Then
    the space $\dot{B}^{s_{1}}\cap \dot{B}^{s_{2}}$ is a Banach space and satisfies weak compact and Fatou properties: If $u_{k}$ is a uniformly bounded sequence of $\dot{B}^{s_{1}}\cap \dot{B}^{s_{2}}$, then an element $u$ of $\dot{B}^{s_{1}}\cap \dot{B}^{s_{2}}$ and a subsequence $u_{n_{k}}$ exist such that
    \begin{equation}\nonumber
    \begin{aligned}
\lim_{k\rightarrow\infty}u_{n_{k}}=u\quad\text{in}\quad\mathcal{S}'\quad\text{and}\quad\|u\|_{\dot{B}^{s_{1}}\cap \dot{B}^{s_{2}}}\lesssim \liminf_{n_{k}\rightarrow \infty} \|u_{n_{k}}\|_{\dot{B}^{s_{1}}\cap \dot{B}^{s_{2}}}.
    \end{aligned}
    \end{equation}
\end{itemize}
\end{lemma}



The following Morse-type estimates play a fundamental role in the nonlinear analysis. 
\begin{lemma}\label{lemma63}
Let $d\ge1$ be the dimension. The following statements hold:
\begin{itemize}
\item{} Let $s>0$. Then $\dot{B}^{s}\cap L^{\infty}$ is a algebra and
    \begin{equation}\label{uv1}
\begin{aligned}
\|uv\|_{\dot{B}^{s}}\lesssim \|u\|_{L^{\infty}}\|v\|_{\dot{B}^{s} }+ \|v\|_{L^{\infty}}\|u\|_{\dot{B}^{s} }.
\end{aligned}
\end{equation}
\item{}  Let $s_{1}, s_{2}$ satisfy $s_{1}, s_{2}\leq \frac{d}{2}$ and $s_{1}+s_{2}>0$. Then there holds
\begin{equation}\label{uv2}
\begin{aligned}
&\|uv\|_{\dot{B}^{s_{1}+s_{2}-\frac{d}{2}}}\lesssim \|u\|_{\dot{B}^{s_{1}}
}\|v\|_{\dot{B}^{s_{2}}}.
\end{aligned}
\end{equation}
\end{itemize}
\end{lemma}

Next, we present a commutator estimate that is used to control nonlinear terms in medium and high frequencies.

\begin{lemma}\label{lemmacom}
For any $d\ge1$, let $s\in (-\frac{d}{2}-1, \frac{d}{2}+1]$. Then it holds
\begin{align}
&\sum_{j\in\mathbb{Z}}2^{js}\|[u,\dot{\Delta}_{j}]\partial_{x_{i}}v\|_{L^{2}}\lesssim\|\nabla u\|_{\dot{B}^{\frac{3}{2}} }\|v\|_{\dot{B}^{s}},\quad\quad i=1,2,...d.\label{commutator}
\end{align}
\end{lemma}

Also, we recall estimates for the composition of functions.
\begin{lemma}\label{lemma64}
Let $s>0$, and $F:I\rightarrow\mathbb{R}$ with $I$ being an open interval of $\mathbb{R}$. Assume that $F(0)=0$ and that $F'$ is smooth on $I$. Let $u,v\in \dot{B}^{s}\cap L^{\infty}$ have value in $I$. There exists a constant $C=C(F',s,d,I)$ such that
\begin{equation}
\begin{aligned}
\|F(f)\|_{\dot{B}^{s}}\leq C(1+\|f\|_{L^{\infty}})^{[s]+1}\|f\|_{\dot{B}^{s}}.\label{F1}
\end{aligned}
\end{equation}
and
\begin{equation}
\begin{aligned}
& \|F(f_{1})-F(f_{2})\|_{\dot{B}^{s}}\\
&\quad\leq F'(0)\|f_{1}-f_{2}\|_{\dot{B}^{s}}\\
&\quad\quad+C(1+\|(f_{1},f_{2})\|_{L^{\infty}})^{[s]+1}\Big(\|f_{1}-f_{2}\|_{\dot{B}^{s}}\|(f_{1},f_{2})\|_{L^{\infty}}+\|f_{1}-f_{2}\|_{L^{\infty}}\|(f_{1},f_{2})\|_{\dot{B}^{s}}\Big).\label{F3}
\end{aligned}
\end{equation}
\end{lemma}

In order to control the nonlinear term $\Phi(n)$ in \eqref{EM1}, we need the following results concerning the composition of quadratic functions. The proof can be found in \cite{c2}.

\begin{lemma}\label{compositionlp} Let $s>0$, $J$ be a given integer, and $F:I\rightarrow\mathbb{R}$ be smooth with $I$ being an open interval of $\mathbb{R}$. Then there exists a constant $C=C(s,p,r,d,I,F'')$ such that, for $\sigma\geq 0$,
\begin{equation}
\begin{aligned}
&\sum_{j\leq J} 2^{js}\|\dot{\Delta}_{j}(F(f)-F(0)-F'(0)f)\|_{L^2}\\
&\quad\leq C(1+\|f\|_{L^{\infty}})^{[s]+1} \|f\|_{L^{\infty}}\Big{(} \sum_{j\leq J} 2^{js}\|\dot{\Delta}_{j}f\|_{L^2}+2^{J(s-\sigma)} \sum_{j\geq  J-1}2^{j\sigma}\|\dot{\Delta}_{j}f\|_{L^2}\Big{)},\label{q1}
\end{aligned}
\end{equation}
and for any $\sigma\in\mathbb{R}$ that
\begin{equation}
\begin{aligned}
&\sum_{j\geq J-1} 2^{js}\|\dot{\Delta}_{j}(F(f)-F(0)-F'(0)f)\|_{L^2}\\
&\quad\leq C(1+\|f\|_{L^{\infty}})^{[s]+1}\|f\|_{L^{\infty}}\Big{(} 2^{J(s-\sigma)}\sum_{j\leq J} 2^{j\sigma}\|\dot{\Delta}_{j}f\|_{L^2}+\sum_{j\geq  J-1}2^{js}\|\dot{\Delta}_{j}f\|_{L^2}\Big{)}.\label{q2}
\end{aligned}
\end{equation}

\end{lemma}

\begin{lemma}\label{lemmaL1L2}
Let $T>0$ be given time, $E_{1}(t), E_{2}(t)$ and $E_{3}(t)$ be three absolutely continuous nonnegative functions on $[0,T)$. Suppose that there exists a functional $\mathcal{L}(t)\sim E_{1}^2(t)+E_{2}^2(t)+E^2_{3}(t)$ such that
\begin{equation}
\begin{aligned}
&\frac{d}{dt}\mathcal{L}(t)+a_{1}E^2_{1}(t)+a_{2} E^2_{2}(t)+a_{3}E^2_{3}(t)\leq Cg_{1}(t)\sqrt{\mathcal{L}(t)}+Cg_{2}(t)E_{1}(t),\quad\quad t\in (0,T),\label{decayineq1}
\end{aligned}
\end{equation}
where $a_{1}$, $a_{2}, a_{3}$ are strictly positive constants. Then, there exists a constant $C>0$ independent of $T$ and $a_{1}$, $a_{2}, a_{3}$ such that
    if $g_{1}(t)\in L^1(0,T)$ and $g_{2}(t)\in L^2(0,T)$, then we have
    \begin{equation}\label{L2time}
    \begin{aligned}
    &\sup_{t\in[0,T]}( E_{1}(t)+E_{2}(t)+E_{3}(t))\\
    &\quad\quad+\sqrt{a_{1}}\|E_{1}\|_{L^2(0,T)}+\sqrt{a_{2}}\|E_{2}\|_{L^2(0,T)}+\sqrt{a_{3}}\|E_{3}\|_{L^2(0,T)}\\
    &\quad\leq C (E_{1}(0)+E_{2}(0)+E_{3}(0))+C\|g_{1}\|_{L^1(0,T)}+\frac{C}{\sqrt{a_{1}}}\|g_{2}\|_{L^2(0,T)}.
    \end{aligned}
    \end{equation}
\end{lemma}

\begin{proof}
Integrating \eqref{decayineq1} over $[0,T]$ yields
\begin{equation}\nonumber
\begin{aligned}
&\sup_{t\in[0,T]}\mathcal{L}(t)+\int_{0}^{T}\big(a_1 E^2_{1}(t)+a_{2} E^2_{2}(\tau)+a_{3}E^2_{3}(t)\big)\, dt\\
&\quad\leq C\int_{0}^{T}g_{1}(t)\,dt\sup_{t\in[0,T]}\sqrt{\mathcal{L}(t)}+C\Big(\int_{0}^{T}g_{2}^2(t)\,dt\Big)^{\frac{1}{2}}\Big(\int_{0}^{T}E^2_{1}(t)\,dt\Big)^{\frac{1}{2}}\\
&\quad\leq \frac{1}{2}\sup_{t\in[0,T]}\mathcal{L}(t)+\frac{a_1}{2}\int_{0}^{T} E_1^2(t)\,dt+C^2 \Big(  \int_{0}^{T}g_{1}(t)\,dt\Big)^2+\frac{C^2}{a_1} \int_{0}^{T}E^2_{1}(t)\,dt.
\end{aligned}
\end{equation}
Therefore, after taking the square root, we obtain \eqref{L2time}.
\end{proof}

We consider the following Cauchy problem for the damped heat equation in $\mathbb{R}^{d}$:
\begin{equation}
\left\{
\begin{aligned}
&\partial_t u- c_{1} \Delta u+c_{2} u =f,\\
&u(0, x)=u_0(x).
\end{aligned}
\right.\label{Heat}
\end{equation}

\begin{lemma}\label{maximalL2}
Let $s\in\mathbb{R}$, $T>0$ be given time, and $c_{i}$ $(i=1,2)$ be strictly positive constants. Assume $u_{0}\in\dot{B}^{s}$, and $f=f_{1}+f_{2}+f_{3}$ with $f_{i}$ $(i=1,2,3)$ satisfying $f_{1}\in L^1(0,T;\dot{B}^{s})$, $f_{2}\in \widetilde{L}^2(0,T;\dot{B}^{s-1})$ and $f_{3}\in \widetilde{L}^2(0,T;\dot{B}^{s})$. If $u$ is the solution to the Cauchy problem \eqref{Heat}, then $u$ satisfies
\begin{equation}\label{maximal22}
\begin{aligned}
&\|u\|_{\widetilde{L}^{\infty}_{t}(\dot{B}^{s})}+\sqrt{c_{1}}\|u\|_{\widetilde{L}^2_{t}(\dot{B}^{s+1})}+\sqrt{c_{2}}\|u\|_{\widetilde{L}^2_{t}(\dot{B}^{s})}\\
&\quad\leq C(\|u_{0}\|_{\dot{B}^{s}}+\|f_{1}\|_{L^1_{t}(\dot{B}^{s})}+\frac{1}{\sqrt{c_{1}}}\|f_{2}\|_{\widetilde{L}^2_{t}(\dot{B}^{s-1})}+\frac{1}{\sqrt{c_{2}}}\|f_{3}\|_{\widetilde{L}^2_{t}(\dot{B}^{s})}),\quad t\in(0,T),
\end{aligned}
\end{equation}
where $C>0$ is a constant independent of $c_{i}$ $(i=1,2)$ and $T$.
\end{lemma}

\begin{proof}
Taking the $L^2$ inner product of $\eqref{Heat}$ with $u_j$ and using Young's inequality, we obtain
\begin{align}
\frac{d}{dt} \|u_j\|_{L^2}^2+\frac{1}{2}c_{1}2^{2j}\| u_j\|^2+\frac{1}{2}c_{2}\|u_{j}\|_{L^2}^2 \leq \|u_j\|_{L^2} \,\|\dot{\Delta}_{j}f_{1}\|_{L^2}+\frac{2^{-2j}}{c_{1}}\|\dot{\Delta}_{j}f_{2}\|_{L^2}^2+\frac{1}{c_{2}}\|\dot{\Delta}_{j}f_{3}\|_{L^2}^2.\label{881}
\end{align}
Integrating \eqref{881} over $[0,t]$ yields
\begin{equation}\label{882}
\begin{aligned}
&\|u_j\|_{L^{\infty}_{t}(L^2)}^2+\frac{1}{2} c_{1}2^{2j}\int_0^t\|u_j\|_{L^2}^2\,d\tau+\frac{1}{2}c_{2}\int_0^t\|u_j\|_{L^2}^2\,d\tau\\
&\quad\leq \|u_j(0)\|_{L^2}^2+\int_0^t\|\dot{\Delta}_{j}f_{1}\|_{L^2}d\tau \|u_j\|_{L^{\infty}_{t}(L^2)}+\frac{2^{-2j}}{c_{1}}\int_0^t\|\dot{\Delta}_{j}f_{2}\|_{L^2}^2\,d\tau+\frac{1}{c_{2}}\int_0^t\|\dot{\Delta}_{j}f_{3}\|_{L^2}^2\,d\tau.
\end{aligned}
\end{equation}
Employing \eqref{882} and Young's inequality, we arrive at
\begin{equation}\nonumber
\begin{aligned}
&\|u_j\|_{L^{\infty}_{t}(L^2)}+\sqrt{c_{1}}2^{j}\|u_j\|_{L^2_{t}(L^2)}+\sqrt{c_{2}}\|u_j\|_{L^2_{t}(L^2)}\\
&\quad \lesssim  \|u_j(0)\|_{L^2}+\|\dot{\Delta}_{j}f_{1}\|_{L^1_{t}(L^2)}+\frac{2^{-j}}{\sqrt{c_{1}}}\|\dot{\Delta}_{j}f_{2}\|_{L^2_{t}(L^2)}+\frac{1}{\sqrt{c_{2}}}\|\dot{\Delta}_{j}f_{3}\|_{L^2_{t}(L^2)},
\end{aligned}
\end{equation}
which leads to \eqref{maximal22}.
\end{proof}





\bigbreak
\bigbreak
\bigbreak
\bigbreak
\textbf{Acknowledgments}
T. Crin-Barat is supported by the Alexander von Humboldt-Professorship program and the Deutsche
Forschungsgemeinschaft (DFG, German Research Foundation) under project C07 of the
Sonderforschungsbereich/Transregio 154 ‘Mathematical Modelling, Simulation and
Optimization using the Example of Gas Networks’ (project ID: 239904186). L.-Y. Shou is supported by the National Natural Science Foundation of China (12301275) and the China Postdoctoral Science Foundation (2023M741694).
J. Xu is partially supported by the National Natural Science Foundation of China (12271250, 12031006) and the Fundamental Research Funds for the Central Universities, NO. NP2024105. 


\vspace{2mm}

\textbf{Conflict of interest.} The authors do not have any possible conflict of interest.

\vspace{2mm}

\textbf{Data availability statement.}
 Data sharing not applicable to this article as no data sets were generated or analysed during the current study.

\bibliographystyle{abbrv} 

\bibliography{Reference}

\begin{thebibliography}{10}

\bibitem{HJR}
H.~Bahouri, J.-Y. Chemin, and R.~Danchin.
\newblock {\em {F}ourier Analysis and Nonlinear Partial Differential
  Equations}, volume 343 of {\em Grundlehren der Mathematischen
  Wissenschaften}.
\newblock Springer, Heidelberg, 2011.

\bibitem{BZ}
K.~Beauchard and E.~Zuazua.
\newblock Large time asymptotics for partially dissipative hyperbolic systems.
\newblock {\em Arch. Ration. Mech. Anal}, 199:177–227, 2011.

\bibitem{BHN}
S.~Bianchini, B.~Hanouzet, and R.~Natalini.
\newblock Asymptotic behavior of smooth solutions for partially dissipative
  hyperbolic systems with a convex entropy.
\newblock {\em Comm. Pure Appl. Math.}, 60, 1559-1622, 2007.

\bibitem{Cere}
C.~Cercignani.
\newblock {\em The {B}oltzmann Equation and its Applications}.
\newblock Springer-Verlag, New York, 1988.

\bibitem{chenEM}
F.~Chen.
\newblock {\em Introduction to Plasma Physics and Controlled Fusion}, volume~1.
\newblock Plenum Press, New York, 1984.

\bibitem{CJW}
G.-Q. Chen, J.~W. Jerome, and D.~Wang.
\newblock Compressible {E}uler–{M}axwell equations.
\newblock {\em Transp. Theory Stat. Phys.}, 29:311--331, 2000.

\bibitem{chen1994}
G.-Q. Chen, C.~Levermore, and T.-P. Liu.
\newblock Hyperbolic conservation laws with stiff relaxation terms and entropy.
\newblock {\em Comm. Pure Appl. Math.}, 47(6):787--830, 1994.

\bibitem{CoulombelGoudon}
J.-F. Coulombel and T.~Goudon.
\newblock The strong relaxation limit of the multidimensional isothermal
  {E}uler equations.
\newblock {\em Trans. Amer. Math. Soc.}, 359(2):637–648, 2007.

\bibitem{CBD2}
T.~Crin-Barat and R.~Danchin.
\newblock Partially dissipative hyperbolic systems in the critical regularity
  setting : {T}he multi-dimensional case.
\newblock {\em J. Math. Pures Appl. (9)}, 165:1--41, 2022.

\bibitem{CBD1}
T.~Crin-Barat and R.~Danchin.
\newblock Partially dissipative one-dimensional hyperbolic systems in the
  critical regularity setting, and applications.
\newblock {\em Pure Appl. Anal.}, 4(1):85--125, 2022.

\bibitem{CBD3}
T.~Crin-Barat and R.~Danchin.
\newblock Global existence for partially dissipative hyperbolic systems in the
  $\textsc{L}^p$ framework, and relaxation limit.
\newblock {\em Math. Ann.}, 386:2159--2206, 2023.

\bibitem{c2}
T.~Crin-Barat, Q.~He, and L.-Y. Shou.
\newblock The hyperbolic-parabolic chemotaxis system for vasculogenesis: Global
  dynamics and relaxation limit toward a {K}eller-{S}egel model.
\newblock {\em SIAM J. Math. Anal.}, 55 (5):4445--4492, 2023.

\bibitem{Dafermos1}
C.~M. Dafermos.
\newblock {\em Hyperbolic Conservation Laws in Continuum Physics, 3rd edition}.
\newblock Grundlehren Math. Wiss., vol. 325, Springer, Heidelberg, Dordrecht,
  London and New York, 2010.

\bibitem{danchinnoterelaxation}
R.~Danchin.
\newblock Partially dissipative systems in the critical regularity setting, and
  strong relaxation limit.
\newblock {\em EMS Surv. Math. Sci.}, 9, 135-192, 2022.

\bibitem{DAI1}
Y.~Deng, A.~D. Ionescu, and B.~Pausader.
\newblock The {E}uler–{M}axwell system for electrons: Global solutions in 2d.
\newblock {\em Arch. Ration. Mech. Anal.}, 225:771--871, 2017.

\bibitem{Duan2011}
R.~Duan.
\newblock Global smooth flows for the compressible {E}uler-{M}axwell system:
  {T}he relaxation case.
\newblock {\em J. Hyperbolic Differ. Equ.}, 8(2), 375-413, 2011.

\bibitem{DLZ}
R.~Duan, Q.~Liu, and C.~Zhu.
\newblock The {C}auchy problem on the compressible two-fluids {E}uler-{M}axwell
  equations.
\newblock {\em SIAM J. Math. Anal.}, 44:102--133, 2012.

\bibitem{FL}
K.~Friedrichs and P.~Lax.
\newblock Systems of conservation equations with a convex extension.
\newblock {\em Proc. Natl. Acad. Sci. USA}, 68:1686--1688, 1971.

\bibitem{GMEM1}
P.~Germain and N.~Masmoudi.
\newblock Global existence for the {E}uler-{M}axwell system.
\newblock {\em Ann. Scient. \'Ec. Norm. Sup.}, 47:469--503, 2014.

\bibitem{Godunov2}
S.~Godunov.
\newblock An interesting class of quasi-linear systems.
\newblock {\em Dokl. Akad. Nauk SSSR}, 139:521--523, 1961.

\bibitem{GuoIon}
Y.~Guo, A.~D. Ionescu, and B.~Pausader.
\newblock Global solutions of the {E}uler--{M}axwell two-fluid system in 3{D}.
\newblock {\em Ann. Math. (2)}, 183:377--498, 2016.

\bibitem{hajjejpeng12}
M.-L. Hajjej and Y.-J. Peng.
\newblock Initial layers and zero-relaxation limits of {E}uler-{M}axwell
  equations.
\newblock {\em J. Differential Equations}, 252(2):1441–1465, 2012.

\bibitem{HanouzetNatalini}
B.~Hanouzet and R.~Natalini.
\newblock Global existence of smooth solutions for partially dissipative
  hyperbolic systems with convex entropy.
\newblock {\em Arch. Ration. Mech. Anal.}, 169, 89-117, 2003.

\bibitem{IonLie}
A.~D. Ionescu and V.~Lie.
\newblock Long term regularity of the one-fluid {E}uler-{M}axwell system in
  3{D} with vorticity.
\newblock {\em Adv. Math.}, 325:719--769, 2018.

\bibitem{jerome1}
J.~W. Jerome.
\newblock The {C}auchy problem for compressible hydrodynamic-{M}axwell systems:
  {A} local theory for smooth solutions.
\newblock {\em Differ. Integral Equ.}, 16:1345--1368, 2003.

\bibitem{JinXin}
S.~Jin and Z.~Xin.
\newblock The relaxation schemes for systems of conservation laws in arbitrary
  space dimensions.
\newblock {\em Commun. Pure Appl. Math.}, 48, 235-276, 1995.

\bibitem{Junca}
S.~Junca and M.~Rascle.
\newblock Strong relaxation of the isothermal {E}uler system to the heat
  equation.
\newblock {\em Z. Angew. Math. Phys.}, 53, 239–264, 2002.

\bibitem{Katolocal}
T.~Kato.
\newblock The {C}auchy problem for quasi-linear symmetric hyperbolic systems.
\newblock {\em Arch. Ration. Mech. Anal.}, 58:181--205, 1975.

\bibitem{KY}
S.~Kawashima and W.-A. Yong.
\newblock Dissipative structure and entropy for hyperbolic systems of balance
  laws.
\newblock {\em Arch. Ration. Mech. Anal.}, 174, 345–364, 2004.

\bibitem{Lax1}
P.~Lax.
\newblock Development of singularities of solutions of nonlinear hyperbolic
  partial differential equations.
\newblock {\em J. Math. Phys.}, 5:611--613, 1964.

\bibitem{Lemarié23}
V.~Lemarié.
\newblock Parabolic-elliptic {K}eller-{S}egel's system.
\newblock {\em arXiv:2307.05981}.

\bibitem{LiPengZhao1d}
Y.~Li, Y.-J. Peng, and L.~Zhao.
\newblock Convergence rate from hyperbolic systems of balance laws to parabolic
  systems.
\newblock {\em Appl. Anal.}, 100(5):1079–1095, 2021.

\bibitem{liEM}
Y.~Li, Y.-J. Peng, and L.~Zhao.
\newblock Convergence rates in zero-relaxation limits for {E}uler-{M}axwell and
  {E}uler-{P}oisson systems.
\newblock {\em J. Math. Pures Appl. (9)}, 154, 185-211, 2021.

\bibitem{CoulombelLin}
C.~Lin and J.-F. Coulombel.
\newblock The strong relaxation limit of the multidimensional {E}uler
  equations.
\newblock {\em Nonlinear Differ. Equ. Appl.}, 20:447--461, 2013.

\bibitem{liuguopeng19}
C.~Liu, Z.~Guo, and Y.-J. Peng.
\newblock Global stability of large steady-states for an isentropic
  {E}uler–{M}axwell system in $\mathbb{R}^3$.
\newblock {\em Commun. Math. Sci.}, 17(7):1841--1860, 2019.

\bibitem{Liu}
T.~Liu.
\newblock Hyperbolic conservation laws with relaxation.
\newblock {\em Comm. Math. Phys.}, 60:153--175, 1987.

\bibitem{Majdalocal}
A.~Majda.
\newblock {\em Compressible Fluid Flow and Systems of Conservation Laws in
  Several Space Variables}.
\newblock Appl. Math. Sci., vol. 53, Springer-Verlag, Berlin, New York, 1984.

\bibitem{mar0}
P.~Marcati and A.~Milani.
\newblock The one-dimensional {D}arcy’s law as the limit of a compressible
  {E}uler flow.
\newblock {\em J. Differential Equations}, 84, 129-147, 1990.

\bibitem{MMS}
P.~Marcati, A.~Milani, and P.~Secchi.
\newblock Singular convergence of weak solutions for a quasilinear
  nonhomogeneous hyperbolic system.
\newblock {\em Manuscripta Math.}, 60:49--69, 1988.

\bibitem{marcatiparabolicrelaxation}
P.~Marcati and B.~Rubino.
\newblock Hyperbolic to parabolic relaxation theory for quasilinear first order
  systems.
\newblock {\em J. Differential Equations}, 162:359–399, 2000.

\bibitem{Mori}
N.~Mori.
\newblock An {S}\&{K} mixed condition for symmetric hyperbolic systems with
  non-symmetric relaxations.
\newblock {\em Eur. J. Math}, 6:590--618, 2020.

\bibitem{pengzhu14}
H.~Peng and C.~Zhu.
\newblock Asymptotic stability of stationary solutions to the compressible
  {E}uler-{M}axwell equations.
\newblock {\em Indiana Univ. Math. J}, 62(4):1203--1235, 2013.

\bibitem{Peng12}
Y.-J. Peng.
\newblock Global existence and long-time behavior of smooth solutions of
  two-fluid {E}uler-{M}axwell equations.
\newblock {\em Ann. Inst. H. Poincar\'e Anal. Non Lin\'eaire}, 29:737--759,
  2012.

\bibitem{peng2015}
Y.-J. Peng.
\newblock Stability of non-constant equilibrium solutions for
  {E}uler–{M}axwell equations.
\newblock {\em J. Math. Pures Appl. (9)}, 103(1), 39–67, 2015.

\bibitem{PYJEMSIAM}
Y.-J. Peng, S.~Wang, and Q.~Gu.
\newblock Relaxation limit and global existence of smooth solutions of
  compressible {E}uler–{M}axwell equations.
\newblock {\em SIAM J. Math. Anal}, 43(2), 944-970, 2011.

\bibitem{Peng16AIHP}
Y.-J. Peng and V.~Wasiolek.
\newblock Parabolic limit with differential constraints of first-order
  quasilinear hyperbolic systems.
\newblock {\em Ann. Inst. H. Poincar\'e Anal. Non Lin\'eaire}, 33:1103--1130,
  2016.

\bibitem{PengWasiolek}
Y.-J. Peng and V.~Wasiolek.
\newblock Uniform global existence and parabolic limit for partially
  dissipative hyperbolic systems.
\newblock {\em J. Differential Equations}, volume 260, Issue 9:7059--7092,
  2016.

\bibitem{RG1}
H.~Rishbeth and O.~K. Garriott.
\newblock {\em Introduction to Ionospheric Physics}.
\newblock Academic Press, 1969.

\bibitem{SK}
S.~Shizuta and S.~Kawashima.
\newblock Systems of equations of hyperbolic-parabolic type with applications
  to the discrete {B}oltzmann equation.
\newblock {\em Hokkaido Math. J.}, 14, 249-275, 1985.

\bibitem{UedaWangKawa2012}
Y.~Ueda, , S.~Wang, and S.~Kawashima.
\newblock Dissipative structure of the regularity-loss type and time asymptotic
  decay of solutions for the {E}uler–{M}axwell system.
\newblock {\em SIAM J. Math. Anal.}, 44:2002--2017, 2012.

\bibitem{UDK1}
Y.~Ueda, R.~Duan, and S.~Kawashima.
\newblock Decay structure for symmetric hyperbolic systems with non-symmetric
  relaxation and its application.
\newblock {\em Arch. Ration. Mech. Anal.}, 205:239--266, 2012.

\bibitem{UedaKawa2011}
Y.~Ueda and S.~Kawashima.
\newblock Decay property of regularity-loss type for the {E}uler-{M}axwell
  system.
\newblock {\em Methods Appl. Anal.}, 18:245--268, 2011.

\bibitem{Vicenti}
W.~Vicenti and C.~Kruger.
\newblock {\em Introduction to physical gas dynamics}.
\newblock Robert E. Krieger, Melbourne, 1982.

\bibitem{Villani}
C.~Villani.
\newblock {\em Hypocoercivity}.
\newblock Mem. Am. Math. Soc., 2010.

\bibitem{Whitham}
G.~B. Whitham.
\newblock {\em Linear and Nonlinear Waves}.
\newblock Wiley, New York, 1974.

\bibitem{XuEM}
J.~Xu.
\newblock Global classical solutions to the compressible {E}uler–{M}axwell
  equations.
\newblock {\em SIAM J. Math. Anal.}, 43(6), 2688–2718, 2011.

\bibitem{XK1}
J.~Xu and S.~Kawashima.
\newblock Global classical solutions for partially dissipative hyperbolic
  system of balance laws.
\newblock {\em Arch. Ration. Mech. Anal}, 211, 513–553, 2014.

\bibitem{XK2}
J.~Xu and S.~Kawashima.
\newblock The optimal decay estimates on the framework of {B}esov spaces for
  generally dissipative systems.
\newblock {\em Arch. Ration. Mech. Anal}, 218, 275–315, 2015.

\bibitem{XK1D}
J.~Xu and S.~Kawashima.
\newblock Frequency-localization {D}uhamel principle and its application to the
  optimal decay of dissipative systems in low dimensions.
\newblock {\em J. Differential Equations}, 261, 2670-2701, 2016.

\bibitem{xumorikaw1}
J.~Xu, N.~Mori, and S.~Kawashima.
\newblock ${L}^p$–${L}^q$–${L}^r$ estimates and minimal decay regularity
  for compressible {E}uler–{M}axwell equations.
\newblock {\em J. Math. Pures Appl. (9)}, 104(2), 965-981, 2015.

\bibitem{XuWang}
J.~Xu and Z.~Wang.
\newblock Relaxation limit in {B}esov spaces for compressible {E}uler
  equations.
\newblock {\em J. Math. Pures Appl. (9)}, 99:43--61, 2013.

\bibitem{XuEMtwo}
J.~Xu, J.~Xiong, and S.~Kawashima.
\newblock Global well-posedness in critical {B}esov spaces for two-fluid
  {E}uler–{M}axwell equations.
\newblock {\em SIAM J. Math. Anal.}, 45(3), 1422–1447, 2013.

\bibitem{Yong}
W.-A. Yong.
\newblock Entropy and global existence for hyperbolic balance laws.
\newblock {\em Arch. Ration. Mech. Anal}, 172, 47–266, 2004.

\end{thebibliography}
\vfill 
\end{document}